\documentclass{amsart}
\usepackage{appendix}
\usepackage{amssymb}

\usepackage{thmtools}
\usepackage{thm-restate}
\usepackage{hyperref}
\usepackage{cleveref}

\usepackage{mathrsfs}

\usepackage{graphicx}

\usepackage[cmtip,all]{xy}

\usepackage{amsthm}
\newtheorem{theorem}{Theorem}[section]

\newtheorem{lem}[theorem]{Lemma}
\newtheorem{lemma}[theorem]{Lemma}
\newtheorem{cor}[theorem]{Corollary}
\newtheorem{prop}[theorem]{Proposition}
\newtheorem*{conjecture*}{Conjecture}

\theoremstyle{definition}
\newtheorem{defn}[theorem]{Definition} 
\newtheorem{exmp}[theorem]{Example}

\newtheorem{remark}[theorem]{Remark}
\newtheorem*{question*}{Question}

\newtheorem{construction}[theorem]{Construction}

\newcommand{\Id}{\mathrm{Id}}
\newcommand{\lf}{\mathrm{lf}}
\newcommand{\Aut}{\mathrm{Aut}}
\newcommand{\Mat}{\mathrm{Mat}}
\newcommand{\End}{\mathrm{End}}
\newcommand{\PSL}{\mathrm{PSL}}

\newcommand{\PGL}{\mathrm{PGL}}

\newcommand{\id}{\mathrm{id}}
\newcommand{\al}{\alpha}

\newcommand{\CQ}{\mathcal Q}
\newcommand{\CalC}{\mathcal C}

\newcommand{\SA}{\mathcal A}
\newcommand{\NN}{\mathbb N}
\newcommand{\RR}{\mathbb R}
\newcommand{\ZZ}{\mathbb Z}
\newcommand{\CC}{\mathbb C}
\newcommand{\rank}{\operatorname{rank}}

\newcommand{\Qbb}{{\mathbb Q}}
\newcommand{\Rbb}{{\mathbb R}}
\newcommand{\Zbb}{{\mathbb Z}}
\newcommand{\Nbb}{{\mathbb N}}
\newcommand{\Cbb}{{\mathbb C}}

\makeatletter
\def\l@subsection{\@tocline{2}{0pt}{2.5pc}{5pc}{}}
\makeatother

\usepackage[style=alphabetic,
]{biblatex}

\usepackage{booktabs}  
\usepackage{adjustbox}
\usepackage{tikz}
\usepackage{quiver}
\usepackage{hyperref}
\usepackage{mathtools}
\usepackage{enumitem}
\usepackage{soul}

\numberwithin{equation}{section}

\begin{document}

\title{{Quantum Cellular Automata:}\\  The Group, the Space, and the Spectrum }


%    author one information
\author{Mattie Ji}
\address{University of Pennsylvania}
\curraddr{}
\email{mji13@sas.upenn.edu}
\thanks{}

%    author two information
\author{Bowen Yang}
\address{Harvard University}
\curraddr{}
\email{bowen\_yang@g.harvard.edu}
\thanks{}

%    \subjclass is required.

\date{}

\dedicatory{}

\begin{abstract}
Over an arbitrary commutative ring $R$, we develop a theory of quantum cellular automata. We then use algebraic K-theory to construct a space $\mathbf{Q}(X)$ of quantum cellular automata (QCA) on a given metric space $X$. In most cases of interest, $\pi_0 \mathbf{Q}(X)$ classifies QCA up to quantum circuits and stabilization. Notably, the QCA spaces are related by homotopy equivalences $\mathbf{Q}(*) \simeq \Omega^n \mathbf{Q}(\Zbb^n)$ for all $n$, which shows that the classification of QCA on Euclidean lattices is given by an $\Omega$-spectrum indexed by the dimension $n$. As a corollary, we also obtain a non-connective delooping of the K-theory of Azumaya $R$-algebras, which may be of independent interest. We also include a section leading to the $\Omega$-spectrum for QCA over $C^*$-algebras with unitary circuits.
\end{abstract}

\maketitle

\renewcommand\contentsname{Table of Contents}
\tableofcontents

\section{Introduction}

\subsection{Quantum Cellular Automata}\label{subsec::intro}A quantum spin system~\cite{ Nachtergaele2004QuantumSpinSystems, naaijkens2017quantum} is a mathematical model in quantum many-body physics in which degrees of freedom are arranged on the points of a lattice $X$ and interact locally. Such systems are typically described in the language of operator algebras~\cite{bratteli2012operator1,bratteli2012operator2}. One associates to the lattice $X$ a $C^*$-algebra of observables, which for the purposes of this introduction we denote by $\SA(X)$. A state $\omega$ on $\SA(X)$ is  a positive normalized linear functional\footnote{That is, $\omega : \SA(X) \to \CC$ is a linear functional satisfying $\omega(I)=1$ and $\omega(a^*a)\ge 0$ for all $a\in \SA(X)$, where $I$ is the identity in $\SA(X)$.} encoding expectation values. A state is called \emph{gapped}~\cite{HastingsKoma2006,NachtergaeleSims2006} if it satisfies a certain energy stability condition.

A gapped \textit{phase} of matter~\cite{beaudry2024homotopical, kapustin2022local, artymowicz2024mathematical} in spatial dimension $d$ is an equivalence class of gapped states on the lattice $X=\ZZ^d$. Two gapped states are said to lie in the same phase if they can be connected by a continuous path of gapped states. This definition is motivated by the fact that states in the same phase share the same stable low-energy behavior.

Two gapped states can be stacked into a single gapped state by taking the tensor product of their operator algebras and multiplying the linear functionals. This induces a structure of commutative monoids on gapped phases. The identity element is called the \textit{trivial phase}. A state is called \textit{invertible} if there exists another state whose stack with it lies in the trivial phase. For each $d$, the set of invertible gapped phases, or simply invertible phases, form an abelian group $\mathcal I(\ZZ^d)$.

Over the past two decades, our understanding of invertible phases has been profoundly reshaped. A central conjecture, due to Alexei Kitaev~\cite{Kitaev2011,,Kitaev2013, Kitaev2015,Kitaev2019, Kitaev2023}, points to a surprising connection to stable homotopy theory.
\begin{conjecture*}[Kitaev's conjecture]
For each spatial dimension $d$, there is a space-level construction 
$\mathrm{IP}(\ZZ^d)$ of invertible states with a canonical isomorphism
\[
\pi_0(\mathrm{IP}(\ZZ^d))\cong \mathcal I(\ZZ^d).
\]
Moreover, these spaces assemble naturally into an $\Omega$-spectrum
$\mathbb{IP}=\{\mathrm{IP}(\ZZ^d)\}_{d\ge 0}$.
Equivalently, they come equipped with homotopy equivalences
\[
\mathrm{IP}(\ZZ^d) \simeq \Omega( \mathrm{IP}(\ZZ^{d+1})).
\]
\end{conjecture*}

Substantial evidence has accumulated and significant progress has been made for Kitaev's conjecture, from topological field theory \cite{kapustin2014bosonic, freed2021reflection}, functional analysis \cite{ogata2021classification,ogata2021invariant,kapustin2020hall,kapustin2021classification,sopenko2021index}, symmetry-protected topological phases \cite{chen2011classification,chen2012symmetry,chen2013symmetry}, and matrix product states \cite{qi2025charting,beaudry2025classifying}. Recently, \cite{kubota2025stable} provides a construction of $\mathbb{IP}$. The above list is far from complete; see the introduction of \cite{kubota2025stable} for a more comprehensive bibliography and further discussion of the conjecture. 

A closely related topic is the study of \textit{quantum cellular automata} (QCA), which may be viewed as dynamical counterparts of invertible phases. A QCA in $d$ spatial dimensions is a strictly locality-preserving automorphism of the algebra of observables $\SA(\ZZ^d)$ governing a $d$-dimensional quantum spin system. The set of QCA $\CQ^*(\ZZ^d)$ form a group under composition. It contains a normal subgroup $\CalC^*(\ZZ^d)$ called the group of quantum circuits. \cite{freedman2020classification} showed that $\CQ^*(\ZZ^d)/ \CalC^*(\ZZ^d)$ is an abelian group.

A QCA $\al$ acts on the collection of invertible states taking one invertible state to another. Specifically, given an invertible state $\omega : \SA(\ZZ^d)\longrightarrow \CC,$ the composition $$\omega \circ \al: \SA(\ZZ^d)\longrightarrow \CC$$ is another invertible state. Furthermore, automorphisms in $ \CalC^*(\ZZ^d)$ do not alter the phase of an invertible state. This defines a group homomorphism \begin{equation}
    \CQ^*(\ZZ^d)/ \CalC^*(\ZZ^d)\longrightarrow \mathcal I(\ZZ^d)
\end{equation} sending a QCA to its image of a system in the trivial phase. Thus, understanding the structure of QCA has become an important way to probe invertible phases, and vice versa~\cite{haah2023nontrivial, fidkowski2024qca, sun2025clifford}. The expectation that invertible phases assemble into an $\Omega$-spectrum suggests that QCA themselves should assemble to a similar $\Omega$-spectrum \cite{long2025topological, tu2025anomalies}.

\begin{conjecture*}[The QCA conjecture]
For each spatial dimension $d$, there is a space-level construction $\mathrm{QCA}(\ZZ^d)$ of the QCA classification group with a canonical isomorphism
\[
\pi_0(\mathrm{QCA}(\ZZ^d))\cong \CQ^*(\ZZ^d)/ \CalC^*(\ZZ^d).
\]
Moreover, these spaces assemble naturally into an $\Omega$-spectrum
$\mathbb{QCA}=\{\mathrm{QCA}(\ZZ^d)\}_{d\ge 0}$.
\end{conjecture*}

In this paper, we prove the QCA conjecture using algebraic $K$-theory. In fact, our results go significantly beyond the original formulation in two directions.
\begin{enumerate}
    \item We develop an algebraic theory of QCA over an arbitrary commutative base ring $R$. For each such ring we construct an associated $\Omega$-spectrum
\[
\{\mathbf{Q}(\ZZ^d)\}_{d\ge 0}.
\]
In fact, we show that $\mathbf{Q}(\mathbb{Z}) = K(\operatorname{Az}({R}))$ is the $K$-theory of Azumaya $R$-algebras, so the sequence provides a \textit{non-connective delooping} of $K(\operatorname{Az}(R))$.

\item We prove a similar delooping theorem for QCA defined on a class of more general metric spaces $X$ that depends only on the large-scale (coarse) geometry of $X$. More precisely, we construct an $\Omega$-spectrum
\[
\{\mathbf{Q}(X \times \ZZ^d)\}_{d\ge 0},
\] 
\end{enumerate}
The spectrum $\mathbb{QCA}$ appearing in the QCA conjecture is obtained from a modification of the special case when $R=\CC$ and $X$ is a point (see Section~\ref{sec::star}).

\subsection{Algebraic K-Theory}
The strategy we use to construct the QCA spaces $\mathbf{Q}(\mathbb{Z}^n)$ is via \textit{algebraic K-theory}. Algebraic K-theory, on a high level, is the study of how to break part and assemble objects linearly, which makes the field amenable to classification questions.

Given a ring $R$, its $K_0$-group $K_0(R)$ is defined as the algebraic group completion of the additive monoid of finitely generated projective modules over $R$. Quillen \cite{quillen_plus,quillen_1, quillen_2} defined the higher K-groups as homotopy groups $\pi_i(K(R))$ of a space $K(R)$, with $K_0(R) = \pi_0(K(R))$. The construction of $K(R)$ is a special case of a general construction that takes a symmetric monoidal category $\mathcal{C}$ and produces its K-theory space $K(\mathcal{C})$ as the topological group completion of the classifying space $B\mathcal{C}$. The group completion of a space can have very dramatic effects on its homotopy groups. For example, if $\mathcal{C}$ is the groupoid of finite sets with the symmetric monoidal operation as disjoint union, the Barratt–Priddy–Quillen Theorem states that $\pi_i K(\mathcal{C})$ are the stable homotopy groups of spheres.

Instead of going up, Bass \cite{bass1968algebraic} defined the lower K-groups (also known as negative K-groups) $K_{-i}(R)$ of a ring $R$. Pedersen showed \cite{PEDERSEN1984461} that $K_{-i}(R)$ can be recovered as $\pi_1 K(C_{i+1}(R))$, where $C_{i+1}(R)$ is a symmetric monoidal category associated to $R$. To give an informal description of this category $C_{i+1}(R)$ (see \cite{PEDERSEN1984461, pedersen_weibel}), an object in $C_{i+1}(R)$ is created by placing a finitely generated free $R$-module on each point of the lattice $\Zbb^{i+1} \subset \Rbb^{i+1}$. A morphism between two objects is an $R$-linear isomorphism that is ``uniformly bounded" with respect to the metric on $\Zbb^{i+1}$, i.e., there is a uniformly bounded propagation. The category $C_{i+1}(R)$ is symmetric monoidal under pointwise direct sum. Pedersen and Weibel \cite{pedersen_weibel} then showed that these negative K-groups are the homotopy groups of a non-connective delooping of $K(R)$.

As we will see more formally in Section~\ref{sec::group}, the precise set-up of QCA is  a multiplicative analog of the additive set-up for negative K-theory. Indeed, for the lattice $\Zbb^{i+1} \subset \Rbb^{i+1}$, a quantum spin system is a placement of a matrix algebra over $R$ on each point of the lattice. A morphism between two objects is a $R$-algebra isomorphism that is ``uniformly bounded" (this is called \textit{locality-preserving}) with respect to the metric on $\Zbb^{i+1}$, and a quantum cellular automaton is an automorphism of such quantum spin system. Our perspective is that we can form a category $\textbf{C}(\Zbb^{i+1})$ (see Section~\ref{sec::space}) equipped with a symmetric monoidal structure of pointwise tensor product. The associated loop space $\Omega K(\textbf{C}(\Zbb^{i+1}))$ of the K-theory space produced by Segal's machinery is what we refer to as $\textbf{Q}(\Zbb^{i+1})$, the \textit{space of QCA}. In fact, our setting applies to a general metric space with reasonable assumptions.

In the main text, we develop our construction on a general class of metric spaces $X$. We show in Section~\ref{sec::spectra} that $\textbf{Q}(X)$ can be delooped into an $\Omega$-spectrum. Given the analogy to negative K-theory, the groups $\pi_{0} \textbf{Q}(\Zbb^n) = \mathcal{Q}(\Zbb^n)/\mathcal{C}(\Zbb^n)$ for $n > 1$ can be viewed as ``negative homotopy groups" of the $K$-theory of Azumaya algebras. In fact, we also show that these groups can be thought of as the \textit{higher Brauer groups} over a field (see Example~\ref{example::Brauer_ring} and \ref{example::higher_Brauer}).

\subsection{Related Work and a Conjecture}

Early studies of QCA led to a complete determination of the QCA group in one dimension~\cite{gross2012index}. This was followed by a period of rapid progress, of which we mention only a few.

The group-theoretic foundations of QCA were developed in~\cite{freedman2020classification,freedman2022group}. Highly nontrivial higher-dimensional QCA were constructed in~\cite{haah2023nontrivial,chen2023exactly,shirley2022three}. A version of the QCA conjecture has been established for the special class of Clifford QCA. The resulting $\Omega$-spectrum is identified with algebraic $L$-theory~\cite{haah2025topological,yang2025categorifying}. Connections with topological field theory have begun to emerge~\cite{fidkowski2024qca,sun2025clifford}, and a bulk–boundary correspondence relating QCA to commuting Hamiltonian models was established in~\cite{ruba2025witt,haah2023invertible,haah2021clifford}.

The classification of lattice symmetry anomalies~\cite{tu2025anomalies, kapustinXu2025higher, kapustin2025higher, kawagoe2025anomaly, feng2025onsiteability, czajka2025anomalies} provides additional evidence for the QCA conjecture. More recently, several proposals for the homotopy type of QCA have appeared in the physics literature, including formulations in terms of crossed $n$-cubes~\cite{kapustin2025higher,kapustinXu2025higher} and in the language of $\infty$-categories~\cite{czajka2025anomalies,tu2025anomalies}.

Regarding the computation of QCA groups, Haah~\cite{haah2021clifford} raised the question of whether
\[
\CQ^*(\ZZ^3)/\CalC^*(\ZZ^3)=\pi_0\left(\mathrm{QCA(\ZZ^3)}\right)
\]
agrees with, or is closely related to, the Witt group of braided fusion $\CC$-linear categories~\cite{davydov2010witt,davydov2013structure}. This suggests a connection between QCA and tensor categories. Motivated by Haah's question and our generalized theory, we propose a conjecture for future work. 

\begin{conjecture*}
Let $R=k$ be a field. There exists a natural group homomorphism
\[
\pi_0 \mathbf{Q}(\ZZ^3) \longrightarrow \mathrm{Witt}_k,
\]
where $\mathrm{Witt}_k$ denotes the Witt group of braided fusion $k$-linear categories.
\end{conjecture*} 

\subsection{Main Results and Outline} The rest of the paper is organized as follows. In Section~\ref{sec::group}, we introduce the total QCA group $\CQ(X)$ over a ring $R$ and its subgroup of quantum circuits $\mathcal{C}(X)$. The quotient $\CQ(X)/\CalC(X)$ is called the \textit{QCA classification group}. Let $X=\ZZ$ with the usual metric; our first theorem connects the classification of QCA to that of Azumaya algebras. 
\begin{restatable}{introthm}{KAz}\label{thm::k0_qz}
        There is an explicit surjective homomorphism 
        \begin{equation}
            b: \CQ(\ZZ)\longrightarrow K_0(\mathrm{Az(R)})\text{ with $\ker b = \mathcal{C}(\ZZ)$. }
        \end{equation}
\end{restatable}

In Section~\ref{sec::space}, we define a symmetric monoidal category $\mathbf{C}(X)$ of quantum spin systems and use Segal's $K$-theory to obtain the QCA space $\mathbf{Q}(X):=\Omega(K(\mathbf{C}(X)))$. We show the homotopy type of $K(\mathbf{C}(X))$ admits a description in terms of Quillen's plus-construction \cite{quillen_plus} and the total QCA group. 

\begin{restatable}{introthm}{plus}\label{thm::plus_construction}
   The commutator subgroup $[\mathcal{Q}(X), \mathcal{Q}(X)]$ of $\mathcal{Q}(X)$ is perfect. $B\mathcal{Q}(X)^+$, the plus-construction at the commutator subgroup, is a weakly simple space. Furthermore, the space $K(\mathbf{C}(X))$ is equivalent to $K_0(\mathbf{C}(X)) \times B\mathcal{Q}(X)^+$.
\end{restatable}
Using Theorem~\ref{thm::k0_qz}, some results in Section~\ref{subsec::abelianization}, and (optionally) Theorem~\ref{thm::plus_construction}, we show $\mathbf{Q}_i(X):=\pi_i\mathbf{Q}(X)$ and the QCA classification group $\CQ(X)/\CalC(X)$ agree in many cases of interest.%, thereby justifying the name ``QCA space''.
\begin{restatable}{introcor}{Qzero}\label{cor::qca_space_properties} Let $R$ be a commutative ring. 
\begin{enumerate}
    \item For $n > 0$, $\mathbf{Q}_0(\Zbb^n) = K_1(\mathbf{C}(\Zbb^n)) = \mathcal{Q}(\Zbb^n)^{ab} = \mathcal{Q}(\Zbb^n)/\mathcal{C}(\Zbb^n)$.
    \item $\mathbf{Q}_0(\Zbb) = K_0(\operatorname{Az}(R))$.
    \item If $R$ is a field satisfying $(R^\times)^n=R^\times$ for all $n$, $\mathbf{Q}_0(X) = \CQ(X)/\CalC(X).$
    \item For $n > 0$, $\mathbf{Q}_1(\mathbb{Z}^n) = H_2(\mathcal{C}(\Zbb^n); \Zbb)$.%\mathbf{Q}_0(\mathbb{Z}^{n-1})\simeq\
\end{enumerate}
\end{restatable}

For a suitably well-behaved commutative ring $R$, we also compute $K_0(\mathbf{C}(X))$ and show that it is related to the $0$-th ordinary coarse homology group of $X$.

\begin{restatable}{introthm}{coarse}\label{thm::pi0_is_coarse}
Let $R$ be a ring with no non-trivial idempotent elements. The group $K_0(\mathbf{C}(X))$ is isomorphic to $CH_0(X, \Zbb^{\oplus \omega})$, the $0$-th coarse homology group of $X$ with coefficients in $\Zbb^{\oplus \omega}$. Here, $\Zbb^{\oplus \omega}$ is the countably infinite direct sum of $\Zbb$.
\end{restatable}

In fact, following the proof of Theorem~\ref{thm::pi0_is_coarse}, one can show that $K(\textbf{C}(\Zbb^n))$ is connected for $n > 0$ over any commutative ring $R$ (see Remark~\ref{rmk:coarse_connected}). We then verify in Section~\ref{sec::spectra} that the QCA space we construct on Euclidean lattices $\ZZ^n$ satisfies the desired delooping relation.

\begin{restatable}{introthm}{delooping}\label{thm::higher_delooping}
    There is a homotopy equivalence \begin{equation}\label{eq::higher_delooping_relations}
        \mathbf{Q}(\Zbb^{n-1}) \simeq \Omega \mathbf{Q}(\Zbb^{n}), \quad \text{  for } n > 0.
    \end{equation} 
\end{restatable}
The proof of Theorem~\ref{thm::higher_delooping} also adapts to similar delooping for $\mathbf{Q}(X) \simeq \Omega \mathbf{Q}(X \times \Zbb)$ (see Remark~\ref{rmk::more_generally}). 

Section~\ref{sec::star} applies our strategy to construct an $\Omega$-spectrum for QCA over $C^*$-algebras with $R=\CC$. In this setting, automorphisms are additionally required to be $C^*$-algebra maps. \textit{This resolves the original formulation of the QCA conjecture.}

Finally, in Section~\ref{sec::compute}, we compute the homotopy groups of $\mathbf{Q}(*)$ and $\mathbf{Q}(\mathbb{Z})$ over a field. As we showed that $\mathbf{Q}(\ZZ) = K(\operatorname{Az}(R))$, we note the latter has been computed in \cite{c80f3555-900f-35b5-ab05-6d397e6a44e6}. Here, we provide an expository account of this over a field. Appendix~\ref{sec::coarse} discusses some background on coarse homology groups, and Appendix~\ref{sec::group_complete} proves a technical lemma on group completion used in the main paper.\\
 
\noindent \textbf{Conventions and Notations.} Throughout the paper, the Cartesian product of metric spaces is equipped with the $L^\infty$-metric. All commutative rings $R$ are assumed to be unital. Any subalgebra is understood to be unital, i.e., to share the same unit as the ambient algebra. We write $\operatorname{Mat}(R^n)$ for the algebra of $n \times n$ matrices over $R$, which we refer to as a \emph{full matrix algebra} (or simply a \emph{matrix algebra}). We denote by $\operatorname{Mat}_R$ the groupoid of full matrix algebras over $R$, and by $\operatorname{Az}(R)$ the groupoid of Azumaya $R$-algebras.
\\

\noindent \textbf{Acknowledgements.} We thank Jonathan Block for his hospitality at Penn during BY’s visit and for introducing the two authors to each other. We are grateful to Michael Freedman for his foundational work on QCA \cite{freedman2022immersions}, which strongly shaped our perspective. We thank Daniel Krashen for helpful discussions on algebraic details, especially concerning Lemma~\ref{lem::centralizer}, Lemma~\ref{lem::tensor_splitting}, Corollary~\ref{cor::tensor_factor_splitting}, and Lemma~\ref{lem::tensor_complement_exist}. We thank Agn\`{e}s Beaudry for extensive comments and suggestions on the early draft. We thank Ulrich Bunke, Dan Freed, Matthias Ludewig, and Tomer Schlank for helpful conversations.

MJ thanks Mark Behrens and Mona Merling for many helpful conversations and their extensive support throughout this project. MJ also thanks Michael Freedman for the invitation to speak at Harvard’s Freedman Seminar, Shane Kelly for discussions related to Section~\ref{sec::compute} and for pointing out Remark~\ref{rmk:when_sk_vanish}, and Peter May for discussions concerning \cite{May1977EInfinityRingSpaces}. MJ thanks Charles Weibel for helpful conversations, particularly regarding \cite{bass1968algebraic, May1977EInfinityRingSpaces, c80f3555-900f-35b5-ab05-6d397e6a44e6, 951cf383-5a4d-3775-b8b2-0bafd2f162fe, pedersen_weibel}, and for helpful comments on the early draft.

BY thanks Michael Hopkins and Anton Kapustin for extensive discussions and encouragement throughout this project. BY is indebted to Shmuel Weinberger for introducing him to the work of Pedersen and Weibel during a visit to Chicago, which Shmuel kindly arranged. Last but not least, BY benefited from Mike’s chocolate and jokes on numerous occasions.

MJ is partially supported by the National Science Foundation Graduate Research Fellowship (DGE-2236662). BY acknowledges support from the Simons Foundation through the Simons Collaboration on Global Categorical Symmetries.

\section{The Quantum Cellular Automata Group}\label{sec::group}
In this section, we introduce a setup that leads to a notion of QCA over any commutative ring \(R\). Our definitions are new, though strongly inspired by the original formulations in physics. When \(R = \CC\), the resulting notion differs from the original definition in that we do not impose a \(*\)-algebra structure. In Section~\ref{sec::star}, we discuss the original definition of QCA.

\subsection{Quantum Spin System}

Let $R$ be a commutative ring with identity and $(X, \rho)$ be a metric space. We use $\NN$ to denote the set of positive integers. 

\begin{defn}
A  functions $q: X \rightarrow \NN$ is called locally finite if 
\begin{equation}
    \Lambda=\{x\in X: q_x>1\}
\end{equation} is a locally finite subset of $X$. That is, the intersection between $\Lambda$ and any bounded subset of $X$ is finite\footnote{In particular, $\Lambda$ is finite if $X$ is bounded.}. Frequently, $\Lambda$ is called a `lattice'. We denote the collection of such functions by $\NN^X_\mathrm{lf}$. This forms a partially ordered set with $q\leq r$ if and only if $q_x\mid r_x$ for all $x\in X$.
\end{defn}

To motivate the next definition, we note that $\NN$ under multiplication is a commutative monoid. Moreover, $\NN$ can be identified with the monoid of isomorphism classes of full matrix algebras over $R$ under tensor product by 
$$q\in \NN \mapsto \Mat(R^q),$$
where $\Mat(R^q)$ is the ring of $q \times q$ matrices over $R$.
\begin{defn}
    An \textit{$R$-quantum spin system} consists of a commutative ring $R$, a metric space $(X, \rho)$, and a locally finite function $q\in \NN^X_\lf$. Its \textit{algebra of local observables} is defined as
    \begin{equation}
  \SA(X,q) = \bigotimes_{x\in X} \Mat(R^{q_x})= \varinjlim_{B\subset X \text{ bounded }} \bigotimes_{x\in B} \Mat(R^{q_x}). 
\end{equation}
The tensor products are taken over $R$.
    
\end{defn}
We elaborate on the definition. Let $\mathcal{P}_0(\Lambda)$ denote the set of finite subsets of $\Lambda$. For 
$\Gamma \in\mathcal{P}_0(\Lambda)$ define
\begin{equation}
  \SA(\Gamma,q) := \bigotimes_{x\in \Gamma} \Mat(R^{q_x}),
\end{equation}
which is a finite tensor product over $R$.
For any $Y\subset X$ bounded, $Y\cap \Lambda\in \mathcal P_0(\Lambda)$ because of local finiteness of $q$. Define the algebra of \emph{observables supported on $Y$} as 
\begin{equation}
    \SA(Y, q)= \SA(Y\cap \Lambda, q).
\end{equation}
If $\Gamma\subset \Gamma'$ are finite subsets, define the inclusion
\begin{equation}
  \iota_{\Gamma,\Gamma'}\colon \SA(\Gamma, q) \longrightarrow \SA(\Gamma',q),\qquad 
  \iota_{\Gamma,\Gamma'}(A) = A \otimes \Id_{\Gamma'\setminus \Gamma},  
\end{equation}
where $\Id_{\Gamma'\setminus \Gamma}$ is the identity element in $\SA({\Gamma'\setminus \Gamma},q)$.  
Then $\{\SA(\Gamma,q)\}_{\Gamma\in\mathcal{P}_0(\Lambda)}$ forms a directed system.  

For any $Z\subset X$ not necessarily bounded, define the inductive limit
\begin{equation}
      \SA(Z,q) = \varinjlim_{\Gamma\in \mathcal{P}_0( \Lambda\cap Z)} \SA(\Gamma,q). 
\end{equation}
In particular, 
  
\begin{equation}
\SA(X,q) = \varinjlim_{\Gamma\in \mathcal{P}_0(\Lambda)} \SA(\Gamma,q).
\end{equation}

\begin{defn}
Let \(X\) be a metric space and let \(Y\subseteq X\). We say that \(Y\) is \emph{coarsely dense} in \(X\) if there exists $r>0$ such that every point \(x\in X\) satisfies
\[
\rho(x,Y)\leq r.
\]
\end{defn}

In physical applications, the index set
\(\Lambda\) is always a countable, uniformly discrete, and locally finite of bounded geometry. For the rest of the article we always assume a \textit{metric space \(X\) contains a countable, coarsely dense subspace that is uniformly discrete, locally finite metric space of bounded geometry}. See Appendix~\ref{sec::coarse} for relevant definitions. Typical examples include bounded metric spaces, $\ZZ^n$, $\RR^n$ and complete Riemannian manifolds of bounded geometry.
\begin{remark}
Algebras of this type have been studied under the name \emph{supernatural matrix algebras}
\cite{bar2023locally, bar2023algebra}. 
Under our assumption, the algebra \(\SA(X,q)\) may be realized as a \emph{countable}
direct limit. Fix an enumeration
\[
\Lambda=\{\lambda_1,\lambda_2,\lambda_3,\dots\}.
\]
Then the increasing chain of finite subsets
\[
\{\lambda_1\}
\subset
\{\lambda_1,\lambda_2\}
\subset
\{\lambda_1,\lambda_2,\lambda_3\}
\subset
\cdots
\]
is cofinal in \(\mathcal{P}_0(\Lambda)\).
Consequently, the directed system defining \(\SA(X,q)\) may be computed along this
sequence, and \(\SA(X,q)\) is the corresponding countable direct limit. Section~9
of \cite{bar2023locally} is especially relevant for the analysis of this countable
direct-limit.
\end{remark}

\begin{prop}\label{prop::free_module}
    For metric space $X$ and $q\in \NN^X_\lf$, $\SA(X, q)$ is a free $R$-module. 
\end{prop}
\begin{proof}
    For each $x\in X$, choose a basis $B_x$ for $\Mat(R^{q_x})$. Assume $\Id\in B_x$ for all $x$. The set
    \begin{equation}
S = \left\{
\bigotimes_{x \in X} b_x \;\middle|\;
b_x \in B_x \text{ and } b_x = \Id \text{ for all but finitely many } x
\right\}
\end{equation}
    is a $R$-linear basis. Indeed, any element $a \in \mathcal{A}(X, q)$ has finite support and is hence contained in some finite dimensional matrix algebra, which is free and whose chosen basis is contained in $S$. This shows that $S$ is spanning. Linear independence follows similarly a finite linear combination of elements in $S$ still has finite support. Note this basis is always countable because of our assumption on $(X, q)$. 
\end{proof}
Any two systems given by $q$ and $q'$ \textit{stack} to a new system given by their point-wise product $q\cdot q'$. Furthermore, we fix a concrete algebra isomorphism $\Phi$ between $\SA(X, q)\otimes \SA(X, q')$ and $\SA(X, q \cdot q')$ given point-wise by the Kronecker product. 
\begin{defn}\label{def::pointwise_stacking}
For positive integers \(q_x, q'_x\), the \emph{Kronecker product isomorphism}
\begin{equation}
   \Phi_x :
\Mat(R^{q_x}) \otimes \Mat(R^{q'_x})
\;\longrightarrow\;
\Mat(R^{q_x q'_x}) 
\end{equation}
is defined by
\begin{equation}
   \Phi_x(A \otimes B)
:=
\begin{pmatrix}
a_{11} B & a_{12} B & \cdots & a_{1 q_x} B \\
a_{21} B & a_{22} B & \cdots & a_{2 q_x} B \\
\vdots   & \vdots   & \ddots & \vdots   \\
a_{q_x 1} B & a_{q_x 2} B & \cdots & a_{q_x q_x} B
\end{pmatrix},
\end{equation}
where
\begin{equation*}
 A = (a_{ij}) \in \Mat(R^{q_x}),\;
B \in \Mat(R^{q'_x}).
\end{equation*}
We let $\Phi_x$ be the identity if either $q_x=1$ or $q'_x=1$.
\end{defn}

\begin{remark}[Inverse map]
Given \(C \in \Mat(R^{q_x q'_x})\), write it as a
\(q_x\times q_x\) block matrix
\[
C =
\begin{pmatrix}
C_{11} & C_{12} & \cdots & C_{1 q_x} \\
C_{21} & C_{22} & \cdots & C_{2 q_x} \\
\vdots & \vdots & \ddots & \vdots \\
C_{q_x 1} & C_{q_x 2} & \cdots & C_{q_x q_x}
\end{pmatrix},
\qquad
C_{ij} \in \Mat(R^{q'_x}).
\]
Then the inverse map
\(\Phi_x^{-1} : \Mat(R^{q_x q'_x})
\to \Mat(R^{q_x})\otimes \Mat(R^{q'_x})\)
is given by
\[
\Phi_x^{-1}(C)
=
\sum_{i,j=1}^{q_x} e^{(q_x)}_{ij} \otimes C_{ij},
\]
where \(e^{(q_x)}_{ij}\) is the matrix unit in $\Mat(R^{q_x})$.
\end{remark}

To define the equivalence of quantum spin systems, we need the notion of locality-preserving isomorphism. 

\begin{defn}\label{def:locality_preserve_iso}
Given two $R$-quantum spin systems $q, q'$ over a metric space $(X,\rho)$, a 
\emph{locality-preserving homomorphism} (respectively, \emph{isomorphism}) 
is an $R$-algebra homomorphism (respectively, isomorphism)
\[
\alpha : \SA(X,q) \longrightarrow \SA(X,q')
\]
of finite spread. That is, there exists $\ell >0$ such that for every $x\in X$ and every
\[
A \in \SA(\{x\},q) = \mathrm{Mat}(R^{q_x}),
\]
we have
\[
\alpha(A) \in \SA(D_\ell(x),q'),
\qquad
D_\ell(x) := \{y\in X : \rho(x,y)\le \ell\}.
\]

A spin system equipped with a locality-preserving automorphism $(q,\alpha)$ is called a 
\emph{quantum cellular automaton} (QCA). For fixed $q$, the set of QCA forms a group under composition, denoted $\CQ(X,q)$.
\end{defn}

Locality-preserving isomorphisms define an equivalence relation between spin systems. Whether two spin systems are equivalent is closely tied to the coarse homology of the underlying metric space $(X,\rho)$ and, in many cases, is independent of the choice of $R$. The precise relationship is established in Section~\ref{sec::space}, and Appendix~\ref{sec::coarse} provides a brief overview of coarse homology theory.

In practice, this perspective allows us to identify a quantum spin system on a space $X$ with one defined on a countable subset of $X$ having the same large-scale geometry. Here, we illustrate this principle with several representative examples.

\begin{exmp}\label{exp::quantum_spin_system_example}
Let $X$ be a bounded metric space and fix a point $x_0 \in X$. 
Any spin system $q$ on $X$ is equivalent to the spin system $q'$ defined by
\[
q'_{x_0} = \prod_{x\in X} q_x < \infty,
\qquad
q'_{x'} = 1 \quad \text{for } x' \ne x_0.
\]
\end{exmp}

\begin{exmp}\label{exp::real}
Let $X=\RR$. Any spin system $q$ on $\RR$ is equivalent to the spin system $q'$ defined by
\[
q'_n = \prod_{\lfloor x \rfloor = n} q_x < \infty
\quad \text{for each } n\in\ZZ,\text{ and }
q'_{x'} = 1 \quad \text{for } x' \notin \ZZ.
\]
\end{exmp}

\begin{exmp}
For any coarsely dense subspace $\Lambda \subseteq X$, we can map any spin system $q$ on $X$ to a spin system $q'$ on $\Lambda$. As $\Lambda$ is coarsely dense, there exists a partition $\{X_\lambda: \lambda \in \Lambda\}$ with $\lambda \in X_\lambda$ and $X_\lambda$ is uniformly bounded in diameter. Let 
\begin{equation}
q'_\lambda=\prod_{x\in X_\lambda} q_x. 
\end{equation}
\end{exmp}
\begin{remark}
By our assumption on \(X\), any spin system under consideration may be taken to be supported on a subspace
\[
\Lambda \subseteq X
\]
which is countable, uniformly discrete, locally finite, and of bounded geometry.
\end{remark}

\begin{lem}\label{lem: stabilization}
    The Kronecker product induces a canonical inclusion
\begin{equation}
    \iota_{q\to r}:\CQ(X,q)\hookrightarrow \CQ(X,r),
\end{equation}
for any $q, r\in \NN^X_\mathrm{lf}$ with $q\mid r$. 
\end{lem}
\begin{proof}
    Write $r=q\cdot q'$. Given a QCA $\alpha: \SA(X,q)\rightarrow \SA(X,q)$, define \begin{equation}
        \tilde \alpha=\Phi\circ(\alpha\otimes \id)\circ \Phi^{-1}: \SA(X,q\cdot q')\rightarrow \SA(X,q\cdot q'), 
    \end{equation}
    where $\id$ is the identity automorphism on $\SA(X, q').$ 
\end{proof}
This inclusion, called \textit{stabilization}, allows composition between QCA defined on spin systems with distinct $q$. One could also stack two locality-preserving isomorphisms $\al$ and $\beta$ and create an isomorphism on the stacked systems. We will often suppress the conjugation by \( \Phi \), writing \( \alpha \otimes \beta \) in place of \( \Phi \circ (\alpha \otimes \beta) \circ \Phi^{-1} \) when stacking two locality-preserving isomorphisms.

\begin{defn}\label{def::qca_group}
The \emph{total QCA group} is the direct limit in the category of groups:
\[
\CQ(X)\ :=\ \varinjlim_{q\in \NN^X_\mathrm{lf}}\ \CQ(X,q).
\]
Concretely, $\CQ(X)$ is the set of equivalence classes $[q,\alpha]$ with
$\alpha\in\CQ(X,q)$, modulo the condition that
\[
[q,\alpha]\sim [r,\beta] \text{ if and only if there exists } s\ \text{with}\ q\mid s,\ r\mid s\ \text{such that}\
\iota_{q\to s}(\alpha)=\iota_{r\to s}(\beta).
\]
The product is
\[
[q,\alpha]\cdot[r,\beta]\;\coloneqq\;
\big[s,\ \iota_{q\to s}(\alpha)\circ \iota_{r\to s}(\beta)\big],
\qquad s_x \coloneqq \mathrm{lcm}(q_x,r_x).
\]
\end{defn}
Note that $\mathcal{Q}(X)$ is in particular a filtered colimit of groups. Sometimes, we write $\CQ(X;R)$ to emphasize the base ring $R$.

\subsection{Quantum Circuits}\label{subsec::circuits}
We record several important subgroups of~$\CQ(X)$. Although their elements admit simple descriptions, they represent a rich class of QCA and are widely studied in quantum information and quantum many-body physics. We also observe that each of these subgroups contains the commutator subgroup $[\CQ(X), \CQ(X)]$ of~$\CQ(X)$, a fact that will be particularly relevant in later sections.
\begin{defn}
Let $q \in \NN^{X}_{\mathrm{lf}}$ be a spin system.  
A \emph{single-layer circuit} consists of:

\begin{enumerate}
\item A uniformly bounded partition
\[
X = \coprod_{j} X_j,
\qquad
\sup_j \mathrm{diam}(X_j) < \infty,
\]
and we define
\[
q(X_j) := \prod_{x\in X_j} q_x.
\]

\item For each $X_j$, a choice of automorphism $\alpha_j$ of one of the following types:
\begin{itemize}
    \item \textit{General automorphism:}
    \[
    \alpha_j \in \Aut_R \SA(X_j,q)
    \cong
    \Aut_R \Mat(R^{q(X_j)}),
    \]

    \item \textit{Inner automorphism:}
    \[
    \alpha_j \in \PGL_{q(X_j)}(R),
    \]
    i.e.\ conjugation by an invertible matrix.

    \item \textit{Special linear automorphism:}
    \[
    \alpha_j \in \PSL_{q(X_j)}(R),
    \]
    i.e.\ conjugation by a matrix of determinant one.

    \item \textit{Elementary automorphism:}
    \[
    \alpha_j(A)=EAE^{-1}
    \quad \text{for some elementary matrix } E.
    \]
\end{itemize}
\end{enumerate}

The product $\bigotimes_j \alpha_j$ defines a locality-preserving automorphism of $\SA(X,q)$ by acting independently on the disjoint tensor factors of $\SA(X,q)$.

Figure~\ref{fig::qca_example} illustrates a single block in a single-layer circuit.

\medskip

We denote by
\[
\CQ(X,q)\supseteq
\CalC(X,q)\supseteq
\CalC^{\mathrm{inn}}(X,q)\supseteq
\CalC^{\mathrm{sp}}(X,q)\supseteq
\CalC^{\mathrm{el}}(X,q)
\]
the subgroups generated by single-layer circuits of the corresponding gate types.
Passing to the direct limit in $q$, we obtain
\[
\CQ(X) \supseteq 
\CalC(X) \supseteq 
\CalC^{\mathrm{inn}}(X) \supseteq 
\CalC^{\mathrm{sp}}(X) \supseteq 
\CalC^{\mathrm{el}}(X).
\]
\end{defn}

\begin{figure}[htb]
    \centering
\[\scalebox{1.2}{
\begin{tikzpicture}[x=0.75pt,y=0.75pt,yscale=-1,xscale=1]

%Shape: Circle [id:dp41529935480977775] 
\draw  [fill={rgb, 255:red, 0; green, 0; blue, 0 }  ,fill opacity=1 ] (227,143.75) .. controls (227,141.4) and (228.9,139.5) .. (231.25,139.5) .. controls (233.6,139.5) and (235.5,141.4) .. (235.5,143.75) .. controls (235.5,146.1) and (233.6,148) .. (231.25,148) .. controls (228.9,148) and (227,146.1) .. (227,143.75) -- cycle ;
%Shape: Circle [id:dp02921931563067326] 
\draw  [fill={rgb, 255:red, 0; green, 0; blue, 0 }  ,fill opacity=1 ] (291,143.75) .. controls (291,141.4) and (292.9,139.5) .. (295.25,139.5) .. controls (297.6,139.5) and (299.5,141.4) .. (299.5,143.75) .. controls (299.5,146.1) and (297.6,148) .. (295.25,148) .. controls (292.9,148) and (291,146.1) .. (291,143.75) -- cycle ;
%Straight Lines [id:da9875214711309649] 
\draw    (231.25,143.75) -- (295.25,143.75) ;
%Shape: Circle [id:dp012908570789427531] 
\draw  [fill={rgb, 255:red, 0; green, 0; blue, 0 }  ,fill opacity=1 ] (355,143.75) .. controls (355,141.4) and (356.9,139.5) .. (359.25,139.5) .. controls (361.6,139.5) and (363.5,141.4) .. (363.5,143.75) .. controls (363.5,146.1) and (361.6,148) .. (359.25,148) .. controls (356.9,148) and (355,146.1) .. (355,143.75) -- cycle ;
%Straight Lines [id:da8757635654383464] 
\draw    (295.25,143.75) -- (359.25,143.75) ;
%Shape: Circle [id:dp9070261925340832] 
\draw  [fill={rgb, 255:red, 0; green, 0; blue, 0 }  ,fill opacity=1 ] (419,143.75) .. controls (419,141.4) and (420.9,139.5) .. (423.25,139.5) .. controls (425.6,139.5) and (427.5,141.4) .. (427.5,143.75) .. controls (427.5,146.1) and (425.6,148) .. (423.25,148) .. controls (420.9,148) and (419,146.1) .. (419,143.75) -- cycle ;
%Straight Lines [id:da9545262904283817] 
\draw    (359.25,143.75) -- (423.25,143.75) ;
%Shape: Circle [id:dp9863378920217708] 
\draw  [fill={rgb, 255:red, 0; green, 0; blue, 0 }  ,fill opacity=1 ] (165,143.75) .. controls (165,141.4) and (166.9,139.5) .. (169.25,139.5) .. controls (171.6,139.5) and (173.5,141.4) .. (173.5,143.75) .. controls (173.5,146.1) and (171.6,148) .. (169.25,148) .. controls (166.9,148) and (165,146.1) .. (165,143.75) -- cycle ;
%Straight Lines [id:da8264524118896682] 
\draw    (169.25,143.75) -- (233.25,143.75) ;
%Straight Lines [id:da023278627104373784] 
\draw    (123.25,143.75) -- (187.25,143.75) ;
%Straight Lines [id:da7395466100344348] 
\draw    (407.25,143.75) -- (471.25,143.75) ;
%Straight Lines [id:da9496181767563671] 
\draw [color={rgb, 255:red, 208; green, 2; blue, 27 }  ,draw opacity=1 ]   (295,107) -- (295,71.5) ;
\draw [shift={(295,69.5)}, rotate = 90] [color={rgb, 255:red, 208; green, 2; blue, 27 }  ,draw opacity=1 ][line width=0.75]    (10.93,-3.29) .. controls (6.95,-1.4) and (3.31,-0.3) .. (0,0) .. controls (3.31,0.3) and (6.95,1.4) .. (10.93,3.29)   ;
%Straight Lines [id:da6551590460897346] 
\draw [color={rgb, 255:red, 208; green, 2; blue, 27 }  ,draw opacity=1 ]   (260,55) -- (172.78,99.59) ;
\draw [shift={(171,100.5)}, rotate = 332.92] [color={rgb, 255:red, 208; green, 2; blue, 27 }  ,draw opacity=1 ][line width=0.75]    (10.93,-3.29) .. controls (6.95,-1.4) and (3.31,-0.3) .. (0,0) .. controls (3.31,0.3) and (6.95,1.4) .. (10.93,3.29)   ;
%Straight Lines [id:da6336377952448666] 
\draw [color={rgb, 255:red, 208; green, 2; blue, 27 }  ,draw opacity=1 ]   (325,56.5) -- (362.74,102.95) ;
\draw [shift={(364,104.5)}, rotate = 230.91] [color={rgb, 255:red, 208; green, 2; blue, 27 }  ,draw opacity=1 ][line width=0.75]    (10.93,-3.29) .. controls (6.95,-1.4) and (3.31,-0.3) .. (0,0) .. controls (3.31,0.3) and (6.95,1.4) .. (10.93,3.29)   ;
%Straight Lines [id:da5137613482563] 
\draw [color={rgb, 255:red, 208; green, 2; blue, 27 }  ,draw opacity=1 ]   (164,80) -- (237.24,40.45) ;
\draw [shift={(239,39.5)}, rotate = 151.63] [color={rgb, 255:red, 208; green, 2; blue, 27 }  ,draw opacity=1 ][line width=0.75]    (10.93,-3.29) .. controls (6.95,-1.4) and (3.31,-0.3) .. (0,0) .. controls (3.31,0.3) and (6.95,1.4) .. (10.93,3.29)   ;
%Straight Lines [id:da6752848213629531] 
\draw [color={rgb, 255:red, 208; green, 2; blue, 27 }  ,draw opacity=1 ]   (374,96) -- (339.31,56.01) ;
\draw [shift={(338,54.5)}, rotate = 49.06] [color={rgb, 255:red, 208; green, 2; blue, 27 }  ,draw opacity=1 ][line width=0.75]    (10.93,-3.29) .. controls (6.95,-1.4) and (3.31,-0.3) .. (0,0) .. controls (3.31,0.3) and (6.95,1.4) .. (10.93,3.29)   ;

% Text Node
\draw (218,154.4) node [anchor=north west][inner sep=0.75pt]    {$-1$};
% Text Node
\draw (289,155.4) node [anchor=north west][inner sep=0.75pt]    {$0$};
% Text Node
\draw (353,155.4) node [anchor=north west][inner sep=0.75pt]    {$1$};
% Text Node
\draw (417,155.4) node [anchor=north west][inner sep=0.75pt]    {$2$};
% Text Node
\draw (156,154.4) node [anchor=north west][inner sep=0.75pt]    {$-2$};
% Text Node
\draw (98,130.4) node [anchor=north west][inner sep=0.75pt]  [font=\large]  {$...$};
% Text Node
\draw (478,135.4) node [anchor=north west][inner sep=0.75pt]  [font=\large]  {$...$};
% Text Node
\draw (447,26.4) node [anchor=north west][inner sep=0.75pt]    {$X\ =\ \mathbb{Z}$};
% Text Node
\draw (271,117.4) node [anchor=north west][inner sep=0.75pt]  [font=\small,color={rgb, 255:red, 208; green, 2; blue, 27 }  ,opacity=1 ]  {$\operatorname{Mat}(R^{q_0})$};
% Text Node
\draw (332,109.4) node [anchor=north west][inner sep=0.75pt]  [font=\small,color={rgb, 255:red, 208; green, 2; blue, 27 }  ,opacity=1 ]  {$\operatorname{Mat}(R^{q_1})$};
% Text Node
\draw (401,109.4) node [anchor=north west][inner sep=0.75pt]  [font=\small,color={rgb, 255:red, 208; green, 2; blue, 27 }  ,opacity=1 ]  {$\operatorname{Mat}(R^{q_2})$};
% Text Node
\draw (138,111.4) node [anchor=north west][inner sep=0.75pt]  [font=\small,color={rgb, 255:red, 208; green, 2; blue, 27 }  ,opacity=1 ]  {$\operatorname{Mat}(R^{q_{-2}})$};
% Text Node
\draw (205,111.4) node [anchor=north west][inner sep=0.75pt]  [font=\small,color={rgb, 255:red, 208; green, 2; blue, 27 }  ,opacity=1 ]  {$\operatorname{Mat}(R^{q_{-1}})$};
% Text Node
\draw (245,28.4) node [anchor=north west][inner sep=0.75pt]  [font=\small,color={rgb, 255:red, 208; green, 2; blue, 27 }  ,opacity=1 ]  {$\underbrace{\operatorname{Mat}(R^a)} \otimes \underbrace{\operatorname{Mat}(R^b)}$};
% Text Node
\draw (305,85) node [anchor=north west][inner sep=0.75pt]    {$\textcolor[rgb]{0.82,0.01,0.11}{=}$};
% Text Node
\draw (76,180) node [anchor=north west][inner sep=0.75pt]  [font=\small,color={rgb, 255:red, 208; green, 2; blue, 27 }  ,opacity=1 ]  {$\operatorname{Mat}(R^{q_{-2}}) =\operatorname{Mat}(R^a) \ \otimes \operatorname{Mat}(R^c)$};
% Text Node
\draw (287,179) node [anchor=north west][inner sep=0.75pt]  [font=\small,color={rgb, 255:red, 208; green, 2; blue, 27 }  ,opacity=1 ]  {$\operatorname{Mat}(R^{q_1}) =\operatorname{Mat}(R^b) \ \otimes \operatorname{Mat}(R^d)$};
\end{tikzpicture}
}
\]
\caption{An example of a single block in a single-layer circuit. Here we break $\operatorname{Mat}(R^{q_0})$ into its two tensor factors $\operatorname{Mat}(R^a)$ and $\operatorname{Mat}(R^b)$, we then swap out the correspondent tensor components $\operatorname{Mat}(R^a) \subseteq \operatorname{Mat}(R^{q_{-2}})$ and $\operatorname{Mat}(R^b) \subseteq \operatorname{Mat}(R^{q_1})$ to perform the automorphism. }
\label{fig::qca_example}
\end{figure}
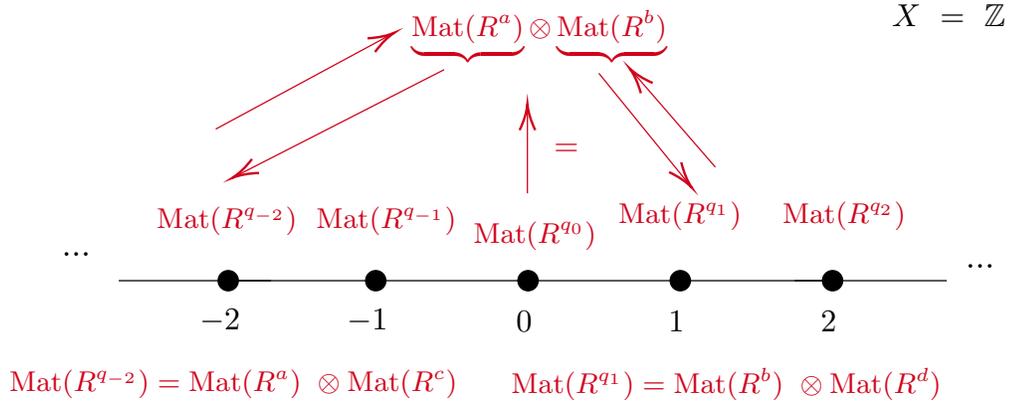

\begin{remark}
    Over the complex numbers, one could also consider circuits built from unitary matrices. In fact, these are what quantum circuits commonly refer to. See Section~\ref{sec::star} for details. 
\end{remark}

\begin{exmp}
   For $X = *$ is a point, we have that $\CQ(*) = \CalC(*)$.
\end{exmp}

\begin{exmp}\label{example::swap}
    Here we give an example of elementary quantum circuits. For a spin system $q\cdot r$, there are two factorizations based on the (inverse) Kronecker product
    \begin{equation}
\SA(X, q\cdot r)\cong \SA(X, q)\otimes \SA(X, r)\cong \SA(X, r)\otimes \SA(X, q),        \end{equation} one can define the $\mathrm{SWAP}_{q,r}: \mathcal{A}(X, q) \otimes \mathcal{A}(X, r) \to \mathcal{A}(X, r) \otimes \mathcal{A}(X, q)$ circuit which switches between these two orderings above. Specifically, it is the partition of $X$ into individual points, and on each point we have an automorphism by the following formula
\begin{equation}
    \begin{pmatrix}
a_{11} B & a_{12} B & \cdots & a_{1 q_x} B \\
a_{21} B & a_{22} B & \cdots & a_{2 q_x} B \\
\vdots   & \vdots   & \ddots & \vdots   \\
a_{q_x 1} B & a_{q_x 2} B & \cdots & a_{q_x q_x} B
\end{pmatrix}\mapsto \begin{pmatrix}
b_{11} A & b_{12} A & \cdots & b_{1 r_x} A \\
b_{21} A & b_{22} A & \cdots & b_{2 r_x} A \\
\vdots   & \vdots   & \ddots & \vdots   \\
b_{r_x 1} A & b_{r_x 2} A & \cdots & b_{r_x r_x} A
\end{pmatrix}
\end{equation}
for each $A=(a_{ij})\in \mathrm{Mat}(R^{q_x})$ and $B=(b_{ij})\in \mathrm{Mat}(R^{r_x})$. After stabilization, a SWAP automorphism is always equal to a conjugation by an elementary matrix. Indeed, the SWAP automorphism is the conjugation by a permutation matrix that is stably even (see \cite{Liu2024} for a proof of this), and even permutation matrices are elementary (see III.1.2.1 of \cite{weibel2013k}).
\end{exmp}

\begin{prop}
    Each of the four circuit subgroups is normal in $\CQ(X)$.
\end{prop}
\begin{proof}
    Let $\alpha=\bigotimes_{j\in J} \alpha_j$ be a single-layer (inner, special, or elementary) circuit, with each $\alpha_j$ an automorphism of $\SA(X_j,q)$ belonging to the corresponding gate-type. Recall that  
    \begin{equation}
    X=\coprod_{j\in J} X_j \text{ and } \sup_{j\in J}\mathrm{diam}(X_j)< \infty.
\end{equation}
For any $\beta\in \CQ(X,q)$ of spread $l$, $\beta\circ \alpha_j \circ \beta^{-1}$ is an automorphism of the same gate-type on $\SA(X_j^l, q)$ with 
\begin{equation}
    X_j^l=\{x\in X: d(x, X_j)\leq l\}.
\end{equation}
Above, we are regarding $\alpha_j$ as an automorphism of the entire algebra $\SA(X, q)$ which acts as identity outside $X_j$. 
In other words, 
\begin{equation}
\beta\circ \alpha \circ \beta^{-1}=\prod_j \beta\circ \alpha_j \circ \beta^{-1}
\end{equation} where the factors are mutually commuting but may overlap in support. Though not a single-layer circuit anymore, $\beta\circ \alpha \circ \beta^{-1}$ is still a (multi-layer) circuit of the same type. Indeed, one can stagger the terms in the product and regroup them so that operators with disjoint support are assigned to the same layer. The argument works because of bounded geometry.\footnote{This was pointed out by Matthias Ludewig.}
\end{proof}

\begin{prop}
    The quotient group $\CQ(X)/\CalC^{\mathrm{el}}(X)$ is abelian. In other words, $[\CQ(X), \CQ(X)]\subset \CalC^{\mathrm{el}}(X).$
\end{prop}
\begin{proof}
     Let $\alpha\in \CQ(X,q)$ and $\beta \in \CQ(X, r)$. Within the total QCA group $\CQ(X)$, 
    \begin{align*}
        &[q,\alpha]  [r, \beta]\\
        =& [qr,\alpha\otimes \id_r] [qr, \mathrm{SWAP}_{q, r}] [qr, \id_q\otimes \beta][qr, \mathrm{SWAP}_{q, r}]\\
        =& [qr,\alpha\otimes \id_r][qr, \id_q\otimes \beta][qr, \id_q\otimes \beta^{-1}] [qr, \mathrm{SWAP}_{q, r}] [qr, \id_q\otimes \beta][qr, \mathrm{SWAP}_{q, r}]\\
        =&[qr, \id_q\otimes \beta][qr,\alpha\otimes \id_r][qr, \id_q\otimes \beta^{-1}] [qr, \mathrm{SWAP}_{q, r}] [qr, \id_q\otimes \beta][qr, \mathrm{SWAP}_{q, r}]\\
        =& [r, \beta][q,\alpha][qr, \id_q\otimes \beta^{-1}] [qr, \mathrm{SWAP}_{q, r}] [qr, \id_q\otimes \beta][qr, \mathrm{SWAP}_{q, r}].
    \end{align*} 
Since both $[qr, \id_q\otimes \beta^{-1}] [qr, \mathrm{SWAP}_{q, r}] [qr, \id_q\otimes \beta]$ and $[qr, \mathrm{SWAP}_{q, r}]$ are in $\CalC^{\mathrm{el}}(X)$ we are done. 
\end{proof}

\begin{cor}
    The groups $\CQ(X)/\CalC(X)$, $\CQ(X)/\CalC^{\mathrm{inn}}(X)$, and $\CQ(X)/\CalC^{\mathrm{sp}}(X)$ are also abelian.
\end{cor}
\begin{prop} \label{prop::conjugation}
    Let $\beta$ be a locality-preserving isomorphism between $q$ and $q'$. A QCA $[q, \al]\in \CQ(X)$ and its conjugate $[q', \beta \al \beta^{-1}]$ are equivalent in $\CQ(X)^\mathrm{ab}$. Hence they are also equivalent in $\CQ(X)/\CalC(X)$, $\CQ(X)/\CalC^{\mathrm{inn}}(X)$, $\CQ(X)/\CalC^{\mathrm{sp}}(X)$, and $\CQ(X)/\CalC^{\mathrm{el}}(X)$.
\end{prop}
\begin{proof}
    The composition 
    $$\tilde \beta:=\mathrm{SWAP_{q', q}}\circ (\beta\otimes \beta^{-1}):\SA(X, q)\otimes \SA(X, q')\rightarrow \SA(X, q)\otimes \SA(X, q')$$ gives a QCA on $qq'.$ Notice, 
    \begin{align}
            &\tilde \beta \circ (\al\otimes \Id_q') \circ \tilde \beta^{-1}\\
            =&\mathrm{SWAP_{q', q}}\circ (\beta\otimes \beta^{-1})\circ (\al\otimes \Id_{q'}) \circ (\beta^{-1}\otimes \beta) \circ \mathrm{SWAP_{q, q'}}\\
            =& \Id_{q}\otimes (\beta \al \beta^{-1}),
    \end{align}
    where the first identity is on $\SA(X,q')$ and the last on $\SA(X, q).$ Therefore, in the abelianization $\CQ(X)^\mathrm{ab}$, 
    \begin{equation}
        [q, \al]=[qq', \al\otimes \Id_q']= [qq', \Id_{q}\otimes (\beta \al \beta^{-1})]=[q', \beta \al \beta^{-1}].
    \end{equation}
\end{proof}

\subsection{Relation to Azumaya algebra}
Recall that an \(R\)-algebra \(A\) is called an \emph{Azumaya algebra}
if there exists another \(R\)-algebra \(B\) such that
\[
A \otimes_R B \;\cong\; \mathrm{Mat}(R^n)
\]
for some \(n\).  Azumaya algebra over $R$ form a symmetric monoidal category $\mathrm{Az(R)}$ under tensor product. Let $K_0(\mathrm{Az(R)})$ be its Grothendieck group whose elements are represented by fractions of the form $\frac{A}{B}$. Here, $A, B$ are Azumaya algebras. If \(P\) is a faithfully projective\footnote{$P$ is a finitely generated projective $R$-module whose rank at every component of $\operatorname{Spec}(R)$ is non-zero. See II.5.4.2 of \cite{weibel2013k} for more details.} $R$-module, then \(\operatorname{End}_R(P)\)
is an Azumaya algebra. Since
\[
\operatorname{End}_R(P \otimes P') \;\cong\;
\operatorname{End}_R(P)\otimes \operatorname{End}_R(P'),
\]
there is a monoidal functor 
\[
\operatorname{End}_R : \mathrm{FP}(R) \longrightarrow \mathrm{Az}(R),
\]
where $\operatorname{FP}(R)$ denotes the monoidal category of finitely generated faithfully projective $R$-module with tensor product. Hence, we obtain a group homomorphism
\[
K_0(\mathrm{FP}(R)) \longrightarrow K_0(\mathrm{Az}(R)).
\]

The cokernel \(\operatorname{Br}(R)\) of this map is called the
\emph{Brauer group} of \(R\).
Thus \(\operatorname{Br}(R)\) is generated by classes \([A]\) of Azumaya
algebras, subject to the relations
\[
[A \otimes_R B] = [A] + [B], \qquad
[\operatorname{End}_R(P)] = 0.
\]

\begin{remark}
    A modern definition of Azumaya R-algebras typically states that $A$ is Azumaya if there exists an $R$-algebra $B$ such that $A \otimes_{R} B \cong \operatorname{End}_{R}(P)$, where $P$ is a faithfully projective $R$-module. As pointed out in Chapter IX.7 of \cite{bass1968algebraic}, the category of finitely generated free $R$-modules is actually cofinal in $\operatorname{FP}$ under tensor product. Hence, the two definitions agree with each other.
\end{remark}

For convenience, we also record a fact about Azumaya algebras that we will utilize in later sections.
\begin{lemma}[The Centralizer Theorem, See Theorem II.4.3 of \cite{DeMeyer1971}]\label{lem::centralizer}
    Let $R$ be a commutative ring. Let $A$ be a finite-dimensional Azumaya $R$-algebra, and let $B \subseteq A$ be a unital $R$-subalgebra which is also Azumaya over $R$. Let
\[
C := \{ a \in A \mid ab = ba \text{ for all } b \in B \}
\]
be the centralizer of $B$ in $A$. Then $C$ is an Azumaya $R$-algebra, and the natural multiplication map
\[
B \otimes_R C \longrightarrow A, 
\qquad b \otimes c \mapsto bc,
\]
is an isomorphism of $R$-algebras. Hence, $B$ and $C$ are mutual centralizers in $A$.
\end{lemma}

\KAz*    
    \begin{proof}
        Let $[q, \alpha]\in \CQ(\ZZ)$ for $q\in \NN^\ZZ_{\mathrm{lf}}$ and $\alpha\in \CQ(\ZZ, q)$ with spread $l$. 
        Choose any $\gamma\in \ZZ$. Informally, we would like to define $ b(\al)$ as
        \begin{equation}\frac{\al(\SA(\ZZ\cap (-\infty, \gamma], q))}{\SA(\ZZ\cap (-\infty, \gamma], q)} \in K_0\mathrm{Az(R)}.
        \end{equation}
    To make it well-defined, we use the finite spread condition, 
    \begin{equation}
       \SA(\ZZ\cap (-\infty, \gamma-l], q)) \subset\al(\SA(\ZZ\cap (-\infty, \gamma], q)) \subset \SA(\ZZ\cap (-\infty, \gamma+l], q)).
    \end{equation}
    Lemma~\ref{lem::tensor_splitting} below implies 
    \begin{equation}
        \al(\SA(\ZZ\cap (-\infty, \gamma], q))=\SA(\ZZ\cap (-\infty, \gamma-l], q)\otimes B,
    \end{equation}
        where $B\subset \SA(\ZZ\cap (\gamma-l, \gamma+l], q)$ is a unital subalgebra. 
    Similarly, 
        \begin{equation}
        \al(\SA(\ZZ\cap (\gamma, \infty), q))=C\otimes \SA(\ZZ\cap (\gamma+l,\infty), q),
    \end{equation}
        where $C\subset \SA(\ZZ\cap (\gamma-l, \gamma+l], q)$ is a subalgebra.

    As $\alpha$ is an algebra automorphism, $B\otimes C= \SA(\ZZ\cap (\gamma-l, \gamma+l],q)$, elements in $B$ and $C$ commute, and $B\cap C=R$. This implies $B$ is Azumaya. Now the expression for $b(\alpha)$ becomes
    \begin{equation} \frac{\SA(\ZZ\cap (-\infty, \gamma-l], q)\otimes B}{\SA(\ZZ\cap (-\infty, \gamma-l], q)\otimes \SA(\ZZ\cap (\gamma-l, \gamma], q)}= \frac{B}{\SA(\ZZ\cap (\gamma-l, \gamma], q)}\in K_0\mathrm{Az(R)}.
        \end{equation}
The right-hand-side is the definition of $b(\alpha)$. Using a different $l$ does not change the value as long as it is greater than the spread of $\alpha$. 
Since 
\begin{equation}
    \frac{\al(\SA(\ZZ\cap (-\infty, \gamma+1], q))}{\SA(\ZZ\cap (-\infty, \gamma+1], q))}=\frac{\al(\SA(\ZZ\cap (-\infty, \gamma], q))\otimes\al(\SA(\{\gamma+1\}, q))}{\SA(\ZZ\cap (-\infty, \gamma], q))\otimes\SA(\{\gamma+1\}, q)},
\end{equation}
and $\SA(\{\gamma+1\}, q)\cong \mathrm{Mat}(R^{q_{\gamma+1}})\cong\al(\SA(\{\gamma+1\}, q))$, $b(\al)$ does not depend on $\gamma$. Stabilization, or Kronecker product between $\al$ and identity, does not change $b(\al)$ either. Without loss of generality, let $[q, \al], [q, \beta]\in \CQ(\ZZ)$ have the same $q$. Verify that 
\begin{align*}
    b(\beta \al)&=\frac{\beta\al(\SA(\ZZ\cap (-\infty, \gamma], q))}{\SA(\ZZ\cap (-\infty, \gamma], q))}\\
    &=\frac{\beta(\SA(\ZZ\cap (-\infty, \gamma-l], q))\otimes \beta(B)}{\SA(\ZZ\cap (-\infty, \gamma-l], q)\otimes \SA(\ZZ\cap (\gamma-l, \gamma], q)}\\
    &=\frac{\beta(\SA(\ZZ\cap (-\infty, \gamma-l], q))}{\SA(\ZZ\cap (-\infty, \gamma-l], q)}\otimes \frac{ \beta(B)}{\SA(\ZZ\cap (\gamma-l, \gamma], q)}\\
    &=b(\beta)\otimes b(\alpha)\in K_0\mathrm{Az(R)},
\end{align*}
where, for the last equality, we used the independence of $\gamma$ and that $\beta(B)\cong B$. Thus, $b$ is a well-defined group homomorphism as claimed. 

To check that $\CalC(\ZZ)$ is in the kernel, simply take a single-layer circuit $\alpha = \bigotimes \al_j$ where $\al_j$'s are automorphisms of algebra supported on disjoint bounded intervals. Choose $\gamma$ to be the end point of one of the intervals so that $\al(\SA(\ZZ\cap (-\infty, \gamma], q))=\SA(\ZZ\cap (-\infty, \gamma], q)$ and $b(\al)=1.$

Conversely, if $b(\al)=1$ then $B\cong \SA(\ZZ\cap (\gamma-l, \gamma], q)$. Choose an automorphism $\beta_0\in \mathrm{Aut}(\SA(\ZZ\cap (\gamma-l, \gamma+l],q))$ which sends $B$ to $\SA(\ZZ\cap (\gamma-l, \gamma], q)$. Then $\beta_0\al$ is a product of automorphisms on $\SA(\ZZ\cap (-\infty, \gamma], q)$ and $\SA(\ZZ\cap (\gamma, \infty))$ respectively. Choose similarly for all $k\in \ZZ$ $$\beta_k\in \mathrm{Aut}(\SA(\ZZ\cap (\gamma-(2k-1)l, \gamma+(2k+1)l],q))$$ and notice $\beta=\bigotimes_k \beta_k$ and $\beta \al$ are both in $\CalC(\ZZ)$. Therefore, $\al\in \CalC(\ZZ).$

For surjectivity, consider the element $\frac{A}{B} \in K_0(\operatorname{Az}(R))$ where $A$ and $B$ are both Azumaya $R$-algebras. By definition, there exists $R$-algebras $A'$ and $B'$ such that $A \otimes A' \cong \operatorname{Mat}(R^n)$ and $B \otimes B' \cong \operatorname{Mat}(R^m)$. Now consider the particular QCA  $\alpha$ with spin system $q$ illustrated in Figure~\ref{fig::surjectivity}. Now observe that $\alpha$ has spread $\leq 2$, and 
\[\alpha(\mathcal{A}(\Zbb \cap (-\infty, -1], q)) = \mathcal{A}(\Zbb \cap (-\infty, -3], q) \otimes \underbrace{B'}_{\text{on -2}} \otimes \underbrace{A \otimes A'}_{\text{on -1}} \otimes \underbrace{A}_{\text{on 1}}.\]
Then we have that
$$b(\alpha) = \frac{B' \otimes A \otimes A' \otimes A}{\mathcal{A}(\Zbb \cap (-3, -1], q)} = \frac{B' \otimes A \otimes A' \otimes A}{B \otimes B' \otimes A \otimes A'} = \frac{A}{B}.$$
\end{proof}
\begin{cor}
   Every $\al\in \CalC(\ZZ)$ is equal to a composition of at most two single-layer circuits.    
\end{cor}
\begin{proof}
    If $b(\al)=1$, we can repeat the procedure in the proof above to find a single-layer circuit $\beta$ such that $\beta\al$ is a single-layer circuit itself. 
\end{proof}

\begin{figure}[htb]
    \centering
\[\begin{tikzpicture}[x=0.75pt,y=0.75pt,yscale=-1,xscale=1]

%Shape: Circle [id:dp8192351982516747] 
\draw  [fill={rgb, 255:red, 0; green, 0; blue, 0 }  ,fill opacity=1 ] (247,163.75) .. controls (247,161.4) and (248.9,159.5) .. (251.25,159.5) .. controls (253.6,159.5) and (255.5,161.4) .. (255.5,163.75) .. controls (255.5,166.1) and (253.6,168) .. (251.25,168) .. controls (248.9,168) and (247,166.1) .. (247,163.75) -- cycle ;
%Shape: Circle [id:dp1494125308968981] 
\draw  [fill={rgb, 255:red, 0; green, 0; blue, 0 }  ,fill opacity=1 ] (311,163.75) .. controls (311,161.4) and (312.9,159.5) .. (315.25,159.5) .. controls (317.6,159.5) and (319.5,161.4) .. (319.5,163.75) .. controls (319.5,166.1) and (317.6,168) .. (315.25,168) .. controls (312.9,168) and (311,166.1) .. (311,163.75) -- cycle ;
%Straight Lines [id:da08880630242146059] 
\draw    (251.25,163.75) -- (315.25,163.75) ;
%Shape: Circle [id:dp5330459662045467] 
\draw  [fill={rgb, 255:red, 0; green, 0; blue, 0 }  ,fill opacity=1 ] (375,163.75) .. controls (375,161.4) and (376.9,159.5) .. (379.25,159.5) .. controls (381.6,159.5) and (383.5,161.4) .. (383.5,163.75) .. controls (383.5,166.1) and (381.6,168) .. (379.25,168) .. controls (376.9,168) and (375,166.1) .. (375,163.75) -- cycle ;
%Straight Lines [id:da5970158768558252] 
\draw    (315.25,163.75) -- (379.25,163.75) ;
%Shape: Circle [id:dp1555949704330255] 
\draw  [fill={rgb, 255:red, 0; green, 0; blue, 0 }  ,fill opacity=1 ] (439,163.75) .. controls (439,161.4) and (440.9,159.5) .. (443.25,159.5) .. controls (445.6,159.5) and (447.5,161.4) .. (447.5,163.75) .. controls (447.5,166.1) and (445.6,168) .. (443.25,168) .. controls (440.9,168) and (439,166.1) .. (439,163.75) -- cycle ;
%Straight Lines [id:da17772431741463846] 
\draw    (379.25,163.75) -- (443.25,163.75) ;
%Shape: Circle [id:dp05511724043407429] 
\draw  [fill={rgb, 255:red, 0; green, 0; blue, 0 }  ,fill opacity=1 ] (185,163.75) .. controls (185,161.4) and (186.9,159.5) .. (189.25,159.5) .. controls (191.6,159.5) and (193.5,161.4) .. (193.5,163.75) .. controls (193.5,166.1) and (191.6,168) .. (189.25,168) .. controls (186.9,168) and (185,166.1) .. (185,163.75) -- cycle ;
%Straight Lines [id:da5879365419273478] 
\draw    (189.25,163.75) -- (253.25,163.75) ;
%Straight Lines [id:da002743523416992133] 
\draw    (143.25,163.75) -- (207.25,163.75) ;
%Straight Lines [id:da6752757702228851] 
\draw    (427.25,163.75) -- (491.25,163.75) ;
%Straight Lines [id:da6920057393785505] 
\draw  [dash pattern={on 4.5pt off 4.5pt}]  (283,56.4) -- (285,248) ;
%Shape: Circle [id:dp6841835849951644] 
\draw  [fill={rgb, 255:red, 0; green, 0; blue, 0 }  ,fill opacity=1 ] (132,163.75) .. controls (132,161.4) and (133.9,159.5) .. (136.25,159.5) .. controls (138.6,159.5) and (140.5,161.4) .. (140.5,163.75) .. controls (140.5,166.1) and (138.6,168) .. (136.25,168) .. controls (133.9,168) and (132,166.1) .. (132,163.75) -- cycle ;
%Straight Lines [id:da35377716004810755] 
\draw    (121,163.75) -- (185,163.75) ;
%Curve Lines [id:da5406790158266856] 
\draw [color={rgb, 255:red, 208; green, 2; blue, 27 }  ,draw opacity=1 ]   (136,225.4) .. controls (150.48,232.64) and (202.2,237.64) .. (223.77,224.5) ;
\draw [shift={(226,223)}, rotate = 143.13] [fill={rgb, 255:red, 208; green, 2; blue, 27 }  ,fill opacity=1 ][line width=0.08]  [draw opacity=0] (8.93,-4.29) -- (0,0) -- (8.93,4.29) -- cycle    ;
%Curve Lines [id:da5729274898241261] 
\draw [color={rgb, 255:red, 208; green, 2; blue, 27 }  ,draw opacity=1 ]   (248,226.4) .. controls (262.48,233.64) and (327.24,238.08) .. (349.72,224.9) ;
\draw [shift={(352,223.4)}, rotate = 143.13] [fill={rgb, 255:red, 208; green, 2; blue, 27 }  ,fill opacity=1 ][line width=0.08]  [draw opacity=0] (8.93,-4.29) -- (0,0) -- (8.93,4.29) -- cycle    ;
%Curve Lines [id:da9234951334332673] 
\draw [color={rgb, 255:red, 208; green, 2; blue, 27 }  ,draw opacity=1 ]   (369,227.4) .. controls (383.48,234.64) and (448.24,239.08) .. (470.72,225.9) ;
\draw [shift={(473,224.4)}, rotate = 143.13] [fill={rgb, 255:red, 208; green, 2; blue, 27 }  ,fill opacity=1 ][line width=0.08]  [draw opacity=0] (8.93,-4.29) -- (0,0) -- (8.93,4.29) -- cycle    ;
%Curve Lines [id:da5420426269012757] 
\draw [color={rgb, 255:red, 208; green, 2; blue, 27 }  ,draw opacity=1 ]   (421,124.4) .. controls (379.84,95.98) and (343.48,95.41) .. (309.1,122.69) ;
\draw [shift={(307,124.4)}, rotate = 320.36] [fill={rgb, 255:red, 208; green, 2; blue, 27 }  ,fill opacity=1 ][line width=0.08]  [draw opacity=0] (8.93,-4.29) -- (0,0) -- (8.93,4.29) -- cycle    ;
%Curve Lines [id:da5977857369420386] 
\draw [color={rgb, 255:red, 208; green, 2; blue, 27 }  ,draw opacity=1 ]   (299,125.4) .. controls (257.84,96.98) and (221.48,96.41) .. (187.1,123.69) ;
\draw [shift={(185,125.4)}, rotate = 320.36] [fill={rgb, 255:red, 208; green, 2; blue, 27 }  ,fill opacity=1 ][line width=0.08]  [draw opacity=0] (8.93,-4.29) -- (0,0) -- (8.93,4.29) -- cycle    ;
%Curve Lines [id:da7011602105865165] 
\draw [color={rgb, 255:red, 208; green, 2; blue, 27 }  ,draw opacity=1 ]   (175,124.4) .. controls (133.84,95.98) and (97.48,95.41) .. (63.1,122.69) ;
\draw [shift={(61,124.4)}, rotate = 320.36] [fill={rgb, 255:red, 208; green, 2; blue, 27 }  ,fill opacity=1 ][line width=0.08]  [draw opacity=0] (8.93,-4.29) -- (0,0) -- (8.93,4.29) -- cycle    ;
%Curve Lines [id:da13725134166711317] 
\draw [color={rgb, 255:red, 208; green, 2; blue, 27 }  ,draw opacity=1 ]   (513,120.4) .. controls (471.84,91.98) and (465.25,93.33) .. (432.07,120.69) ;
\draw [shift={(430,122.4)}, rotate = 320.36] [fill={rgb, 255:red, 208; green, 2; blue, 27 }  ,fill opacity=1 ][line width=0.08]  [draw opacity=0] (8.93,-4.29) -- (0,0) -- (8.93,4.29) -- cycle    ;
%Curve Lines [id:da3509134115235506] 
\draw [color={rgb, 255:red, 208; green, 2; blue, 27 }  ,draw opacity=1 ]   (69,223.4) .. controls (83.4,230.6) and (100.56,239.27) .. (119.61,226.69) ;
\draw [shift={(122,225)}, rotate = 143.13] [fill={rgb, 255:red, 208; green, 2; blue, 27 }  ,fill opacity=1 ][line width=0.08]  [draw opacity=0] (8.93,-4.29) -- (0,0) -- (8.93,4.29) -- cycle    ;

% Text Node
\draw (238,174.4) node [anchor=north west][inner sep=0.75pt]    {$-1$};
% Text Node
\draw (309,175.4) node [anchor=north west][inner sep=0.75pt]    {$0$};
% Text Node
\draw (373,175.4) node [anchor=north west][inner sep=0.75pt]    {$1$};
% Text Node
\draw (437,175.4) node [anchor=north west][inner sep=0.75pt]    {$2$};
% Text Node
\draw (176,173.4) node [anchor=north west][inner sep=0.75pt]    {$-2$};
% Text Node
\draw (95,148.4) node [anchor=north west][inner sep=0.75pt]  [font=\large]  {$...$};
% Text Node
\draw (498,155.4) node [anchor=north west][inner sep=0.75pt]  [font=\large]  {$...$};
% Text Node
\draw (297,131.4) node [anchor=north west][inner sep=0.75pt]  [font=\small,color={rgb, 255:red, 208; green, 2; blue, 27 }  ,opacity=1 ]  {$\underbrace{B} \otimes B'$};
% Text Node
\draw (122,174.4) node [anchor=north west][inner sep=0.75pt]    {$-3$};
% Text Node
\draw (114,197.4) node [anchor=north west][inner sep=0.75pt]  [font=\small,color={rgb, 255:red, 208; green, 2; blue, 27 }  ,opacity=1 ]  {$\underbrace{A} \otimes A'$};
% Text Node
\draw (225,196.4) node [anchor=north west][inner sep=0.75pt]  [font=\small,color={rgb, 255:red, 208; green, 2; blue, 27 }  ,opacity=1 ]  {$\underbrace{A} \otimes A'$};
% Text Node
\draw (350,198.4) node [anchor=north west][inner sep=0.75pt]  [font=\small,color={rgb, 255:red, 208; green, 2; blue, 27 }  ,opacity=1 ]  {$\underbrace{A} \otimes A'$};
% Text Node
\draw (413,132.4) node [anchor=north west][inner sep=0.75pt]  [font=\small,color={rgb, 255:red, 208; green, 2; blue, 27 }  ,opacity=1 ]  {$\underbrace{B} \otimes B'$};
% Text Node
\draw (166,130.4) node [anchor=north west][inner sep=0.75pt]  [font=\small,color={rgb, 255:red, 208; green, 2; blue, 27 }  ,opacity=1 ]  {$\underbrace{B} \otimes B'$};
% Text Node
\draw (43,203.4) node [anchor=north west][inner sep=0.75pt]  [font=\large,color={rgb, 255:red, 208; green, 2; blue, 27 }  ,opacity=1 ]  {$...$};
% Text Node
\draw (501,115.4) node [anchor=north west][inner sep=0.75pt]  [font=\large,color={rgb, 255:red, 208; green, 2; blue, 27 }  ,opacity=1 ]  {$...$};
% Text Node
\draw (43,114.4) node [anchor=north west][inner sep=0.75pt]  [font=\large,color={rgb, 255:red, 208; green, 2; blue, 27 }  ,opacity=1 ]  {$...$};
% Text Node
\draw (476,203.4) node [anchor=north west][inner sep=0.75pt]  [font=\large,color={rgb, 255:red, 208; green, 2; blue, 27 }  ,opacity=1 ]  {$...$};

\end{tikzpicture}
\]
    \caption{A QCA $\alpha$ such that $b(\alpha) = \frac{A}{B} \in K_0(\operatorname{Az}_{R})$. Here we place an alternating sequence of $A \otimes A' \cong \operatorname{Mat}(R^n)$ and $B \otimes B' \cong \operatorname{Mat}(R^m)$ on $\mathbb{Z}$, and $\alpha$ simultaneously transports the tensor factor $B$ to the left and the tensor factor $A$ to the right. We divide the line between $-1$ and $0$ in the construction of the element $\beta(\alpha)$.}
    \label{fig::surjectivity}
\end{figure}
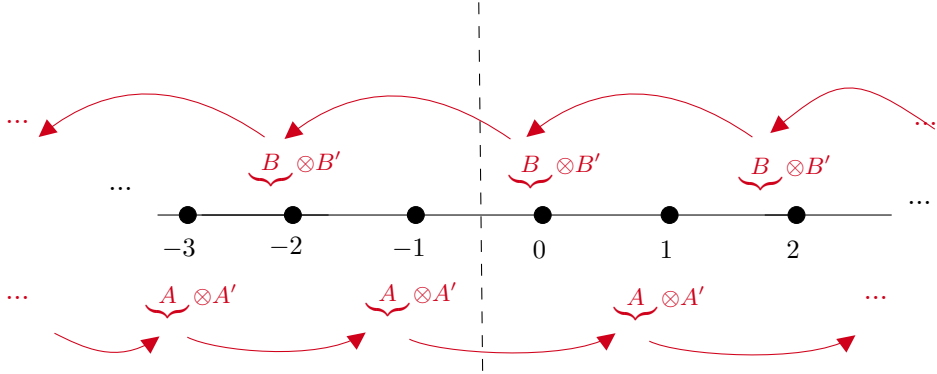

\begin{remark}\label{rmk::extend_factor}
Theorem~\ref{thm::k0_qz} shows that $\CQ(\ZZ)/\CalC(\ZZ)$ can be interpreted as an alternative definition of $K_0\mathrm{Az}(R)$. Intuitively, a QCA on $\ZZ$ can be viewed as pumping an Azumaya algebra along the chain. For a general metric space $X$, an analogous phenomenon occurs: a QCA defined on $X \times \ZZ$ produces a class of infinite-dimensional $R$-algebras. These play a role analogous to that of Azumaya algebras relative to matrix algebras in their relationship with $\SA(X,q)$. See Construction~\ref{defn::admissible} in Section~\ref{sec::spectra}.
\end{remark}

We end this subsection with some useful technical lemmas.

\begin{lemma}\label{lem::tensor_splitting}
Let $R$ be a commutative unital ring. Let $A$ and $C$ be unital $R$-algebras, and regard
$A$ as the unital subalgebra $A\otimes_R 1 \subseteq A\otimes_R C$.
Assume that $A$ is a \emph{locally matrix $R$-algebra}, i.e.\ there exists a directed system
of unital $R$-subalgebras $(A_\lambda)_{\lambda\in\Lambda}$ with
\[
A=\bigcup_{\lambda\in\Lambda} A_\lambda,
\qquad
A_\lambda \cong \Mat(R^{n(\lambda)})
\ \text{ as unital $R$-algebras.}
\]
Let $B\subseteq A\otimes_R C$ be a unital $R$-subalgebra such that
\[
A\otimes_R 1 \subseteq B
\quad\text{and}\quad
1_B = 1_{A\otimes C}.
\]
Define
\[
D \;:=\; \{\,c\in C \;:\; 1\otimes c \in B\,\} \ \subseteq\ C.
\]
Then $D$ is a unital $R$-subalgebra of $C$ and
\[
B \;=\; A\otimes_R D \ \subseteq\ A\otimes_R C.
\]
\end{lemma}

\begin{proof}
It can be verified that $D$ is a unital $R$-subalgebra of $C$. Moreover, we have
\[
(A\otimes 1)\cdot (1\otimes D) = A\otimes D \subseteq B.
\]
Thus, it remains to prove the reverse inclusion $B\subseteq A\otimes D$.

\medskip
\noindent\textbf{Step 1: The finite matrix case.}
Here we give an elementary proof but note that this can be significantly shortened using ideas behind what are called ``Peirce decomposition". We explain how the shortened proof works in Remark~\ref{rmk::short} later. Here we prove the following claim. 

Let $A_0=\Mat(R^n)$ and let $B_0\subseteq A_0\otimes_R C$ be a unital $R$-subalgebra with
$A_0\otimes 1\subseteq B_0$. If
\[
D_0 := \{\,c\in C:\ 1\otimes c \in B_0\,\},
\]
then $B_0 = A_0\otimes_R D_0$.

Fix a system of matrix units $\{e_{ij}\}_{1\le i,j\le n}\subseteq A_0$.
(So $e_{ij}e_{k\ell}=\delta_{jk}e_{i\ell}$ and $\sum_{i=1}^n e_{ii}=1$.)
Since $A_0\otimes 1\subseteq B_0$, we have $e_{ij}\otimes 1\in B_0$ for all $i,j$.

Let $b\in B_0$ be arbitrary. For each $1\le i,j\le n$, consider the element
\[
x_{ij} := (e_{1i}\otimes 1)\, b \,(e_{j1}\otimes 1)\ \in\ B_0,
\]
since $B_0$ is a subalgebra. Because
\[
(e_{1i}\otimes 1)(A_0\otimes C)(e_{j1}\otimes 1)
\;=\;
(e_{1i}A_0e_{j1})\otimes C
\;=\;
e_{11}\otimes C,
\]
there exists a \emph{unique} element $c_{ij}\in C$ such that
\[
x_{ij} = e_{11}\otimes c_{ij}.
\]
In particular $e_{11}\otimes c_{ij}\in B_0$.

We now show that $c_{ij}\in D_0$.
For each $1\le k\le n$ we compute, using $e_{k1}e_{11}e_{1k}=e_{kk}$,
\[
(e_{k1}\otimes 1)\,(e_{11}\otimes c_{ij})\,(e_{1k}\otimes 1)
=
(e_{kk}\otimes c_{ij})\in B_0.
\]
Summing over $k$ yields
\[
1\otimes c_{ij}
=
\left(\sum_{k=1}^n e_{kk}\right)\otimes c_{ij}
=
\sum_{k=1}^n (e_{kk}\otimes c_{ij})
\in B_0,
\]
hence $c_{ij}\in D_0$.

Finally we reconstruct $b$ from the $c_{ij}$:
\begin{align*}
\sum_{i,j=1}^n e_{ij}\otimes c_{ij}
&=
\sum_{i,j=1}^n (e_{i1}\otimes 1)(e_{11}\otimes c_{ij})(e_{1j}\otimes 1) \\
&=
\sum_{i,j=1}^n (e_{i1}\otimes 1)\,(e_{1i}\otimes 1)\, b \,(e_{j1}\otimes 1)\,(e_{1j}\otimes 1) \\
&=
\left(\sum_{i=1}^n e_{i1}e_{1i}\otimes 1\right)\, b \,\left(\sum_{j=1}^n e_{j1}e_{1j}\otimes 1\right) \\
&=
\left(\sum_{i=1}^n e_{ii}\otimes 1\right)\, b \,\left(\sum_{j=1}^n e_{jj}\otimes 1\right) \\
&=
(1\otimes 1)\, b \,(1\otimes 1)
=
b.
\end{align*}
Since each $c_{ij}\in D_0$, this shows $b\in A_0\otimes D_0$, hence
$B_0\subseteq A_0\otimes D_0$.

The reverse inclusion $A_0\otimes D_0\subseteq B_0$ is immediate because
$A_0\otimes 1\subseteq B_0$ and $1\otimes D_0\subseteq B_0$, and their products
generate $A_0\otimes D_0$. Thus $B_0=A_0\otimes D_0$.

\medskip
\noindent\textbf{Step 2: Reduction to the finite case using the locally matrix hypothesis.}
Let $b\in B$ be arbitrary. As an element of the algebraic tensor product
$A\otimes_R C$, it is a \emph{finite} $R$-linear combination of pure tensors:
\[
b=\sum_{t=1}^N a_t\otimes c_t,
\qquad a_t\in A,\ c_t\in C.
\]
Since $A=\bigcup_\lambda A_\lambda$, there exists $\lambda$ such that
$a_t\in A_\lambda$ for all $t$, hence $b\in A_\lambda\otimes_R C$.

Define the (unital) subalgebra
\[
B_\lambda := B\cap (A_\lambda\otimes_R C)\ \subseteq\ A_\lambda\otimes_R C.
\]
Then $A_\lambda\otimes 1 \subseteq B_\lambda$ because $A\otimes 1\subseteq B$
and $A_\lambda\subseteq A$. Apply the claim (Step 1) to $A_0:=A_\lambda\cong \Mat(R^{n(\lambda)})$
and $B_0:=B_\lambda$. We obtain
\[
B_\lambda = A_\lambda\otimes_R D_\lambda,
\qquad
D_\lambda := \{\,c\in C:\ 1\otimes c \in B_\lambda\,\}.
\]

We now identify $D_\lambda$ with $D$.
Note that $1\otimes c\in A_\lambda\otimes C$ for \emph{every} $\lambda$ and every $c\in C$,
because $1\in A_\lambda$. Hence
\[
1\otimes c \in B_\lambda
\quad\Longleftrightarrow\quad
1\otimes c \in B\ \text{ and }\ 1\otimes c\in A_\lambda\otimes C
\quad\Longleftrightarrow\quad
1\otimes c \in B.
\]
Thus $D_\lambda = D$.

Therefore,
\[
b\in B_\lambda = A_\lambda\otimes_R D \subseteq A\otimes_R D.
\]
Since $b\in B$ was arbitrary, we conclude $B\subseteq A\otimes_R D$.

\end{proof}

\begin{remark}\label{rmk::short}
It was pointed out to us by Daniel Krashen that the following argument significantly shortens the proof of Step 1. Proposition 13.9 of \cite{rowen2008graduate} states that for $W$ any ring that contains matrix units $\{e_{ij}\}_{1 \leq i, j \leq n}$ satisfying the identities given in the proof, then $W \cong M_n(e_{11} W e_{11})$. In the context of Step 1, we apply to $B_0$ and obtain the $e_{11} B_0 e_{11} \otimes A_0 \cong M_n(e_{11} B_0 e_{11}) \to B_0$ is an isomorphism. One then note that $e_{11} B_0 e_{11}$ is canonically isomorphic to $D_0$.
\end{remark}

\begin{cor}\label{cor::tensor_factor_splitting}
    The same holds if $A$ is a tensor factor of a locally matrix $R$-algebra with the same identity element, i.e., there exists a unital $R$-algebra $A'$ with $A\otimes_R A'$ locally matrix, provided that the natural map $B \to A' \otimes_R B$ is injective.  
\end{cor}

In particular, the map $B \to A' \otimes_R B$ is injective if (a) $B$ is flat and $R \subseteq A'$ or (b) if $A'$ is faithfully flat (Lemma 10.82.15. of \cite{stacks-project}).
\begin{proof}
    Since $B\subseteq A\otimes_R C$, $A'\otimes_R B\subseteq A'\otimes_R A\otimes_R C$. Apply the lemma, $A'\otimes_R B=A'\otimes_R A\otimes_R D$, with $D=\{c\in C: 1\otimes 1\otimes c \in A'\otimes_R B\}=\{c\in C: 1\otimes c\in B\}$. The reverse containment is clear, and the forward containment holds precisely because the map $B \to A' \otimes_R B$ is injective. This then implies $B=A\otimes_R D.$
\end{proof}

\subsection{Classification is Abelianization}\label{subsec::abelianization}
Lastly, we show that, in most cases of interest, the classification group of QCA is given by abelianizing the total QCA group. This certifies that the classification of QCA arises canonically, rather than as an ad hoc construction.
\begin{lemma}\label{lem::roots_k_commutator}
    Let $k$ be a field satisfying $(k^\times)^n=k^\times$ for all $n$ (e.g. algebraically closed fields). The group $\mathcal{C}(X; k)$ is equal to the commutator subgroup of $\mathcal{Q}(X; k)$.
\end{lemma}

\begin{proof}
    Since $k$ has $n$-th roots for all $n$, so $\operatorname{PGL}_n(k) = \operatorname{PSL}_n(k)$, and $\operatorname{PSL}_n(k)$ is simple and hence perfect for $n \geq 3$. The proposition above shows that $Q(X; k)/C(X; k)$ is abelian, so we have that $[Q(X; k), Q(X; k)] \subseteq C(X; k)$. As we will see later in the paper, $Q(X; k)$ may be constructed as $\operatorname{Aut}(S)$ of an appropriate symmetric monoidal category $S$, and hence its commutator subgroup is its maximal perfect subgroup. Thus, it suffices to show that $C(X; k)$ is perfect.\\

    Now, let $\bigotimes_j \alpha_j$ be a single-layer circuit, indexed over bounded $X_j$'s that form a disjoint partition of $X$. Up to stabilization, each $\alpha_j$ is without loss an automorphism of $\operatorname{Mat}(R^{q(X_j)})$ where $q(X_j) \geq 3$. Since $(k^\times)^n=k^\times$ for all $n$, $\operatorname{Aut}(\operatorname{Mat}(R^{q(X_j)})) \cong \operatorname{PGL}_{q(X_j)}(R) \cong \operatorname{PSL}_{q(X_j)}(R)$ is perfect. This means that we can write $\alpha_j = x_{j, 1} ... x_{j, n_j}$ where $x_{j, \bullet}$ is a commutator. Observe that if the $n_j$'s can be made to be uniformly bounded with respect to $j$, then we are done. Indeed, suppose we choose $N \geq n_j$ for all $j$, then without loss we can make $n_j = N$ for all $j$ by padding the identity element, which is also a commutator. Now for $1 \leq i \leq N$, write $x_{j, i} = z_{j, i}^{-1} y_{j, i}^{-1} z_{j, i} y_{j, i}$, we define
    \[\beta_i := \bigotimes_{j} x_{j, i} = (\bigotimes_{j} z_{j, i}^{-1}) (\bigotimes_{j} y_{j, i}^{-1}) (\bigotimes_{j} z_{j, i}) (\bigotimes_{j} y_{j, i})\]
    \[= (\bigotimes_{j} z_{j, i})^{-1} (\bigotimes_{j} y_{j, i})^{-1} (\bigotimes_{j} z_{j, i}) (\bigotimes_{j} y_{j, i}) \in [C(X; k), C(X; k)].\]
    Observe that $\bigotimes_j \alpha_j = \beta_1 ... \beta_N \in [C(X; k), C(X; k)]$. This shows that $C(X; k) \subseteq [C(X; k), C(X; k)]$, and hence $C(X; k)$ is perfect.\\
    
    As noted above, there is a subtlety that the $n_j$'s may not be uniformly bounded with respect to $j$. Fortunately, by Theorem 1 and 2 of \cite{2713f7c6-4827-3cbb-87c7-e44de1469d43} and the fact that $\operatorname{PSL}_n(k)$ is a quotient of $\operatorname{SL}_n(k)$, the $n_j$'s can be made to be uniformly bounded (in fact, $N$ can be chosen to be $2$).
\end{proof}

 The same proof above also shows the following.
 \begin{lemma}
     Let $k$ be a field. The group $C^{sp}(X; k)$ is equal to the commutator subgroup of $\mathcal{Q}(X; k)$.
 \end{lemma}

 The identifications above rely on the ring $R$ we are working with. Now we derive an identification that relies on the space rather than the ring.
 \begin{defn}\label{def::blendId}
     On metric spaces of the form $X \times \Zbb$, we say a QCA $(q, \al)$ \textit{blends into the identity} if there exists $\ell \in \ZZ$ such that $\alpha$ restricts to identity on either $\SA(X\times \ZZ_{\leq l}, q)$ or $\SA(X\times \ZZ_{\geq l}, q)$.
 \end{defn}
\begin{lemma} \label{lem::blending}
    Any $[q,\al]\in \CQ(X\times \ZZ)$ which blends into identity is in $[\CQ(X\times \ZZ), \CQ(X\times \ZZ)].$
\end{lemma}
 \begin{proof}
 Given any QCA $[q, \al]$, define its \textit{translate} $[Tq, T\alpha]$ by
    \begin{equation}
(Tq)(x, n+1) =  q(x, n) \quad \text{for all } n\in \ZZ,
\end{equation}
with a locality-preserving isomorphism \begin{equation}
    \tau: \SA(X\times \ZZ, q)\rightarrow \SA(X\times \ZZ, Tq)
\end{equation} that identifies $\SA(X\times \{n\}, q)$ with  $\SA(X\times \{n+1\}, Tq)$. Then
\begin{equation}
    T\al := \tau \al \tau^{-1}.
\end{equation}
By Proposition~\ref{prop::conjugation}, a QCA and its translate are equal in $\CQ(X\times \ZZ)^\mathrm{ab}$.

Now, given a QCA that restricts to identity on $\SA(X\times \ZZ_{\leq 0}, q)$, choose a representative $[q, \al]\in \CQ(X)$ for it that satisfies $q\equiv 1$ on $X\times \ZZ_{\leq 0}$. This is always possible because of the assumption. 
       
The infinite product
\begin{equation}
    q^S:=\prod_{m=0}^\infty T^m q
\end{equation}
is a well-defined spin system because $q^S$ takes a finite value on every point in $X\times \ZZ$. Therefore,
\begin{equation}
    \left[q^S, \al^S:=\bigotimes_{m=0}^\infty T^m\al\right]\in \CQ(X\times \ZZ).
\end{equation}
see Figure~\ref{fig:swindle} for an illustration of $q^S$ and $\alpha^S$. We deduce that in $\CQ(X\times \ZZ)^\mathrm{ab}$,
\begin{equation}
    [Tq^S, T\al^S]=\left[\prod_{m=1}^\infty T^m q, \bigotimes_{m=1}^\infty T^m\al\right]=[q^S, \al^S],
\end{equation} 
Moreover, it is easy to check that  
\begin{equation}
    [q^S, \al^S]=[q, \al] \cdot  [Tq^S, T\al^S].
\end{equation}
In conclusion, $[q, \al]$ is trivial in $\CQ(X\times \ZZ)^\mathrm{ab}$.
 \end{proof}
\begin{prop}\label{prop::all_circuit_same}
     On metric spaces of the form $X \times \Zbb$ and over any ring $R$, the four circuit subgroups defined are all equal to the commutator subgroup of $\mathcal{Q}(X \times \Zbb)$. 
 \end{prop}
 \begin{proof}
     It suffices to show that a single-layer circuit of any type is trivial in the abelianization $\mathcal{Q}(X \times \Zbb)^\mathrm{ab}.$ Clearly, a single-layer circuit $\al=\bigotimes_j\al_j$ is a product of two circuits that blend into the identity. The result follows from Lemma~\ref{lem::blending}.
 \end{proof}
Note that Proposition~\ref{prop::all_circuit_same} can alternatively follow from Lemma~\ref{lem::contractible} later, whose proof follows a similar theme here.

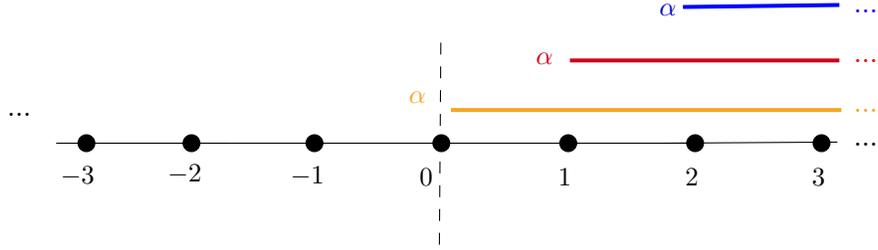
\begin{figure}[htb]
    \centering      
\[\begin{tikzpicture}[x=0.75pt,y=0.75pt,yscale=-1,xscale=1]

%Shape: Circle [id:dp9779630109483365] 
\draw  [fill={rgb, 255:red, 0; green, 0; blue, 0 }  ,fill opacity=1 ] (287,203.75) .. controls (287,201.4) and (288.9,199.5) .. (291.25,199.5) .. controls (293.6,199.5) and (295.5,201.4) .. (295.5,203.75) .. controls (295.5,206.1) and (293.6,208) .. (291.25,208) .. controls (288.9,208) and (287,206.1) .. (287,203.75) -- cycle ;
%Shape: Circle [id:dp35766382947007025] 
\draw  [fill={rgb, 255:red, 0; green, 0; blue, 0 }  ,fill opacity=1 ] (351,203.75) .. controls (351,201.4) and (352.9,199.5) .. (355.25,199.5) .. controls (357.6,199.5) and (359.5,201.4) .. (359.5,203.75) .. controls (359.5,206.1) and (357.6,208) .. (355.25,208) .. controls (352.9,208) and (351,206.1) .. (351,203.75) -- cycle ;
%Straight Lines [id:da18615588937954652] 
\draw    (291.25,203.75) -- (355.25,203.75) ;
%Shape: Circle [id:dp7422276404918946] 
\draw  [fill={rgb, 255:red, 0; green, 0; blue, 0 }  ,fill opacity=1 ] (415,203.75) .. controls (415,201.4) and (416.9,199.5) .. (419.25,199.5) .. controls (421.6,199.5) and (423.5,201.4) .. (423.5,203.75) .. controls (423.5,206.1) and (421.6,208) .. (419.25,208) .. controls (416.9,208) and (415,206.1) .. (415,203.75) -- cycle ;
%Straight Lines [id:da45408200161951917] 
\draw    (355.25,203.75) -- (419.25,203.75) ;
%Shape: Circle [id:dp07869951523705676] 
\draw  [fill={rgb, 255:red, 0; green, 0; blue, 0 }  ,fill opacity=1 ] (479,203.75) .. controls (479,201.4) and (480.9,199.5) .. (483.25,199.5) .. controls (485.6,199.5) and (487.5,201.4) .. (487.5,203.75) .. controls (487.5,206.1) and (485.6,208) .. (483.25,208) .. controls (480.9,208) and (479,206.1) .. (479,203.75) -- cycle ;
%Straight Lines [id:da05365117772653849] 
\draw    (419.25,203.75) -- (483.25,203.75) ;
%Shape: Circle [id:dp7100941260974538] 
\draw  [fill={rgb, 255:red, 0; green, 0; blue, 0 }  ,fill opacity=1 ] (225,203.75) .. controls (225,201.4) and (226.9,199.5) .. (229.25,199.5) .. controls (231.6,199.5) and (233.5,201.4) .. (233.5,203.75) .. controls (233.5,206.1) and (231.6,208) .. (229.25,208) .. controls (226.9,208) and (225,206.1) .. (225,203.75) -- cycle ;
%Straight Lines [id:da7586259108071153] 
\draw    (229.25,203.75) -- (293.25,203.75) ;
%Straight Lines [id:da09001136273515431] 
\draw    (183.25,203.75) -- (247.25,203.75) ;
%Straight Lines [id:da7886738129632346] 
\draw    (467.25,203.75) -- (555,203.2) ;
%Shape: Circle [id:dp30174563517256725] 
\draw  [fill={rgb, 255:red, 0; green, 0; blue, 0 }  ,fill opacity=1 ] (172,203.75) .. controls (172,201.4) and (173.9,199.5) .. (176.25,199.5) .. controls (178.6,199.5) and (180.5,201.4) .. (180.5,203.75) .. controls (180.5,206.1) and (178.6,208) .. (176.25,208) .. controls (173.9,208) and (172,206.1) .. (172,203.75) -- cycle ;
%Straight Lines [id:da42054131263252403] 
\draw    (161,203.75) -- (225,203.75) ;
%Shape: Circle [id:dp8760156956688893] 
\draw  [fill={rgb, 255:red, 0; green, 0; blue, 0 }  ,fill opacity=1 ] (522.75+20,203.75) .. controls (522.75+20,201.4) and (524.65+20,199.5) .. (527+20,199.5) .. controls (529.35+20,199.5) and (531.25+20,201.4) .. (531.25+20,203.75) .. controls (531.25+20,206.1) and (529.35+20,208) .. (527+20,208) .. controls (524.65+20,208) and (522.75+20,206.1) .. (522.75+20,203.75) -- cycle ;
%Straight Lines [id:da6504326648291378] 
\draw [color={rgb, 255:red, 245; green, 166; blue, 35 }  ,draw opacity=1 ][line width=1.5]    (360.09,187) -- (557,187) ;
%Straight Lines [id:da1830163971608787] 
\draw  [dash pattern={on 4.5pt off 4.5pt}]  (355,153) -- (354.25,256.75) ;
%Straight Lines [id:da6417488381792528] 
\draw [color={rgb, 255:red, 208; green, 2; blue, 27 }  ,draw opacity=1 ][line width=1.5]    (420.09,162) -- (556,162) ;

\draw [color={rgb, 255:red, 2; green, 2; blue, 255 }  ,draw opacity=1 ][line width=1.5]    (477.09,135) -- (556,134) ;

% Text Node
\draw (278,214.4) node [anchor=north west][inner sep=0.75pt]    {$-1$};
% Text Node
\draw (343,215.4) node [anchor=north west][inner sep=0.75pt]    {$0$};
% Text Node
\draw (413,215.4) node [anchor=north west][inner sep=0.75pt]    {$1$};
% Text Node
\draw (477,215.4) node [anchor=north west][inner sep=0.75pt]    {$2$};
% Text Node
\draw (216,213.4) node [anchor=north west][inner sep=0.75pt]    {$-2$};
% Text Node
\draw (135,187) node [anchor=north west][inner sep=0.75pt]  [font=\large]  {$...$};
% Text Node
\draw (562,202) node [anchor=north west][inner sep=0.75pt]  [font=\large]  {$...$};
% Text Node
\draw (162,214.4) node [anchor=north west][inner sep=0.75pt]    {$-3$};
% Text Node
\draw (541,215.4) node [anchor=north west][inner sep=0.75pt]    {$3$};
% Text Node
\draw (337.73,176.24) node [anchor=north west][inner sep=0.75pt]  [color={rgb, 255:red, 245; green, 166; blue, 35 }  ,opacity=1 ]  {$\alpha $};
% Text Node
\draw (562,160) node [anchor=north west][inner sep=0.75pt]  [font=\large]  {$\textcolor[rgb]{0.96,0,0}{...}$};
% Text Node
\draw (401.73,157) node [anchor=north west][inner sep=0.75pt]  [color={rgb, 255:red, 208; green, 2; blue, 27 }  ,opacity=1 ]  {$\alpha $};
% Text Node
\draw (562,185) node [anchor=north west][inner sep=0.75pt]  [font=\large]  {$\textcolor[rgb]{0.96,0.65,0.14}{...}$};

\draw (464,132) node [anchor=north west][inner sep=0.75pt]  [color={rgb, 255:red, 2; green, 2; blue, 208 }  ,opacity=1 ]  {$\alpha $};
% Text Node
\draw (562,135) node [anchor=north west][inner sep=0.75pt]  [font=\large]  {$\textcolor[rgb]{0.14,0,0.96}{...}$};

\end{tikzpicture}
\]
    \caption{An illustration of $q^S$ and $\alpha^S$ appearing in the proof of Lemma~\ref{lem::blending}. The figure illustrates the case where $X$ is a point.}
    \label{fig:swindle}
\end{figure}

\section{The infinite loop space of spin systems and QCA}\label{sec::space}

In this section, we will construct a topological enrichment of the total QCA group (Definition~\ref{def::qca_group}). As hinted in the previous section, we will use algebraic K-theory to build suitable spaces that encode the QCA group. The specific approach we will take is to apply Segal's construction for the algebraic K-theory of a symmetric monoidal category \cite{Segal1974}.

\subsection{Constructing the QCA Space}

Before defining our QCA space, we should first specify the symmetric monoidal category we are working with.
\begin{defn}
 The category of quantum spin systems is a category $\mathbf{C}(X)$ whose:
 \begin{itemize}
     \item Objects are $\mathcal{A}(X, q)$ for each quantum spin system $q: X \to \Nbb_{> 0}$. Equivalently, the objects may be viewed as just the quantum spin systems.
     \item Morphisms are locality-preserving isomorphisms (recall Definition~\ref{def:locality_preserve_iso}).
 \end{itemize}
 $\mathbf{C}(X)$ is symmetric monoidal under the pointwise stacking operation defined around Definition~\ref{def::pointwise_stacking}.
\end{defn}

\begin{exmp}\label{exp::matrix_one_pt}
    Let $X = *$ be a point, then $\mathbf{C}(*)$ is the category $\operatorname{Mat}_{R}$ of matrix algebras $\operatorname{Mat}(R^n)$, with morphisms as $R$-algebra automorphisms, symmetric monoidal under tensor product. Note that $\pi_0(B \operatorname{Mat}_{R}) = \Zbb_{>0}$ is the multiplicative monoid of positive integers, whose group completion is the group of positive rationals $\Qbb_{>0}$ under multiplication. When $R$ is an algebraically closed field such as $\CC$, this coincides with $\operatorname{Az}(R)$.
\end{exmp}

Since there is a local finiteness requirement to quantum spin systems, we see $\mathbf{C}(\Rbb^n)$ and $\mathbf{C}(\Zbb^n)$ are equivalent as symmetric monoidal categories (e.g. Example~\ref{exp::real}). In general, we can restrict to a coarsely dense subset $X' \subset X$ such that $\mathbf{C}(X)$ and $\mathbf{C}(X')$ are equivalent. Thus, for the rest of this section, we will \textit{assume $X$ is a locally finite countable metric space of bounded geometry}. This is a general theme that our construction depends only on the large-scale (coarse) geometry of $X$, see Appendix~\ref{sec::coarse} for more details.

Let $\mathcal{C}$ be a (small) symmetric monoidal category, its classifying space $B\mathcal{C}$ has the natural structure of a homotopy commutative, homotopy associative $H$-space. In particular, $\pi_0(B\mathcal{C})$ is a commutative monoid.
\begin{defn}\label{def::k_theory_symmetric}
     The \textit{K-theory space} $K(\mathcal{C})$ is defined as the \textit{group completion} $\iota: B\mathcal{C} \to K(\mathcal{C})$ of $B\mathcal{C}$, which is characterized by the following property
    \begin{enumerate}
        \item $K(\mathcal{C})$ is a group-like H-space. In particular, $\pi_0(K(\mathcal{C}))$ is a group.
        \item The induced map $\iota_*: \pi_0(B\mathcal{C}) \to \pi_0(K(\mathcal{C}))$ is the algebraic group completion of the commutative monoid $\pi_0(B\mathcal{C})$.
        \item The homology of $H_*(K(\mathcal{C}); R)$ is the localization of $H_*(B\mathcal{C}; R)$ at $\pi_0(B\mathcal{C})$ for any commutative ring $R$. Here, the homologies are rings because the relevant spaces are H-spaces.
    \end{enumerate}
    The $i$-th $K$-theory of $\mathcal{C}$ is denoted $K_i(\mathcal{C}) \coloneqq \pi_i(K(\mathcal{C}))$. The construction is functorial with respect to lax symmetric monoidal functors.
\end{defn}

A group completion of $\mathcal{C}$ always exist (see Appendix of \cite{Thomason01011982}). When $\mathcal{C}$ satisfies additional properties, \cite{quillen_2} gave an explicit description for $K(\mathcal{C})$.
\begin{defn}[Definition IV.4.2 of \cite{weibel2013k}]
Given a symmetric monoidal category $\mathcal{S}$, the category $\mathcal{S}^{-1} \mathcal{S}$ has objects as $(m, n) \in \operatorname{obj}(\mathcal{S}) \times \operatorname{obj}(\mathcal{S})$, and a morphism $(m_1, n_1) \to (m_2, n_2)$ is given by an equivalence class of morphisms of the form
\[(m_1, n_1) \xrightarrow{s \otimes -} (s \otimes m_1, s \otimes n_1) \xrightarrow{(f, g)} (m_2, n_2). \]
This morphism is equivalent to 
\[(m_1, n_1) \xrightarrow{t \otimes -} (t \otimes m_1, t \otimes n_1) \xrightarrow{(f', g')} (m_2, n_2) \]
if there is an isomorphism $\alpha: s \to t$ such that $f = f' \circ (\alpha \otimes id_{m_1})$ and $g = g' \circ (\alpha \otimes id_{m_2})$. $\mathcal{S}^{-1} \mathcal{S}$ is symmetric monoidal under pointwise tensor product.
\end{defn}
\begin{theorem}[\cite{quillen_2}]\label{thm::s_inv_s}
For any symmetric monoidal category $\mathcal{S}$ that is (1) a groupoid and (2) tensoring by identity $\operatorname{Aut}(s) \to \operatorname{Aut}(s \otimes t)$ is an injection (e.g., $\mathcal{S} = \mathbf{C}(X)$), the natural map $B\mathcal{S} \to B(\mathcal{S}^{-1} \mathcal{S})$ is the group completion of $B\mathcal{S}$.
\end{theorem}

The space $K(\mathcal{C})$ is the zeroth space of an $\Omega$-spectrum and is an infinite loop space. Historically, there were two ways to show that $K(\mathcal{C})$ is an infinite loop space - one is using the method of operads \cite{may_infinite_loop_space}, and the other is using the method of $\Gamma$-spaces \cite{Segal1974}. The two approaches were shown to be equivalent in \cite{MAY1978205}. In fact, \cite{MAY1978205} showed that all methods to construct infinite loop space structures that follow certain axioms are equivalent. For our purposes, we just note that such a structure exists on the group completion.

We will see later that the QCA group actually corresponds more closely to information at the level of $\pi_1$. This motivates\footnote{This was suggested by Tomer Schlank.} us to apply a loop space on top of the K-theory space, at the component corresponding to the identity element in $\pi_0$. Finally, we define the QCA space as follows.

\begin{defn}
 The \textit{space of QCA} (or the \textit{QCA space}) over $X$ is defined as
 \[\mathbf{Q}(X) \coloneqq \Omega K(\mathbf{C}(X)).\]
 We write $\mathbf{Q}_i(X) \coloneqq \pi_i \mathbf{Q}(X)$. Note that we have $\mathbf{Q}_i(X) = K_{i+1}(\mathbf{C}(X))$ for all $i \geq 0$.
\end{defn}
Although we posed restrictions on $X$ in this section, we remark again that this space of QCA is well-defined on a general metric space.
\begin{remark}
    For a ring $R$, its algebraic K-theory space is defined over a skeletal category. In our context, $\mathbf{C}(X)$ is usually not skeletal, unless $X$ is a single point.
\end{remark}

In the remainder of this section, we will discuss some properties of the space $K(\mathbf{C}(X))$, in particular on its $\pi_1$ and $\pi_0$ groups. Note that when moving to $\mathbf{Q}(X)$, we actually discard the information of $K_0(\mathbf{C}(X))$, but we included a discussion here due to an interesting connection to coarse homology groups.

\subsection{$\pi_1$ and Plus Constructions}

The first model for higher K-theory of rings was given by Quillen \cite{quillen_plus} with what is called the plus construction. In this section, we will derive a plus-construction model for $K(\mathbf{C}(X))$.
\begin{defn}
Let $X$ be a path-connected space and $N \trianglelefteq \pi_1(X)$ be a perfect normal subgroup. The plus construction of $X$ with respect to $N$ is the data $(X^+, i: X \to X^+)$ satisfying the properties that:
\begin{enumerate}
    \item $N = \ker(i_*: \pi_1(X) \to \pi_1(X^+))$.
    \item $i$ is an acyclic map, meaning its homotopy fiber has reduced integral homologies all equal to $0$.
    \item For any map $f: X \to Y$ such that $N \subseteq \ker(f_*: \pi_1(X) \to \pi_1(Y))$, there exists a map $f^+: X^+ \to Y$, unique up to based homotopy, such that $f^+ \circ i \simeq f$ up to based homotopy.
\end{enumerate}
If no $N$ is specified, then by convention we take $X^+$ with respect to the maximal perfect normal subgroup.
\end{defn}
Note that these conditions also implies that $i_*: \pi_1(X) \to \pi_1(X^+)$ is surjective, it then follows that $\pi_1(X^+) = \pi_1(X)/N$. If $X$ is CW, this plus construction always exists by attaching a sequence of cells. The plus-construction $X \to X^+$ at $N$ is initial among all maps $g: X \to Y$ such that $N \subseteq \ker(g_*: \pi_1(X) \to \pi_1(Y))$.

\begin{exmp}
    Let $\operatorname{GL}(R) = \operatorname{colim}_{i \geq 0} \operatorname{GL}_i(R)$ where $\operatorname{GL}_{i}(R) \to \operatorname{GL}_{i+1}(R)$ is given by sending to the block diagonal matrix $A \mapsto \begin{pmatrix}
        A & 0\\
        0 & 1
    \end{pmatrix}$. A lemma of Whitehead \cite{6d7f372b-468f-38fc-b338-2d803947c5dc} shows that the commutator subgroup of $\operatorname{GL}(R)$ is the maximal perfect normal subgroup (see also Chapter III.1 of \cite{weibel2013k}). For $i > 0$, the $i$-th K-theory of $R$ can be defined as\label{exp::algebraic_K_theory}
    \[K_i(R) \coloneqq \pi_i(\operatorname{BGL}(R)^+),\]
    where $\operatorname{BGL}(R)^+$ is the plus construction at the commutator subgroup. One defines the algebraic K-theory space of $R$ as
    \[K(R) = K_0(R) \times \operatorname{BGL}(R)^+\]
    where $K_0(R)$ is given the discrete topology. It was later shown in \cite{quillen_2} that $K(R)$ is equivalent Segal's algebraic K-theory of $R$, which is constructed on a category $\operatorname{Proj}(R)^{f.g., \cong}$. Here, $\operatorname{Proj}(R)^{f.g., \cong}$ denotes the category whose objects are isomorphism classes of finitely generated projective R-modules, morphisms are $R$-linear automorphisms, and is symmetric monoidal under direct sum.
\end{exmp}

The characterization of some Segal K-theory spaces as similar plus construction spaces has been discussed in \cite{951cf383-5a4d-3775-b8b2-0bafd2f162fe}, \cite{c80f3555-900f-35b5-ab05-6d397e6a44e6}, and Chapter VII.2 of \cite{bass1968algebraic}. The plus constructions are done on the classifying space of what is called the \textit{automorphism group} of $\mathcal{C}$. For our purposes, we take the following definition.
\begin{defn}
    Let $\mathcal{C}$ be a symmetric monoidal groupoid and symmetric monoidal operation $\otimes$. The \textbf{automorphism group} of $\mathcal{C}$ is
    \[\operatorname{Aut}(\mathcal{C}) \coloneqq \operatorname{colim}_{x \in \operatorname{obj}(\mathcal{C})} \operatorname{Aut}_{\mathcal{C}}(x),\]
    where $x, y \in \operatorname{obj}(\mathcal{C})$ has a morphism $x \to y$ if $y = x \otimes z$, and the corresponding map $\operatorname{Aut}_{\mathcal{C}}(x) \to \operatorname{Aut}_{\mathcal{C}}(y)$ is given by $- \otimes \operatorname{id}_{z}$.
\end{defn}

For example, $\operatorname{GL}(R) = \operatorname{Aut}(\operatorname{Proj}(R)^{f.g., \cong})$ and $\mathcal{Q}(X) = \operatorname{Aut}(\mathbf{C}(X))$.

We now obtain a plus construction description for $K(\mathbf{C}(X))$ in terms of $\Aut(\mathbf{C}(X)) = \mathcal{Q}(X)$. We say a path-connected space $X$ is \textit{weakly simple} if $\pi_1(X)$ acts trivially on the homology of its universal cover. Note that path-connected $H$-spaces are weakly simple.

\plus*

\begin{proof}
We first show that $[\mathcal{Q}(X), \mathcal{Q}(X)]$ is perfect. For a symmetric monoidal category $\mathcal{S}$, Lemma 1.8 on page 351 of \cite{bass1968algebraic} states that for any $\alpha_1, ..., \alpha_n \in \operatorname{Aut}_{\mathcal{S}}(x)$ such that $\alpha_n ... \alpha_1 = \operatorname{id}_{x}$, the element $\alpha_1 \otimes ... \otimes \alpha_n \in \operatorname{Aut}_{\mathcal{S}}(x^{\otimes n})$ is a commutator. We apply this to $\mathcal{S} = \mathbf{C}(X)$, following the arguments in Proposition 3 of \cite{c80f3555-900f-35b5-ab05-6d397e6a44e6}.\\

Clearly $[\mathcal{Q}(X), \mathcal{Q}(X)]$ Since $\mathcal{Q}(X)$ is a filtered colimit of the $\mathcal{Q}(X, q)$, and taking commutator subgroups commute with filtered colimits, we have that $[\mathcal{Q}(X), \mathcal{Q}(X)]$ is the filtered colimit of $[\mathcal{Q}(X, q), \mathcal{Q}(X, q)]$. Thus each element of $[\mathcal{Q}(X), \mathcal{Q}(X)]$ can be written as a finite product of commutators $[\alpha_i, \beta_i] \in \mathcal{Q}(X, r) = \operatorname{Aut}_{\mathbf{C}(X)}(\mathcal{A}(X, r))$ for some $r$. Now observe that $[\alpha_i, \beta_i] \otimes \operatorname{id} \otimes \operatorname{id} = [\alpha_i \otimes \alpha^{-1}_i \otimes \operatorname{id}, \beta_i \otimes \operatorname{id} \otimes \beta^{-1}_i]$ in $\mathcal{Q}(X, r^3)$. Now the aforementioned Lemma 1.8 on page 351 of \cite{bass1968algebraic} shows that both $\alpha_i \otimes \alpha^{-1}_i \otimes \operatorname{id}$ and $\beta_i \otimes \operatorname{id} \otimes \beta^{-1}_i$ are commutators themselves. Thus, the element $[\alpha_i, \beta_i]$ represents in $E = [\mathcal{Q}(X), \mathcal{Q}(X)]$ is an element in $[E, E]$. This shows that $E \subseteq [E, E]$. Since $[E, E] \subseteq E$, we conclude that $E = [E, E]$ is perfect.\\

To show $B\mathcal{Q}(X)^+$ is weakly simple, we use the following criterion in Section 1 of \cite{WAGONER1972349}: for a group $G$ with perfect $[G, G]$ such that for any $g_1,..., g_n \in G$ and $g \in G$, there exists some $h \in [G, G]$ with $g g_i g^{-1} = h g_i h^{-1}$, $BG^+$ is weakly simple.\\

In our case, we take $\alpha_1, ..., \alpha_n \in \mathcal{Q}(X)$ and $g \in \mathcal{Q}(X)$. Without loss we can take representatives $\alpha_1, ..., \alpha_n, g \in \mathcal{Q}(X, q)$. We can consider the elements
\[\alpha_1 \otimes id_{\mathcal{A}(X, q)}, ..., \alpha_n \otimes id_{\mathcal{A}(X, q)} \text{ and } g \otimes id_{\mathcal{A}(X, q)} \in \mathcal{A}(X, q^2).\]
Now we define
\[h = (g \otimes id_{\mathcal{A}(X, q)}) (\operatorname{id}_{\mathcal{A}(X, q)} \otimes g )^{-1}.\]
Observe that $\operatorname{id}_{\mathcal{A}(X, q)} \otimes g$ commutes with $\alpha_i \otimes id_{\mathcal{A}(X, q)}$ for all $i$. Thus, we have that
\[h (\alpha_i \otimes id_{\mathcal{A}(X, q)}) h^{-1} = (g \otimes id_{\mathcal{A}(X, q)}) (\operatorname{id}_{\mathcal{A}(X, q)} \otimes g )^{-1} (\operatorname{id}_{\mathcal{A}(X, q)} \otimes g ) (\alpha_i \otimes id_{\mathcal{A}(X, q)}) (g \otimes id_{\mathcal{A}(X, q)})^{-1} \]
\[= (g \otimes id_{\mathcal{A}(X, q)}) (\alpha_i \otimes id_{\mathcal{A}(X, q)}) (g \otimes id_{\mathcal{A}(X, q)})^{-1} = (g \alpha_i g^{-1}) \otimes id_{\mathcal{A}(X, q)}. \]
To show that $h$ is in the commutator, we note that
\[g \otimes id_{\mathcal{A}(X, q)} = \operatorname{SWAP}_{q,q} \circ (id_{\mathcal{A}(X, q)} \otimes g) \circ \operatorname{SWAP}_{q,q},\]
where $\operatorname{SWAP}_{q,q}$ is its own inverse and is stably a circuit, as noted in Example~\ref{example::swap}. Multiplying both sides by $(id_{\mathcal{A}(X, q)} \otimes g)^{-1}$ then shows that $h$ is a commutator and hence a circuit. This shows that $B\mathcal{Q}(X)^+$ is a weakly simple space.\\

We now use the telescope construction in \cite{quillen_2}. Write $\mathcal{S} = \mathbf{C}(X)$, Theorem~\ref{thm::s_inv_s} gives a category $\mathcal{S}^{-1} \mathcal{S}$ such that $K(\mathbf{C}(X)) \simeq B(\mathcal{S}^{-1} \mathcal{S})$. For each $\mathcal{A}(X, q)$, we define a functor $F_q$ from the groupoid $\mathcal{Q}(X, q)$ to $\mathcal{S}^{-1} \mathcal{S}$ by sending the unique object to $(\mathcal{A}(X, q), \mathcal{A}(X, q))$ and a morphism $u: \mathcal{A}(X, q) \to \mathcal{A}(X, q)$ to 
\[(\mathcal{A}(X, q), \mathcal{A}(X, q)) \xrightarrow{1 \otimes - } (\mathcal{A}(X, q), \mathcal{A}(X, q)) \xrightarrow{(\operatorname{id}, u)} (\mathcal{A}(X, q), \mathcal{A}(X, q)).\]
For another quantum spin system $r: X \to \Zbb_{> 0}$, we have the following diagram:
% https://q.uiver.app/#q=WzAsNCxbMCwwLCJcXG1hdGhjYWx7UX0oWCwgcSkiXSxbMiwwLCJcXG1hdGhjYWx7UX0oWCwgcXIpIl0sWzAsMSwiXFxtYXRoY2Fse1N9XnstMX0gXFxtYXRoY2Fse1N9Il0sWzIsMSwiXFxtYXRoY2Fse1N9XnstMX0gXFxtYXRoY2Fse1N9Il0sWzAsMSwiIC0gXFxvdGltZXMgXFxvcGVyYXRvcm5hbWV7aWR9X3tcXG1hdGhjYWx7QX0oWCwgcil9Il0sWzAsMiwiRl9xIiwyXSxbMSwzLCJGX3IiXSxbMiwzLCJIKHIpIl1d
\[\begin{tikzcd}
	{\mathcal{Q}(X, q)} && {\mathcal{Q}(X, qr)} \\
	{\mathcal{S}^{-1} \mathcal{S}} && {\mathcal{S}^{-1} \mathcal{S}}
	\arrow["{ - \otimes \operatorname{id}_{\mathcal{A}(X, r)}}", from=1-1, to=1-3]
	\arrow["{F_q}"', from=1-1, to=2-1]
	\arrow["{F_{qr}}", from=1-3, to=2-3]
	\arrow["{H(r)}", from=2-1, to=2-3]
\end{tikzcd},\]
where $H(r): \mathcal{S}^{-1} \mathcal{S} \to \mathcal{S}^{-1} \mathcal{S}$ sends $(\mathcal{A}(X, q_1), \mathcal{A}(X, q_2))$ to $(\mathcal{A}(X, q_1 r),$ $ \mathcal{A}(X, q_2 r))$ and tensor the corresponding morphisms by $\operatorname{id}_{\mathcal{A}(X, r)}$. Note that restricting the diagram above to the images of $F_q$ and $F_r$ in $\mathcal{S}^{-1} \mathcal{S}$ actually yields a natural transformation $F_q \implies F_{qr}$. This implies that the induced map $BF_q: BQ(X, q) \to B(\mathcal{S}^{-1} \mathcal{S})$ and $B(F_{qr} \circ - \otimes_{id}{\mathcal{A}(X, r)}): B\mathcal{Q}(X, q) \to B\mathcal{Q}(X, qr) \to B(\mathcal{S}^{-1} \mathcal{S})$ are homotopic. Furthermore, they land in the component at identity as there is a clear map from $(1, 1) \to (\mathcal{A}(X, q), \mathcal{A}(X, q))$ for any $q$. This gives a homotopy commutative diagram
% https://q.uiver.app/#q=WzAsMyxbMCwwLCJCXFxtYXRoY2Fse1F9KFgsIHEpIl0sWzEsMCwiQlxcbWF0aGNhbHtRfShYLCBxcikiXSxbMSwxLCJCKFxcbWF0aGNhbHtTfV57LTF9IFxcbWF0aGNhbHtTfSkiXSxbMCwxXSxbMSwyXSxbMCwyXV0=
\[\begin{tikzcd}
	{B\mathcal{Q}(X, q)} & {B\mathcal{Q}(X, qr)} \\
	& {B(\mathcal{S}^{-1} \mathcal{S})_{(1, 1)}}
	\arrow[from=1-1, to=1-2]
	\arrow[from=1-1, to=2-2]
	\arrow[from=1-2, to=2-2]
\end{tikzcd}.\]
Taking the homotopy colimit yields a map
\[\phi: B\mathcal{Q}(X) = B(\operatorname{colim}_q \mathcal{Q}(X, q)) = \operatorname{hocolim}_q (B\mathcal{Q}(X, q)) \to B(\mathcal{S}^{-1} \mathcal{S})_{(1,1)}.\]
Note here we used the fact that classifying space commutes with filtered colimits. Since $B\mathcal{S} \to B(\mathcal{S}^{-1} \mathcal{S})$ is the group-completion map, we have that
\[H_*(B(\mathcal{S}^{-1} \mathcal{S})) = (\pi_0 \mathcal{S})^{-1} H_*(B\mathcal{S}) = \operatorname{colim}_{q \in \pi_0(\mathcal{S})} H_*(B\mathcal{S}),\]
The identity component of $B(\mathcal{S}^{-1} \mathcal{S})$ is obtained by considering the trajectory of $(B\mathcal{S})_1$ under the colimit. Thus, when we restricting to the identity, we have that
\[H_*(B(\mathcal{S}^{-1} \mathcal{S})_{(1,1)}) = \operatorname{colim}_{[q] \in \pi_0(\mathcal{S})} H_*(B \mathcal{Q}(X, q)).\]
We remark here that the colimit here does not commute with $H_*(-)$ because a colimit of $\mathcal{Q}(X, q)$ over $[q] \in \pi_0(\mathcal{S})$ may not be well-defined due to ambiguities in the choice of the representatives. However, the choices are all equivalent up to conjugation, so a colimit of their homologies does make sense. On the other hand, the colimit $\mathcal{Q}(X) = \operatorname{colim}_{q \in \operatorname{obj}(\mathcal{S})} B\mathcal{Q}(X, q)$ is well-defined and its homology is the colimit of the diagram:
% https://q.uiver.app/#q=WzAsMyxbMCwwLCJcXG9wZXJhdG9ybmFtZXtvYmp9KFxcbWF0aGNhbHtTfSkiXSxbMSwwLCJcXHBpXzAoXFxtYXRoY2Fse1N9KSJdLFszLDAsIlxcb3BlcmF0b3JuYW1le0FifSJdLFswLDEsIlxccGk6cSBcXG1hcHN0byBbcV0iLDJdLFsxLDIsIkY6W3FdIFxcdG8gSF8qKFxcbWF0aGNhbHtRfShYLCBxKSkiLDIseyJsYWJlbF9wb3NpdGlvbiI6ODB9XSxbMCwyLCJxIFxcdG8gSF8qKFxcbWF0aGNhbHtRfShYLCBxKSkiLDAseyJjdXJ2ZSI6LTJ9XV0=
\[\begin{tikzcd}
	{\operatorname{obj}(\mathcal{S})} & {\pi_0(\mathcal{S})} && {\operatorname{Ab}}
	\arrow["{\pi:q \mapsto [q]}"', from=1-1, to=1-2]
	\arrow["{q \to H_*(\mathcal{Q}(X, q))}", curve={height=-12pt}, from=1-1, to=1-4]
	\arrow["{F:[q] \to H_*(\mathcal{Q}(X, q))}"'{pos=0.8}, from=1-2, to=1-4]
\end{tikzcd}.\]
The map $\phi$ realizes the natural map $$ H_*(\mathcal{Q}(X)) = \operatorname{colim}_{q \in \operatorname{obj}(\mathcal{S})} H_*(\mathcal{Q}(X, q)) = \operatorname{colim}(F \circ \pi) \to \operatorname{colim}(F) = H_*(\mathcal{B}(\mathcal{S}^{-1} \mathcal{S})_{(1,1)}),$$
which is an isomorphism as the functor $\pi: \operatorname{obj}(\mathcal{S}) \to \pi_0(\mathcal{S})$ is a cofinal functor (see Definition 4.17.1 of \cite{stacks-project}). To  show $\pi$ is cofinal, for any $[q] \in \pi_0(\mathcal{S})$, we need to check that the set $S$ of pairs $(x \in \operatorname{obj}(\mathcal{S}), b: [q] \to \pi(x))$ is non-empty, and every element in $S$ is equivalent to one another under the equivalence relation generated by $(x, b) \sim (x', b')$ if there is a morphism $a: x \to x'$ such that $b' = H(a) \circ b$. Clearly $S$ is non-empty by choosing $x = q$. For any $(x, b) \in \mathcal{S}$, we can without loss write think of $b$ as the map $- \otimes b$. Then for any $(x, b), (x', b') \in \mathcal{S}$, we have that $(x, b) \sim (x \otimes x', b \otimes x')$ and $(x', b') \sim (x' \otimes x = x \otimes x', b' \otimes x)$. It suffices to check that $b \otimes x'$ and $b' \otimes x$ are in the same isomorphism class to conclude that $(x, b) \sim (x', b')$. This is true in our case because $\otimes$ in $\mathcal{S}$ is cancellative.

Thus, $\phi$ is an integral homology isomorphism. $B(\mathcal{S}^{-1} \mathcal{S})$ is weakly simple since it is  an infinite loop space, so $B(\mathcal{S}^{-1} \mathcal{S})_{(1,1)}$ is also weakly simple. The map $\phi$ induces both a homology and $\pi_1$-isomorphism from $B\mathcal{Q}(X)^+ \to B(\mathcal{S}^{-1} \mathcal{S})_{(1,1)}$ between weakly simple spaces, which is an equivalence by Lemma 1.1 of \cite{WAGONER1972349}.
\end{proof}

Applying Theorem~\ref{thm::k0_qz}, Lemma~\ref{lem::roots_k_commutator}, Proposition~\ref{prop::all_circuit_same}, and (optionally) Theorem~\ref{thm::plus_construction} with Proposition IV.1.7 of \cite{weibel2013k}, we have another corollary.

\Qzero*

The corollary justifies calling $\mathbf{Q}(X)$ the space of QCA, because in many reasonable cases $\mathbf{Q}_0(X)$ is the classification of quantum cellular automata up to quantum circuits and stabilization. Recall also that any $K(\mathcal{C})$ is the zeroth space of an $\Omega$-spectrum, so $\textbf{Q}(X)$ is also part of an $\Omega$-spectrum. If one is only interested in producing an $\Omega$-spectrum that lifts $\mathcal{Q}(\Zbb^n)/\mathcal{C}(\Zbb^n)$, the construction here appears to be a natural way to lift this.

\begin{remark}
    The plus construction description is not necessary in obtaining Corollary~\ref{cor::qca_space_properties}(1) or (4). Corollary IV.4.8.1 of \cite{weibel2013k} shows that $K_1$ and $K_2$ of any $\mathcal{S}$ from Theorem~\ref{thm::s_inv_s} is same as Bass's definition of $K_1, K_2$ with $K_1(\mathcal{S}) \coloneqq \operatorname{colim}_{x \in \operatorname{obj}(\mathcal{S})} (\operatorname{Aut}_{\mathcal{S}}(x))^{ab}$ and $K_2(\mathcal{S}) \coloneqq \operatorname{colim}_{x \in \operatorname{obj}(\mathcal{S})} H_2([\operatorname{Aut}_{\mathcal{S}}(x), \operatorname{Aut}_{\mathcal{S}}(x)]; \Zbb)$. For $\mathcal{S} = \textbf{C}(X)$: since abelianization is left adjoint, commutator commutes with filtered colimits, we have $K_1(\mathbf{C}(X)) = (\mathcal{Q}(X))^{ab}$ and $K_2(\mathbf{C}(X)) = H_2(\mathcal{Q}(X); \Zbb)$..
\end{remark}

\subsection{$\pi_0$ and Coarse Homology Groups}

In the previous subsection, we obtained a description that
$$K(\textbf{C}(X)) \simeq K_0(\textbf{C}(X)) \times B\CQ(X)^+.$$
Here we would like to compute the term $K_0(\textbf{C}(X))$.\\

To understand the structure of $K_0(\textbf{C}(X))$, we should first understand the commutative monoid $\pi_0(B\textbf{C}(X))$, which involves understanding the isomorphism classes of objects in $\textbf{C}(X)$. This entails us to think about when two algebras of observables are isomorphic to each other. We now break the sad news that the pictures in Figure~\ref{fig::qca_example} and Figure~\ref{fig::surjectivity} may be somewhat misleading as to what QCA (and also locality-preserving isomorphisms) generally look like. In both figures, the map $\alpha: \mathcal{A}(X, q) \to \mathcal{A}(X, r)$ sends a tensor factor $T$ at a point $x \in X$ to a tensor factor at a point $y \in Y$. However, $\alpha(T)$, for a general $\alpha$, may not respect the pointwise decomposition of tensor factors in $\mathcal{A}(X, q)$ and thus can have support over multiple points. Fortunately, $\alpha$ having finite spread implies $\alpha(T)$ can only be supported at finitely many points, but it may very well be supported at more than 1 point.

We refer to the type of morphisms that appear in Figure~\ref{fig::qca_example} and Figure~\ref{fig::surjectivity} as \textit{shifts}.
\begin{defn}
    We say a locality-preserving isomorphism $\alpha: \mathcal{A}(X, q) \to \mathcal{A}(X, r)$ is a \textit{shift} if $\alpha$ breaks, moves, and merges prime-degree matrix tensor factors pointwise. That is, for $B = \operatorname{Mat}(R^{p})$ a tensor factor of $\operatorname{Mat}(R^{q_x}) \subseteq \mathcal{A}(X, q)$, where $p$ is a prime dividing $q_x$, $\alpha(B)$ is a tensor factor of $\operatorname{Mat}(R^{r_y})$ for some $y \in X$.
\end{defn}

The next proposition shows us that if a locality-preserving isomorphism $\alpha: \mathcal{A}(X, q) \to \mathcal{A}(X, r)$ exists, we can always find another isomorphism $\beta: \mathcal{A}(X, q) \to \mathcal{A}(X, r)$ that is a shift. Thus, we only need to be concerned about shifts when calculating $\pi_0(B\textbf{C}(X))$.
\begin{prop}\label{prop::replace_shift}
     Let $R$ be a ring with no non-trivial idempotent elements. Let $\alpha: \mathcal{A}(X, q) \to \mathcal{A}(X, r)$ be a locality-preserving isomorphism, then there exists an automorphism $f$ of $\mathcal{A}(X, r)$ such that $f \circ \alpha$ becomes a shift.
\end{prop}

\begin{exmp}
If $q = r$, we can take $f = \alpha^{-1}$ and $f \circ \alpha = \operatorname{id}$ is a shift.    
\end{exmp}

We defer the proof of Proposition~\ref{prop::replace_shift} to the end of this subsection. For now, we will see how the proposition enables us to connect to coarse homology groups. We only discuss the case for the $0$-th coarse homology group here, and Appendix~\ref{sec::coarse} gives a brief account of coarse homology theory in general. 
\begin{defn}\label{def::coarse_diagonal}
    Let $N$ be a commutative monoid. We define the 0-th and 1-st \textit{coarse chains} $CC_{i}(X, N), i = 0, 1$ as
    \begin{itemize}
        \item $CC_{0}(X, N)$ is the collection of locally finite formal $N$-combintations of points in $X$.
        \item $CC_1(X, N)$ consists of formal sums $\sum n_{(x_0, x_1)} [(x_0, x_1)]$ where $n_{(x_0, x_1)} \in N$ subject to two conditions, 
\begin{itemize}
    \item locally finite: each point $x \in X$ appears as a component only in finitely many terms, 
    \item bounded: there exists a uniform $l > 0$ so that every term $(x_0, x_1)$ with $n_{(x_0, x_1)} \neq 0$ is contained in the $l$-neighborhood of the diagonal in $X^{2}$. Here $X^2$ is equipped with the $L^{\infty}$-metric.
\end{itemize} 
    \end{itemize}
For an abelian group $A$, it is possible to define a boundary homomorphism $\partial: CC_{1}(X, A) \to CC_{0}(X, A)$ in the usual sense on summands and extend locally-finitely. The \textit{0-th coarse homology group} is defined as
\[CH_{0}(X, A) \cong CC_0(X, A)/\operatorname{im}(\partial).\]
We say that two $0$-coarse-chains $a, b$ are \textit{$l$-homologous} if $a-b$ is a boundary of a $1$-coarse-chain with bound $l$.
\end{defn}

Given an algebra of observables $\mathcal{A}(X, q)$, we observe there is a natural way to realize it as a 0-coarse-chain.
\begin{defn}
    Given a quantum spin system $q$ and prime number $p$, the $p$-degree of $q$ denoted by $\deg_p(q)$ is the formal sum 
    \[\sum_{x\in X} n_x [x], \ \text{where } p^{n_x}| q_x  \text{ and } p^{n_x+1}\nmid  q_x. \]
Note that $n_x$ is also sometimes written as $v_p(q_x)$. The $p$-degree takes value in $CC_0(X, \NN)\subset CC_0(X, \ZZ)$ of the metric space $X$. 

Collectively over all primes, we obtain an element $\operatorname{deg}(q)$ given by
\[\operatorname{deg}(q) = \sum_{x \in X} (v_2(q_x), v_3(q_x), v_5(q_x), v_7(q_x), ...) [x] \in CC_0(X, \NN^{\oplus\omega}) \subset CC_0(X, \Zbb^{\oplus \omega}).\]
Here $(\bullet)^{\oplus \omega}$ denotes the direct sum of countably infinite many copies of the same object. Furthermore since $X$ is assumed to be locally finite, $\deg$ provides a natural identification of $\NN^X_\mathrm{lf}$ with $CC_0(X, \NN^{\oplus \omega})$.
\end{defn}

\begin{lemma}\label{lem::gp_cc0}
    The group completion of $CC_0(M, \Nbb^{\oplus \omega})$ is $CC_0(M, \Zbb^{\oplus \omega})$ with natural inclusion $\iota$.
\end{lemma}

\begin{proof}
    We prove this via universal property. Indeed, suppose $\phi: CC_0(X, \Nbb^{\oplus \omega}) \to A$ is a map of commutative monoids. For any sum $\sum_{x \in X} a_x [x] \in CC_0(M, \Zbb^{\oplus \omega})$, we decompose
    \[\sum_{x \in X} a_x [x] = \sum_{x \in X} (b_x - c_x) [x] = \sum_{x \in X} b_x [x] - \sum_{x \in X} c_x [x],\quad b_x, c_x \in \Nbb^{\oplus \omega},\]
   where $b_x$ is the truncation of $a_x$ with the positive entries and $c_x$ is the truncation of $a_x$ with the negative entries and inverted. We define $\phi': CC_0(M, \Zbb^{\oplus \omega}) \to A$ by $\phi'(\sum_{x \in X} a_x [x]) = \phi(\sum_{x \in X} b_x [x]) - \phi(\sum_{x \in X} c_x [x])$. Clearly we have that $\phi' \circ \iota = \phi$. Clearly, this is uniquely defined as well. Indeed for any $\psi$ such that $\psi \circ \iota = \phi$, we have that
   \begin{align*}
       \psi(\sum_{x \in X} a_x [x]) &= \psi(\sum_{x \in X} b_x [x] - \sum_{x \in X} c_x [x])\\
       &= \psi(\sum_{x \in X} b_x [x]) - \psi(\sum_{x \in X} c_x [x])\\
       &= \phi(\sum_{x \in X} b_x [x]) - \phi(\sum_{x \in X} c_x [x])\\
       &= \phi'(\sum_{x \in X} a_x [x]),
   \end{align*}
   where the second last line follows from $b_x, c_x \in \Nbb^{\oplus \omega}$.
\end{proof}

Let $\mathcal R$ denote the equivalence relations on $\NN^X_\mathrm{lf} = CC_0(X, \Nbb^{\oplus \omega})$ by locality-preserving isomorphisms, so $CC_0(X, \Nbb^{\oplus \omega})/\mathcal R = \pi_0(B\textbf{C}(X))$. We wish to now compute the group completion of $CC_0(X, \Nbb^{\oplus \omega})/\mathcal R$. To do this, we will use the following general lemma whose proof can be found in Appendix~\ref{sec::group_complete}.
\begin{lemma}\label{lem::gp_quotient_commute}
Let $M$ be a commutative monoid, $\iota: M \to M^{gp}$ denote the group completion map, and $\mathcal R \subseteq M \times M$ an equivalence relation on $M$ such that $M/\mathcal R$ is still a commutative monoid. The group completion of $M/\mathcal R$ is $M^{gp}/I$, where $I$ is the subgroup generated by $\iota(a) - \iota(b)$ for all $(a, b) \in \mathcal R$. 
\end{lemma}

Using the lemma, we may successfully compute the group completion of $CC_0(X, \Nbb^{\oplus \omega})/\mathcal R$ as the 0-th coarse homology.

\coarse*

\begin{proof}
    By Lemma~\ref{lem::gp_cc0} and Lemma~\ref{lem::gp_quotient_commute}, we have that the group completion of $CC_0(X, \Nbb^{\oplus \omega})/\mathcal R$ is isomorphic to $CC_0(X, \Zbb^{\oplus \omega})/I$, where $I$ is generated by $\iota(a) - \iota(b)$ for all $(a, b) \in \mathcal R$. We claim that $I$ is the same as $I' = \operatorname{im}(\partial: CC_1(X, \Zbb^{\oplus \omega}) \to CC_0(X, \Zbb^{\oplus \omega}))$.\\

    Indeed, take $\iota(a) - \iota(b) \in I$. Since $a$ and $b$ are related by a locality-preserving isomorphism. By Proposition~\ref{prop::replace_shift}, we can assume $a$ and $b$ are related by a shift. This means that there is actually a bijection of prime divisors, with uniformly bounded spread, between $a$ and $b$ on $X$. It follows $\iota(a) - \iota(b)$ is the boundary of some locally finite formal sum $\sum a_{x_0 x_1} [x_0 x_1]$, possibly living in a larger space than $CC_1(X, \Zbb^{\oplus \omega})$ such that $[x_0 x_1]$ has uniformly bounded length. To check that this is in $CC_1(X, \Zbb^{\oplus \omega})$, it suffices for us to show that the collection $(x_0, x_1)$ that appears has uniformly bounded distance to the diagonal. Indeed, this follows from observing that $\max(d(x_0, x_1), d(x_0, x_0)) = d(x_0, x_1)$. Thus, we have that $\iota(a) - \iota(b) \in I'$.\\
    
    Now suppose we have $\sum a_{x_0x_1} [x_0, x_1] \in CC_1(X, \Zbb^{\oplus \omega})$. For $c \in \Zbb^{\oplus \omega}$, we say $c > 0$ if all of its non-zero entries are positive and similarly for $c < 0$. Observe that each $a_{x_0 x_1}$ admits a splitting
    \[a_{x_0 x_1} = a^+_{x_0 x_1} - a^-_{x_0 x_1} \text{ where } a^+_{x_0 x_1}, a^-_{x_0, x_1} > 0.  \]
    In this case, we have that
    \begin{align*}
        \partial(\sum a_{x_0x_1} [x_0, x_1]) &= (\sum a^+_{x_0 x_1} [x_1] - \sum a^+_{x_0 x_1} [x_0]) - (\sum a^-_{x_0 x_1} [x_1] - \sum a^-_{x_0 x_1} [x_0])\\
        &= \underbrace{(\sum a^+_{x_0 x_1} [x_1] + \sum a^-_{x_0 x_1} [x_0])}_{A} - \underbrace{(\sum a^+_{x_0 x_1} [x_0] + \sum a^-_{x_0 x_1} [x_1])}_{B}
    \end{align*}
    Note here we constructed so that $A, B \in CC_0(X, \Nbb^{\oplus \omega})$. It remains to check that $(A, B) \in \mathcal R$. Since the chains are bounded, triangle's inequality implies that $d(x_0, x_1)$ is uniformly bounded over all $(x_0, x_1)$ such that $a_{x_0 x_1} \neq 0$ above. The essential reason why is that
    \[d(x_0, x_1) < d(x_0, x) + d(x, x_1) \leq 2 \max(d(x_0, x), d(x, x_1)).\]
    Now we construct the locality-preserving isomorphism exactly guided by the equation above - for each weighting of $a^+_{x_0, x_1}$ on $x_1$ in $A$, we match it to the weighting of $a^+_{x_0x_1}$ on $x_0$ in $B$, and similarly for $a^-_{x_0 x_1}$. This shows that $(A, B) \in \mathcal R$ and hence that $\partial(\sum a_{x_0x_1} [x_0, x_1]) \in I$. This concludes the proof. 
\end{proof}

\begin{remark}\label{rmk:coarse_connected}
As elaborated in Appendix~\ref{sec::coarse}, the zeroth coarse homology of $\Zbb^n$ is trivial for $n > 0$, so $K(\textbf{C}(\Zbb^n))$ is all connected. This is one motivation for why we can discard the $\pi_0$-information. However, we note that proving $K(\textbf{C}(\Zbb^n))$ is connected actually does not need Proposition~\ref{prop::replace_shift}. Indeed, additional morphisms can only kill $\pi_0$, so if we know the quotient obtained by equivalence relations imposed only by shifts is already trivial, then adding all locality-preserving isomorphisms cannot reduce it further. Thus, we have that $K(\textbf{C}(\Zbb^n))$ is connected over any ring $R$, without needing to impose the constraints of not having non-trivial idempotents. This conclusion is reminiscent of Corollary 1.3.1 of \cite{pedersen_weibel} for the additive case.
\end{remark}

Finally, we now prove Proposition~\ref{prop::replace_shift}.
\begin{proof}[Proof of Proposition~\ref{prop::replace_shift}]
    Let $\alpha: \mathcal{A}(X, q) \to \mathcal{A}(X, r)$ be a locality-preserving isomorphism of spread such that $\alpha$ and $\alpha^{-1}$ have spread less than $\ell$.  We use $d$ to denote the metric on $X$ in this proof. Since the rank is locally constant on projective modules, and the condition on $R$ is equivalent to specifying that $\operatorname{Spec} R$ is connected (e.g., Exercise 1.22 of \cite{atiyah_macdonald_1969}), we have that $\operatorname{Mat}(R^p)$ is an irreducible tensor factor (i.e., no proper tensor factors).

    We will work one prime $p$ at a time, as we will show that the map we construct for each $p$ has an independent spread $< L$. Write $\operatorname{deg}_{p}(q) = \sum_{x \in X} a_x [x]$ and $\operatorname{deg}_{p}(r) = \sum_{x \in X} b_x [x]$.

    For the ease of notation, we will introduce the following combinatorial construction. Recall that a multi-hypergraph \cite{hypergraph_book} is a pair of sets $(V, \mathcal{E})$ where $V$ is thought of as ``vertices" and $\mathcal{E}$ is a multi-set of subsets of $V$, known as ``hyperedges". Here we construct a hypergraph $G = (X, \mathcal{E})$ whose vertices are $X$ and the hyperedges are given by the support of $\alpha(\operatorname{Mat}(R^p))$ for each irreducible tensor factor $\operatorname{Mat}(R^p)$ from $\mathcal{A}(X, q)$ (here the prime $p$ is fixed). For each $e \in \mathcal{E}$, we write $A_e$ to denote the image of said tensor copy over $e$. The $G$ satisfies the following properties:
    \begin{enumerate}
        \item For each $x \in X$, any hyperedge containing $x$ is contained in the disk of radius $L$ centered at $x$, where $L > 2 \ell$. This follows directly from $\alpha, \alpha^{-1}$ having spread $< \ell$.
        \item Let $B_x$ be the number of hyperedges containing $x$, then $B_x$ is finite and $B_x \geq b_x$.\\
        
        \noindent Indeed, $B_x$ being finite follows from the fact that all hyperedges containing $x$ must have originated from tensor copies $T$ on points within radius $\ell$ from $x$, and there are only finitely many of them. The algebra $\mathcal{A}(x, r)$, which has rank whose maximal $p$-power is $(p^2)^{b_x}$, is by construction contained in the algebra $F$ generated by images of $\alpha(T)$'s. Note however that the maximal $p$-power of $\rank F$ is $(p^2)^{B_x}$. Thus, the centralizer theorem (Lemma~\ref{lem::centralizer}) and the prime factorization theorem imply that $b_x \leq B_x$.%, which may have different primes
        
        \item Similarly, for a bounded subset $S \subseteq X$, the number of hyperedges that intersects non-trivially with $S$ is greater than or equal to $\sum_{s \in S} b_s$.
        \item For a hyperedge $e \in E$, there exists $x \in e$ such that $b_x > 0$.\\
        
        \noindent For clarity, we prove this in the generality without doing it ``one prime at a time". Indeed, $A_e$ is a sub-algebra of $\bigotimes_{y \in e} \operatorname{Mat}(R^{r_x})$, so the centralizer theorem implies that $p^2$ divides the latter's rank. If none of the $b_x$'s are non-zero, then the latter's rank cannot be divisible by $p$.
    \end{enumerate}

   We can make a collection $\mathcal{C}$ of maps $f: \alpha(\mathcal{A}(X, q)) \to \mathcal{A}(X, r_{f})$, where $r_f$ depends on $f$, given by sending each $A_e$ to some $x \in e$ that has $b_x > 1$. Observe that $\mathcal{C}$ is non-empty by (4) and $f$ must be of finite spread $< L$ by (1), so $\alpha \circ f$ is a locality-preserving isomorphism.

   In the remainder of this proof, we wish to construct a suitable $f$ such that $r = r_f$. We write $\Lambda \subseteq X$ to be all the points where $b_x > 0$. For a fixed $x \in \Lambda$, we write $B_{L}(x) = \{y \in X\ |\ d(x, y) < L \text{ and } b_y > 0\}$. We claim that we can assign $b_x$ hyperedges containing $x$ to the point $x$ such that each $y \neq x \in \Lambda$ has the desired conditions that:
   \begin{enumerate}[label=(\alph*)]
       \item The number of hyperedges un-assigned that contains $y$ is $\geq b_y$.
       \item The number of hyperedges un-assigned whose intersection with $\Lambda$ is either $\{y\}$ or $\{x, y\}$ is $\leq b_y$.
   \end{enumerate}
   Note that (1) implies we only need to check this for $y \in  B_{L}(x) - \{x\}$. Let us do this by inducting on the number $n$ of hyperedges that we can assigning without violating (a), for $n = 0, ..., b_x$. We will see how to ensure that we do not violate (b) either later.\\

   For $n = 0$, we are done. For $n = 1$, if there are any hyperedges $e$ that intersects with $\Lambda$ exactly at $\{x\}$, then we are done. Otherwise, we suppose no such hyperedge exists. Suppose for contradiction that any hyperedge we remove will violate (a). This implies that for any hyperedge $e$ containing $x$, there exists $y \neq x \in e$ such that $B_y = b_y > 0$. Let $Y$ record all such $y$'s where this occurs. From (3), we must have that
   \[N \coloneqq \# \{\text{hyperedges that intersects non-trivially with } \{x\} \cup Y\} \geq b_x + \sum_{y \in Y} b_y.\]
   On the other hand, clearly we also have that
   \[\sum_{y \in Y} B_y \geq |\bigcup_{y \in Y} \text{hyperedges containing y}| \geq N,\]
   where the left inequality follows from union-bound and the right inequality follows by noting that any hyperedge that intersects non-trivially with $\{x\} \cup Y$ must show up in the union in middle, as we have eliminated all hyperedges whose intersection with $\Lambda$ is only $\{x\}$. This implies that
   \[\sum_{y \in Y} B_y \geq b_x + \sum_{y \in Y} b_y.\]
   By construction of $y$, we also have that $B_y = b_y$ for $y \in Y$, so we have that $0 \geq b_x$, which is a contradiction as long as $b_x > 0$, at least before the induction ends. Evidently, it is also clear from the inequality why we can induct all the way until $n = b_x$.\\

   Now we seek to show that we could always choose hyperedges to remove in the induction so that (b) occurs. We first note that:
   \begin{enumerate}
   \setcounter{enumi}{4}
       \item The number $c_x$ of hyperedges whose intersection with $\Lambda$ is $\{x\}$ must be $\leq b_x$. Indeed, let $S$ be the union of elements in all such $e$'s. The tensor product $E'$ of $A_e$'s over all such $e$'s is a subalgebra of $F' \coloneqq \bigotimes_{s \in S} \operatorname{Mat}(R^{r_s})$. Note $\rank E' = (p^2)^{c_x}$, and $\rank F'$ has maximal prime power $(p^2)^{b_x}$. The centralizer theorem and the prime factorization theorem now imply that $c_x \leq b_x$.
   \end{enumerate}
   Since the induction prioritized removing hyperedges whose intersection with $\Lambda$ is $\{x\}$, all $c_x$ of them must have been used during the induction steps.\\

   Similarly, let $c_y \leq b_y$ be the number of hyperedges whose intersection with $\Lambda$ is $\{y\}$. We also write $d_y$ to be the number of hyperedges whose intersection with $\Lambda$ is $\{x, y\}$. If $c_y + d_y \geq b_y$, then certainly we are allowed to remove a hyperedge associated with $d_y$ in the inductive step after we have removed all hyperedges associated with the quantity $c_x$. Doing this does not violate (a) as the hyperedge being removed only intersects with $\Lambda$ at $\{x, y\}$.\\

   Thus, we see that showing (b) is possible amounts to showing we can remove $\leq b_x - c_x$ such hyperedges such that for all $y$, the quantity $c_y + d_y$ becomes $\leq b_y$. Let $T$ record all $y$'s where $c_y + d_y > b_y$, this amounts to showing that
   \[\sum_{y \in T} c_y + d_y - b_y \leq b_x - c_x.\]
   This is equivalent to showing that
   \[\sum_{y \in T} c_y + d_y \leq (b_x - c_x) + \sum_{y \in T} b_y.\]
   
   Recall $\mathcal{E}$ is the  collection of all hyperedges. We define the set
   \[Y' = \{y \in X\ |\ \{x, y\} = e \cap \Lambda \text{ for some } e \in \mathcal{E}\}.\]
   Then observe that
   \[\sum_{y \in T} c_y + d_y = |\bigcup_{y \in Y'} \text{hyperedge e s.t. }e \cap \Lambda = \{y\} \text{ or } \{x, y\} | \leq (b_x - c_x) + \sum_{y \in T} b_y,\]
   because an argument similar to that of (5) shows that
   \[|\bigcup_{y \in Y'} \text{hyperedge e s.t. }e \cap \Lambda = \{y\} \text{ or } \{x, y\} | + c_x \leq b_x + \sum_{y \in T} b_y.\]
   This shows that (b) can be satisfied as well.\\

   We have now shown that we can satisfy both (a) and (b). After we have done so, we discard the hyperedges we have already assigned and go to a new point $x' \in \Lambda - \{x\}$. We can assign hyperedges to $x'$ in a way that satisfies (a) and (b'), where condition (b') is modified from condition (b) by replacing $\Lambda$ with $\Lambda - \{x\}$. We can keep doing this inductively on a chosen ordering of points in $\Lambda$, and over each prime as well. This does give a well-defined $f$ because for each $A_e$ for some prime $p$, you can inductively define what that tensor factor is sent to. Clearly, by construction this $f$ is injective and has finite spread. It is also surjective because any potentially unassigned tensor factor will eventually intersect with the modified $\Lambda$ at only 1 point, which will be assigned as in (5). This shows that $f$ is a quantum cellular automaton and concludes the proof.
\end{proof}

\section{Delooping and the Omega-Spectrum}\label{sec::spectra}

In this section, we will discuss explicit deloopings of the QCA spaces $\mathbf{Q}(X)$ we constructed, which will give some descriptions for the underlying $\Omega$-spectra. The deloopings we hope for would be similar to those of Pedersen and Weibel in \cite{pedersen_weibel}, which was done for additive categories.

We give a very informal description of what Pedersen and Weibel did. For a small, filtered, idempotent complete additive category $A$, they constructed a category $C_1(A)$ by ``placing copies of $A$ on $\Zbb$" and showed that
\begin{equation}\label{additive_deloop}
\Omega K(C_1(A)^{\cong}) \simeq K(A^{\cong}).
\end{equation}
Pedersen and Weibel similarly built $C_{i}(A)$ by ``placing copies of $A$ on $\Zbb^i$", and it turns out there is a suitable notion of equivalence between $C_{1}(C_i(A))$ and $C_{i+1}(A)$, so one can inductively appeal to (\ref{additive_deloop}) to build a sequence of deloopings of $K(A^{\cong})$.

In this section, we will obtain an analogous delooping of the form 
\begin{equation}\label{actual_deloop}
\mathbf{Q}(*) \simeq \Omega \mathbf{Q}(\Zbb^1), \mathbf{Q}(\Zbb^1) \simeq \Omega \mathbf{Q}(\Zbb^2), ..., \mathbf{Q}(\Zbb^{n-1}) \simeq \Omega \mathbf{Q}(\Zbb^n), ... \quad .
\end{equation}

\begin{remark}
A natural way to try to prove this is perhaps to hope that there is a sequence of the following form and apply $\Omega(\bullet)$ to everything.
\begin{equation}\label{naive_deloop}
K(\textbf{C}(*)) \simeq \Omega K(\textbf{C}(\Zbb^1)), ..., K(\textbf{C}(\Zbb^{n-1})) \simeq \Omega K(\textbf{C}(\Zbb^n)), ... \quad.
\end{equation}
However, (\ref{naive_deloop}), as currently stated, is false. Here we outline some obstructions for (\ref{naive_deloop}) to hold.
\begin{enumerate}
    \item $K(\textbf{C}(*)) \simeq \Omega K(\textbf{C}(\Zbb^1))$ would imply that $K_0(\textbf{C}(*)) = K_1(\textbf{C}(\Zbb^1))$. Theorem~\ref{thm::k0_qz} implies $K_1(\textbf{C}(\Zbb^1)) = K_0(\operatorname{Az}(R))$, and Example~\ref{exp::matrix_one_pt} shows that $K_0(\textbf{C}(*)) = (\Zbb_{>0})^{gp} = \Qbb_{>0}$ is the positive rationals under multiplication. These two groups generally do not equal each other. Furthermore, Remark~\ref{rmk:coarse_connected} tells us that $K(\textbf{C}(\Zbb^n))$ is connected for $n > 0$, so this delooping happening would mean $K_1(\textbf{C}(\Zbb^{n}))$ for $n > 1$ are all $0$. This is also not expected since~\cite{haah2023nontrivial, fidkowski2024qca, sun2025clifford} all construct explicit QCA conjectured to correspond to nontrivial classes in $K_1(\textbf{C}(\Zbb^{3}))$ when $R=\CC$.
    \item The assumption that $A$ is idempotent complete for (\ref{additive_deloop}) suggests that our base categories may not be large enough. Note that the notion of idempotent completion does not really make sense here, as, for example, the inclusion functor $\operatorname{Mat}_R \to \operatorname{Az}(R)$ is not an idempotent completion. However, we see that Theorem~\ref{thm::k0_qz} implies that
    \[K_0(\operatorname{Az}(R)) = K_1(\textbf{C}(\Zbb^1))\]
    We can enlarge the objects of the relevant categories $\textbf{C}(X)$ to $\textbf{C}'(X)$ by allowing placements of Azumaya $R$-algebras at each point, rather than matrix algebras only. Observe that $\textbf{C}'(*) = \operatorname{Az}(R)$ and $\textbf{C}(X)$ is a full cofinal subcategory of $\textbf{C}'(X)$. It is a general fact (IV.4.11 of \cite{weibel2013k}) that for a full cofinal subcategory $S \subset S'$ of symmetric monoidal categories, the inclusion functor induces an isomorphism $K_i(S) \cong K_i(S')$ for $i > 0$. Furthermore, we actually do have now
    \[K_0(\textbf{C}'(*)) = K_0(\operatorname{Az}(R)) = K_1(\textbf{C}(\Zbb^1)) = K_1(\textbf{C}'(\Zbb^1)).\]
    This seems to suggest there might be a delooping $K(\textbf{C}'(*)) \simeq \Omega K(\textbf{C}'(\Zbb^1))$.
    \item Moving over to higher dimensions, however, simply placing an Azumaya algebra at points may not be large enough either, and we might need to enlarge further. Invertible subalgebras defined in~\cite{haah2023invertible} can be used as an attempt at this enlargement. We will see why this is the case in the proof of Theorem~\ref{thm::higher_delooping}).
\end{enumerate}
However, we see that our enlargement of the category should morally only change the $\pi_0$-information, i.e., we want to keep the original $\textbf{C}(\Zbb^n)$ to be full and cofinal. Thus, if we discard the $\pi_0$-information, as in the definition of QCA spaces, then changing the underlying categories does not affect the sequence in \eqref{actual_deloop}. This is the strategy we will take in proving~\eqref{actual_deloop}.
\end{remark}

Another ingredient we need, similar to the case of (\ref{additive_deloop}), is Thomason's simplified double mapping cylinder construction \cite{Thomason01011982}.
\begin{defn}\label{def::cylinder}
    Let $A, B, C$ be symmetric monoidal categories and $u: A \to B$ and $v: A \to C$ be lax symmetric monoidal functors. The simplified double mapping cylinder construction is a symmetric monoidal category $P$ such that:
    \begin{enumerate}
        \item Objects are $(b, a, c)$ for $b \in \operatorname{obj}(B), a \in \operatorname{obj}(A), c \in \operatorname{obj}(C)$.
        \item A morphism $(b, a, c) \to (b', a', c')$ is the equivalence class of maps of the form $(\psi, \psi^-, \psi^+, a^-, a^+)$ where $\psi: a \to a^- \otimes a' \otimes a^+$ is an isomorphism, and $\psi^-: b \otimes u(a^{-}) \to b'$ and $\psi^+: v(a^{+}) \otimes c \to c'$ are morphisms, up to equivalences given by isomorphsims of $a^-$ and $a^+$. For maps $$(\psi, \psi^-, \psi^+, a^-, a^+): (b, a, c) \to (b', a', c')\text{ and }$$ $$(\phi, \phi^-, \phi^+, (a')^-, (a')^+): (b', a', c') \to (b'', a'', c''),$$ their composition is given by $(\phi \circ \psi, \phi^- \circ \psi^-, \phi^+ \circ \psi^+, (a')^{-} \otimes a^-, a^+ \otimes (a')^+)$.
        \item The symmetric monoidal structure is given by applying the symmetric monoidal structure of $A, B,$ and $C$ pointwise.
    \end{enumerate}
\end{defn}

The main theorem of Thomason we will use is the following.
\begin{theorem}[Theorem 5.2 of \cite{Thomason01011982}]\label{thm::thomason_pullback}
    Following the set-up above, the data $(A, B, C, u, v)$ assembles to a homotopy pullback square of the form
% https://q.uiver.app/#q=WzAsNCxbMCwwLCJLKEEpIl0sWzEsMCwiSyhCKSJdLFswLDEsIksoQykiXSxbMSwxLCJLKFApIl0sWzAsMSwidV8qIl0sWzEsM10sWzAsMiwidl8qIiwyXSxbMiwzXV0=
\[\begin{tikzcd}
	{K(A)} & {K(B)} \\
	{K(C)} & {K(P)}
	\arrow["{u_*}", from=1-1, to=1-2]
	\arrow["{v_*}"', from=1-1, to=2-1]
	\arrow[from=1-2, to=2-2]
	\arrow[from=2-1, to=2-2]
\end{tikzcd}\]
This yields a long exact sequence of K-groups as follows:
\[... \to K_{i}(A) \to K_i(B) \oplus K_i(C) \to K_i(P) \to K_{i-1}(A) \to ... \]
\end{theorem}

The homotopy limit of two maps $p: * \to  X$ and $q: * \to X$ is equivalent to the based loop space of $X$ when $p = q$ (e.g. see Example 6.5.3 of \cite{Riehl_2014}). Thus, as a corollary of Theorem~\ref{thm::thomason_pullback}, we obtain the following.
\begin{cor}\label{cor::thomason_corollary}
    Suppose $K(B)$ and $K(C)$ are contractible, then $K(A) \simeq \Omega K(P)$.
\end{cor}
The construction of our desired delooping would be an application of Theorem~\ref{thm::thomason_pullback} and its corollary. We first identify spectra that will be contractible.

\begin{lem}\label{lem::contractible}
    For any metric space $X$, $K(\mathbf{C}(X \times \NN))$, and hence $\mathbf{Q}(X \times \NN)$ is contractible. Here $X \times \NN$ is equipped with the $L^{\infty}$-metric on the product.
\end{lem}

\begin{proof}
    Define a functor
    \begin{equation}
        T\colon \textbf{C}(X\times \NN) \rightarrow  \textbf{C}(X\times \NN).
    \end{equation}
For an object $q\in \NN^{X \times \NN}_\mathrm{lf}$
\begin{equation}
(Tq)(x, n) =
\begin{cases}
  1, & \text{if } x \in X,\ n = 1, \\[4pt]
  q(x, n-1), & \text{otherwise.}
\end{cases}
\end{equation}
Observe that there is an obvious locality-preserving isomorphism 

\begin{equation}
    \tau_q\colon\SA(X \times \NN,q)\xrightarrow{\sim}\SA(X \times \NN,Tq)
\end{equation}
which identifies $\SA(X\times \{n\},q)$ with $\SA(X\times \{n+1\},Tq)$. Since $\SA(X\times \{1\},Tq)=R$, the map is invertible (recall that $\Nbb$ here does not include $0$).

Given a locality-preserving isomorphism $\alpha: \SA(X\times \NN,q)\rightarrow \SA(X\times \NN,r)$, define
\begin{equation}
    T\al\colon \SA(X\times \NN,Tq)\xrightarrow{\tau_r\circ\alpha\circ\tau_q^{-1}}\SA(X\times \NN,Tr).
\end{equation}
It is easy to check that $T$ is a monoidal functor. 
Moreover, $S:=\bigotimes_{k=1}^\infty T^k$ and $\bigotimes_{k=0}^\infty T^k=I\otimes S$ are well-defined monoidal endofunctors on $\textbf{C}(X \times \NN)$. They are naturally isomorphic to each other. That is, for an object $q$, there is a natural isomorphism between $Sq$ and $q\otimes Sq$. On the group completion, $S$ induces a map $s\colon K(\textbf{C}(X \times \NN)) \rightarrow K(\textbf{C}(X \times \NN))$ with homotopy $\id+s\simeq s$, where $+$ is the $H$-space operation on $K(\textbf{C}(X \times \NN))$. It follows that $K(\textbf{C}(X \times \NN))$ is contractible.
\end{proof}

Motivated by the discussions in Remark~\ref{rmk::extend_factor}, we now construct an appropriate category. Since Azumaya algebras span all possible tensor factors of matrix algebras over a point $*$ (i.e., $\mathcal{A}(*, q)$), one might be tempted to include all possible tensor factors in $\mathcal{A}(X, q)$ for a general $X$. We will see in Remark~\ref{rmk::not_local} that this is not the desired notion. Instead, the appropriate category is defined via a construction analogous to that used in the proof of Theorem~\ref{thm::k0_qz}.
\begin{construction}\label{defn::admissible}
    Let $X$ be a metric space with assumptions as before, and $q$ be a quantum spin system on $X$. We define the \textit{admissible tensor factors} on \( X \) to be those tensor factors that arise from the following procedure:
    \begin{enumerate}
        \item Choose a quantum spin system $q$ on $X \times \Zbb$, equipped with the $L^{\infty}$-metric, and a QCA $\alpha$ of spread $< \ell$.
        \item Apply Lemma~\ref{lem::tensor_splitting} to the containments
        \[\mathcal{A}(X \times \Zbb_{\leq -\ell}, q) \subset \alpha(\mathcal{A}(X \times \Zbb_{\leq  0}, q)) \subset \mathcal{A}(X \times \Zbb_{\leq \ell}, q)\]
        to obtain $\mathcal B$ such that
        \[\mathcal{A}(X \times \Zbb_{\leq -\ell}, q) \otimes \mathcal B = \alpha(\mathcal{A}(X \times \Zbb_{\leq 0}, q)). \]
        Note that $\mathcal B = \mathcal{A}(X \times \Zbb_{(-\ell, \ell]}, q) \cap \alpha(\mathcal{A}(X \times \Zbb_{\leq 0}, q)) $, where we write $\ZZ_{(-\ell, \ell]}$ to mean $\ZZ\cap (-\ell, \ell]$.
        \item As $\ZZ_{(-\ell, \ell]}$ is bounded, one can define 
    $$q'(x):=\prod_{n\in \ZZ\cap(-\ell, \ell]} q(x, n)\in \NN^X_{\mathrm{lf}} $$ and regard $\mathcal{B}\subseteq\SA(X, q').$
    \end{enumerate}
    In general, there is no $r\in \NN^X_{\mathrm{lf}}$ with $\mathcal{B}\cong\SA(X, r)$ just like not every Azumaya algebra is a full matrix algebra.
    \end{construction}

\begin{exmp}
    When $X = \Zbb^0 = *$ (or if $X$ is bounded), the admissible tensor factors on a point is the same as Azumaya algebras. This is in part what Theorem~\ref{thm::k0_qz} is about.
\end{exmp}

We first deduce some core properties of admissible tensor factors.
\begin{lemma}\label{lem::admissible_faithfully_flat}
    An admissible tensor factor is faithfully flat as an $R$-module.
\end{lemma}
\begin{proof}
Note that, by definition, an admissible tensor factor $\mathcal B$ satisfies
\[
\mathcal{A}(X \times \Zbb_{\leq -\ell}, q) \otimes_R \mathcal B
=
\alpha\big(\mathcal{A}(X \times \Zbb_{\leq 0}, q)\big)
\]
for some QCA $\alpha$. The right-hand side is isomorphic to $\mathcal{A}(X \times \Zbb_{\leq 0}, q)$, which is free by Proposition~\ref{prop::free_module}. Similarly, $\mathcal{A}(X \times \Zbb_{\leq -\ell}, q)$ is free.

Let $L \to M$ be an injective map of $R$-modules. Then
\[
\mathcal{A}(X \times \Zbb_{\leq -\ell}, q) \otimes_R \mathcal B \otimes_R L
\longrightarrow
\mathcal{A}(X \times \Zbb_{\leq -\ell}, q) \otimes_R \mathcal B \otimes_R M
\]
is injective, since free modules are flat. Because free modules are in fact faithfully flat, it follows that
\[
\mathcal B \otimes_R L
\longrightarrow
\mathcal B \otimes_R M
\]
is also injective. Hence $\mathcal B$ is flat.

On the other hand, if $\mathcal B \otimes_R L = 0$, then
\[
\alpha\big(\mathcal{A}(X \times \Zbb_{\leq 0}, q)\big) \otimes_R L = 0
\]
as well. Since $\alpha\big(\mathcal{A}(X \times \Zbb_{\leq 0}, q)\big)$ is faithfully flat, this forces $L=0$. 
\end{proof}

Lemma~\ref{lem::admissible_faithfully_flat} gives an extended version of the centralizer theorem (Lemma~\ref{lem::centralizer}).
\begin{lemma}\label{lem::tensor_complement_exist}
    Let $B \subseteq A \subseteq C$ where $C = \mathcal{A}(X, q)$. Suppose $A$ and $B$ are both admissible tensor factors of $C$, then $A = B \otimes (A \cap B')$.
\end{lemma}

\begin{proof}
    We have that $B \otimes B' = C$ for some $B'$. We have $B=B\otimes_R 1 \subseteq A \subseteq B \otimes_R B'$. Thus, Corollary~\ref{cor::tensor_factor_splitting} implies that $A = B \otimes_R B''$ for some $B''$. We can identify $B''$, as constructed in the corollary, with $A \cap B'$. This is because $A \cap B'$ is equal to the centralizer of $B$ in $A$. Note we can apply the corollary as all tensor factors contain a copy of $R$, and admissible tensor factors are faithfully flat.
\end{proof}

 For an admissible tensor factor $\mathcal{B}$ arising in Construction~\ref{defn::admissible}, it is not hard to see that the multiplication map $\mathcal B\otimes_R \mathcal B'\rightarrow \mathcal{A}(X, q')$ is an $R$-algebra isomorphism with $\mathcal B'$ the centralizer of $\mathcal B$ (this can be done using Lemma~\ref{lem::tensor_complement_exist}). Lastly, by a symmetric argument,
     \[ \mathcal B'\otimes \mathcal{A}(X \times \Zbb_{>\ell}, q)  = \alpha(\mathcal{A}(X \times \Zbb_{> 0}, q)). \]
 We also have the following characterization of admissible tensor factors. 
\begin{defn}
   A tensor factor $\mathcal B$ of a subalgebra $\mathcal S \subseteq \mathcal{A}(X,q')$ is \emph{invertible} if there exists $\ell' > 0$ such that for every $a \in \mathcal S$ there exist elements $b_i \in \mathcal B$ and $b_i' \in \mathcal B'$, where $\mathcal B'$ denotes the tensor complement of $\mathcal B$ in $\mathcal S$, satisfying
\begin{equation}
    a = \sum_i b_i \otimes b_i',
\end{equation}
and such that the supports of $b_i$ and $b_i'$ are contained in the $\ell'$-neighborhood of the support of $a$.
\end{defn}

\begin{lemma}\label{lem::invert_local_preserve_equal}
    A tensor factor $\mathcal B\subseteq \mathcal{A}(X, q')$ with complement $\mathcal{B}'$ is admissible if and only it is invertible.
\end{lemma}

\begin{proof}[Proof of Lemma~\ref{lem::invert_local_preserve_equal}]
    Let $a\in \mathcal{A}(X, q')$ with support $S\subset X$. For the forward direction, we use the notation introduced in Construction~\ref{defn::admissible}. In particular, we remember a lift $a\in \mathcal{A}(X\times \ZZ_{(-\ell, \ell]}, q)\cong \mathcal{A}(X, q')$. Then, the support of this lift is contained in $S\times \ZZ_{(-\ell, \ell]}$. Since the inverse of $\alpha$ also has spread $\ell$, $\alpha^{-1}(a)$ has support contained in $B_\ell(S)\times \ZZ_{(-2\ell,2\ell]}$. Write
    \begin{equation}
        \alpha^{-1}(a)=\sum_i c_i\otimes c_i', 
    \end{equation}
    with the supports of $c_i$ contained in $B_\ell(S)\times \ZZ_{(-2\ell,0]}$ and that of $c_i'$ contained in $B_\ell(S)\times \ZZ_{(0, 2\ell]}$. Then 
    \begin{equation} \label{eqn::pre-projection}
        a=\sum_i \alpha(c_i)  \al(c_i'), 
    \end{equation}
    with the support of $\alpha(c_i)$ contained in $B_{2\ell}(S)\times \ZZ_{(-3\ell,\ell]}$ and that of $\al(c_i')$ contained in $B_{2\ell}(S)\times \ZZ_{(-\ell, 3\ell]}$. Most importantly, 
    \begin{equation}
        \mathcal{A}(X \times \Zbb_{\leq -\ell}, q) \otimes \mathcal B = \alpha(\mathcal{A}(X \times \Zbb_{\leq 0}, q))\ni \al(c_i)
    \end{equation}
    and 
    \begin{equation}
        \mathcal B'\otimes \mathcal{A}(X \times \Zbb_{>\ell}, q)  = \alpha(\mathcal{A}(X \times \Zbb_{> 0}, q))\ni \al(c_i').
    \end{equation}
    In summary, 
    \begin{equation}\label{project_tensor_1}
        \mathcal{A}(B_{2\ell}(S) \times \Zbb_{(-3\ell, -\ell]}, q) \otimes \left(\mathcal{A}(B_{2\ell}(S) \times \Zbb_{(-\ell, \ell]}, q)\cap \mathcal B\right)\ni \al(c_i)
    \end{equation}
    and 
    \begin{equation}\label{project_tensor_2}
        \mathcal{A}(B_{2\ell}(S) \times \Zbb_{(\ell, 3\ell]}, q) \otimes \left(\mathcal{A}(B_{2\ell}(S) \times \Zbb_{(-\ell, \ell]}, q)\cap\mathcal B'\right)\ni \al(c_i').
    \end{equation}
   If only $\al(c_i)$ and $\al(c_i')$ were in their respective latter tensor factors, we would have been done. Nevertheless, taking projections would resolve this issue. Specifically, both $\mathcal{A}(B_{2\ell}(S) \times \Zbb_{(-3\ell, -\ell]}, q)$ and $\mathcal{A}(B_{2\ell}(S) \times \Zbb_{(\ell, 3\ell]}, q)$ are finitely generated free $R$-modules. Let $\pi^-$ and $\pi^+$ be respectively the projections to the identity.\footnote{This is a module homomorphism to $R$ which sends identity to $1\in R$. In particular, it is not a map of $R$-algebras, which is ok in our usage here.} Write $\alpha(c_i) = \sum_j d_j^{-} \otimes e_j$ and $\alpha(c_i') = \sum_k d_k^{+} \otimes e_k'$ according to their tensor decompositions in (\ref{project_tensor_1}) and (\ref{project_tensor_2}). Apply $\pi^-\otimes \id\otimes \pi^+$, where $\id$ is the identity map on $\SA(X\times \ZZ_{(-l, l]})$, to equation~\eqref{eqn::pre-projection} and obtain
    \begin{align*}
        a&= \pi^{-} \otimes \operatorname{id} \otimes \pi^{+}(1 \otimes a \otimes 1) = \sum_{i, j} \pi^{-}(d_j^{-}) \otimes \operatorname{id}(e_j \otimes e_k') \otimes \pi^{+}(d_k^{+})\\
        &= \sum_{i, j} (\pi^{-}(d_j^{-}) \otimes e_j) \otimes (e_k' \otimes \pi^{+}(d_k^{+})).
    \end{align*}
    We now see that
    \begin{equation}
        b_{ij} :=\pi^{-}(d_j^{-}) \otimes e_j \in \mathcal{A}(B_{2\ell}(S) \times \Zbb_{(-\ell, \ell]}, q)\cap \mathcal B,
    \end{equation}
    \begin{equation}
    b_{ij}':=e_k' \otimes \pi^{+}(d_k^{+})\in \mathcal{A}(B_{2\ell}(S) \times \Zbb_{(-\ell, \ell]}, q)\cap\mathcal B'
    \end{equation}
    and they satisfy the support condition with $\ell'=2\ell.$

Conversely, suppose $\mathcal{B}$ is an invertible tensor complement and $\mathcal{B}'$ is its tensor complement. Consider $\mathcal{A}(X \times \Zbb, q)$, outlined by placing matrix algebras on $X \times \Zbb$ as in Figure~\ref{fig::admissible}, and with a QCA $\alpha$ given by shifting the tensor factor $\mathcal{B}'$ to the left. The map $\alpha$ is locality-preserving as $\mathcal{B}$ and $\mathcal{B}'$ can factor any element of $\mathcal{A}(X, q')$ in a way with uniformly bounded supports. It has spread $< \ell'$ since we are taking the $L^{\infty}$-metric. Carrying out the steps in Construction~\ref{defn::admissible}, we also see that $\mathcal{B}$ is produced and is admissible. 
\end{proof}

\begin{figure}
    \centering
\[
\begin{tikzpicture}[x=0.75pt,y=0.75pt,yscale=-1,xscale=1]

%Shape: Circle [id:dp04793188511164159] 
\draw  [fill={rgb, 255:red, 0; green, 0; blue, 0 }  ,fill opacity=1 ] (290,148.75) .. controls (290,146.4) and (291.9,144.5) .. (294.25,144.5) .. controls (296.6,144.5) and (298.5,146.4) .. (298.5,148.75) .. controls (298.5,151.1) and (296.6,153) .. (294.25,153) .. controls (291.9,153) and (290,151.1) .. (290,148.75) -- cycle ;
%Shape: Circle [id:dp12745166225417282] 
\draw  [fill={rgb, 255:red, 0; green, 0; blue, 0 }  ,fill opacity=1 ] (354,148.75) .. controls (354,146.4) and (355.9,144.5) .. (358.25,144.5) .. controls (360.6,144.5) and (362.5,146.4) .. (362.5,148.75) .. controls (362.5,151.1) and (360.6,153) .. (358.25,153) .. controls (355.9,153) and (354,151.1) .. (354,148.75) -- cycle ;
%Straight Lines [id:da347449926497857] 
\draw    (294.25,148.75) -- (358.25,148.75) ;
%Shape: Circle [id:dp27574170466254766] 
\draw  [fill={rgb, 255:red, 0; green, 0; blue, 0 }  ,fill opacity=1 ] (418,148.75) .. controls (418,146.4) and (419.9,144.5) .. (422.25,144.5) .. controls (424.6,144.5) and (426.5,146.4) .. (426.5,148.75) .. controls (426.5,151.1) and (424.6,153) .. (422.25,153) .. controls (419.9,153) and (418,151.1) .. (418,148.75) -- cycle ;
%Straight Lines [id:da8993040431146623] 
\draw    (358.25,148.75) -- (422.25,148.75) ;
%Shape: Circle [id:dp8479760100358686] 
\draw  [fill={rgb, 255:red, 0; green, 0; blue, 0 }  ,fill opacity=1 ] (482,148.75) .. controls (482,146.4) and (483.9,144.5) .. (486.25,144.5) .. controls (488.6,144.5) and (490.5,146.4) .. (490.5,148.75) .. controls (490.5,151.1) and (488.6,153) .. (486.25,153) .. controls (483.9,153) and (482,151.1) .. (482,148.75) -- cycle ;
%Straight Lines [id:da43478792360452734] 
\draw    (422.25,148.75) -- (486.25,148.75) ;
%Shape: Circle [id:dp6952066550730442] 
\draw  [fill={rgb, 255:red, 0; green, 0; blue, 0 }  ,fill opacity=1 ] (228,148.75) .. controls (228,146.4) and (229.9,144.5) .. (232.25,144.5) .. controls (234.6,144.5) and (236.5,146.4) .. (236.5,148.75) .. controls (236.5,151.1) and (234.6,153) .. (232.25,153) .. controls (229.9,153) and (228,151.1) .. (228,148.75) -- cycle ;
%Straight Lines [id:da9060691121679773] 
\draw    (232.25,148.75) -- (296.25,148.75) ;
%Straight Lines [id:da1780443082550054] 
\draw    (186.25,148.75) -- (250.25,148.75) ;
%Straight Lines [id:da5923248296577643] 
\draw    (470.25,148.75) -- (534.25,148.75) ;
%Straight Lines [id:da8981143887591957] 
\draw    (358.31,85.31) -- (358.19,212.19) ;
%Shape: Circle [id:dp6845974769556276] 
\draw  [fill={rgb, 255:red, 0; green, 0; blue, 0 }  ,fill opacity=1 ] (175,148.75) .. controls (175,146.4) and (176.9,144.5) .. (179.25,144.5) .. controls (181.6,144.5) and (183.5,146.4) .. (183.5,148.75) .. controls (183.5,151.1) and (181.6,153) .. (179.25,153) .. controls (176.9,153) and (175,151.1) .. (175,148.75) -- cycle ;
%Straight Lines [id:da8311402749177544] 
\draw    (164,148.75) -- (228,148.75) ;
%Curve Lines [id:da6700215823011] 
\draw [color={rgb, 255:red, 208; green, 2; blue, 27 }  ,draw opacity=1 ]   (281,111) .. controls (239.84,82.58) and (262.07,81.44) .. (230.03,108.7) ;
\draw [shift={(228,110.4)}, rotate = 320.36] [fill={rgb, 255:red, 208; green, 2; blue, 27 }  ,fill opacity=1 ][line width=0.08]  [draw opacity=0] (8.93,-4.29) -- (0,0) -- (8.93,4.29) -- cycle    ;
%Curve Lines [id:da08982039166969036] 
\draw [color={rgb, 255:red, 208; green, 2; blue, 27 }  ,draw opacity=1 ]   (217,112) .. controls (175.84,83.58) and (198.07,82.44) .. (166.03,109.7) ;
\draw [shift={(164,111.4)}, rotate = 320.36] [fill={rgb, 255:red, 208; green, 2; blue, 27 }  ,fill opacity=1 ][line width=0.08]  [draw opacity=0] (8.93,-4.29) -- (0,0) -- (8.93,4.29) -- cycle    ;
%Curve Lines [id:da5382136594963179] 
\draw [color={rgb, 255:red, 208; green, 2; blue, 27 }  ,draw opacity=1 ]   (344,112) .. controls (302.84,83.58) and (325.07,82.44) .. (293.03,109.7) ;
\draw [shift={(291,111.4)}, rotate = 320.36] [fill={rgb, 255:red, 208; green, 2; blue, 27 }  ,fill opacity=1 ][line width=0.08]  [draw opacity=0] (8.93,-4.29) -- (0,0) -- (8.93,4.29) -- cycle    ;
%Curve Lines [id:da7623691769880806] 
\draw [color={rgb, 255:red, 208; green, 2; blue, 27 }  ,draw opacity=1 ]   (413,111) .. controls (371.84,82.58) and (394.07,81.44) .. (362.03,108.7) ;
\draw [shift={(360,110.4)}, rotate = 320.36] [fill={rgb, 255:red, 208; green, 2; blue, 27 }  ,fill opacity=1 ][line width=0.08]  [draw opacity=0] (8.93,-4.29) -- (0,0) -- (8.93,4.29) -- cycle    ;
%Curve Lines [id:da6028908040862447] 
\draw [color={rgb, 255:red, 208; green, 2; blue, 27 }  ,draw opacity=1 ]   (477,111) .. controls (435.84,82.58) and (458.07,81.44) .. (426.03,108.7) ;
\draw [shift={(424,110.4)}, rotate = 320.36] [fill={rgb, 255:red, 208; green, 2; blue, 27 }  ,fill opacity=1 ][line width=0.08]  [draw opacity=0] (8.93,-4.29) -- (0,0) -- (8.93,4.29) -- cycle    ;
%Curve Lines [id:da2314112761102033] 
\draw [color={rgb, 255:red, 208; green, 2; blue, 27 }  ,draw opacity=1 ]   (542,111) .. controls (500.84,82.58) and (523.07,81.44) .. (491.03,108.7) ;
\draw [shift={(489,110.4)}, rotate = 320.36] [fill={rgb, 255:red, 208; green, 2; blue, 27 }  ,fill opacity=1 ][line width=0.08]  [draw opacity=0] (8.93,-4.29) -- (0,0) -- (8.93,4.29) -- cycle    ;

% Text Node
\draw (281,159.4) node [anchor=north west][inner sep=0.75pt]    {$-\ell'$};
% Text Node
\draw (347,158.4) node [anchor=north west][inner sep=0.75pt]    {$0$};
% Text Node
\draw (416,160.4) node [anchor=north west][inner sep=0.75pt]    {$\ell'$};
% Text Node
\draw (480,160.4) node [anchor=north west][inner sep=0.75pt]    {$2\ell'$};
% Text Node
\draw (219,158.4) node [anchor=north west][inner sep=0.75pt]    {$-2\ell'$};
% Text Node
\draw (138,133.4) node [anchor=north west][inner sep=0.75pt]  [font=\large]  {$...$};
% Text Node
\draw (165,159.4) node [anchor=north west][inner sep=0.75pt]    {$-3\ell'$};
% Text Node
\draw (209,115.4) node [anchor=north west][inner sep=0.75pt]  [font=\small,color={rgb, 255:red, 208; green, 2; blue, 27 }  ,opacity=1 ]  {$\underbrace{\mathcal{B} '} \otimes \mathcal{B}$};
% Text Node
\draw (539,140.4) node [anchor=north west][inner sep=0.75pt]    {$\mathbb{Z}$};
% Text Node
\draw (338,78.4) node [anchor=north west][inner sep=0.75pt]    {$X$};
% Text Node
\draw (145,116.4) node [anchor=north west][inner sep=0.75pt]  [font=\small,color={rgb, 255:red, 208; green, 2; blue, 27 }  ,opacity=1 ]  {$\underbrace{\mathcal{B} '} \otimes \mathcal{B}$};
% Text Node
\draw (272,116.4) node [anchor=north west][inner sep=0.75pt]  [font=\small,color={rgb, 255:red, 208; green, 2; blue, 27 }  ,opacity=1 ]  {$\underbrace{\mathcal{B} '} \otimes \mathcal{B}$};
% Text Node
\draw (341,115.4) node [anchor=north west][inner sep=0.75pt]  [font=\small,color={rgb, 255:red, 208; green, 2; blue, 27 }  ,opacity=1 ]  {$\underbrace{\mathcal{B} '} \otimes \mathcal{B}$};
% Text Node
\draw (405,115.4) node [anchor=north west][inner sep=0.75pt]  [font=\small,color={rgb, 255:red, 208; green, 2; blue, 27 }  ,opacity=1 ]  {$\underbrace{\mathcal{B} '} \otimes \mathcal{B}$};
% Text Node
\draw (470,115.4) node [anchor=north west][inner sep=0.75pt]  [font=\small,color={rgb, 255:red, 208; green, 2; blue, 27 }  ,opacity=1 ]  {$\underbrace{\mathcal{B} '} \otimes \mathcal{B}$};
% Text Node
\draw (541,109.4) node [anchor=north west][inner sep=0.75pt]  [font=\large,color={rgb, 255:red, 208; green, 2; blue, 27 }  ,opacity=1 ]  {$...$};
% Text Node
\draw (122,113.4) node [anchor=north west][inner sep=0.75pt]  [font=\large,color={rgb, 255:red, 208; green, 2; blue, 27 }  ,opacity=1 ]  {$...$};

\end{tikzpicture}
\]
\caption{Figure displaying a quantum spin system on $X \times \Zbb$ such that applying the steps in Construction~\ref{defn::admissible} would produce the admissible tensor factor $\mathcal{B}$.}
\label{fig::admissible}
\end{figure}
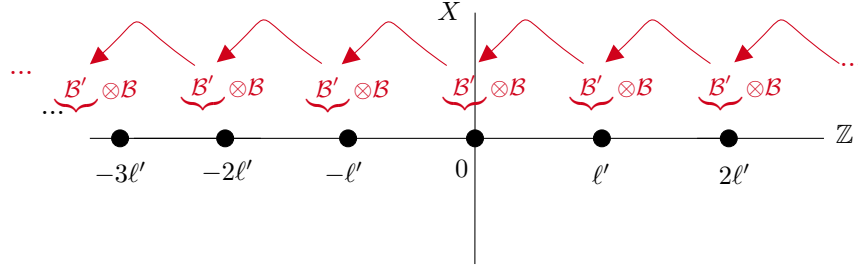

\begin{remark}\label{rmk::not_local}
Here we give an example of a tensor factor $\mathcal{B}$ that is not invertible and see why the map in Figure~\ref{fig::admissible} fails to be locality-preserving. Let $\Rbb$ be the reals and $\mathbb{H}$ be the quaternions. Recall that there is a decomposition $\operatorname{Mat}(\Rbb^2) \otimes \operatorname{Mat}(\Rbb^2) \cong \operatorname{Mat}(\Rbb^4) \cong \mathbb{H} \otimes \mathbb{H}$.

Now let $q: \Zbb \to \Nbb_{> 0}$ be a quantum spin system where $q(0) = 1$ and $q(n) = 2$ for $n \neq 0$, we will build a tensor factor $\mathcal{B}$ of $\mathcal{A}(\Zbb, q)$ as follows. Observe that for $n \neq 0$, $\operatorname{Mat}(\Rbb^{q_n}) \otimes \operatorname{Mat}(\Rbb^{q_{-n}})$ admits a decomposition into $\mathbb{H} \otimes \mathbb{H}$ as explained above, and each $\mathbb{H}$ has support $\{-n, n\}$. We let $\mathcal{B}$ be the tensor factor generated by taking one of these $\mathbb{H}$'s at each $n$ as $n \to \infty$. The map $\alpha$ displayed in Figure~\ref{fig::admissible} cannot be locality-preserving. Indeed, for an element $x \in \mathcal{A}(\{n\}, q)$ that admits a non-trivial decomposition into elements of the two quaternions, the support of $\alpha(x)$ is not uniformly bounded for all $n$.
\end{remark}

Next we prove a property of invertible tensor factors that will be helpful.
\begin{lem}\label{lem::invertible_transitive}
    Let $\mathcal{B}$ be an invertible tensor factor of $\mathcal{A}(X, q')$ and $\mathcal{D}$ an invertible tensor factor of $\mathcal{B}$. Then $\mathcal{D}$ is an invertible tensor factor of $\mathcal{A}(X, q')$.
\end{lem}

\begin{proof}
    Write $\mathcal{B}', \mathcal{D}'$ for the tensor complements of $\mathcal{B}$ and $\mathcal{D}$ respectively. Since $\mathcal{B}$ is invertible, there exists a uniform $\ell > 0$ such that for any $a \in \mathcal{A}(X, q')$, we can write
    \[a = \sum_{i = 1}^N b_i \otimes b_i', b_i \in \mathcal{B}, b_i' \in \mathcal{B}',\]
    and the supports of $b_i, b_i'$ are contained in the $\ell$-ball of $\operatorname{supp}(a)$. Since each $b_i \in \mathcal{B}$, there is some uniform $\ell' > 0$ such that we can write
    \[b_i = \sum_{j = 1}^{N_i} d_{i,j} \otimes d_{i, j}', d_{i, j} \in \mathcal{D}, d_{i, j}' \in \mathcal{D}',\]
    and the supports of $d_{i, j}, d_{i, j}'$ are contained in the $\ell'$-ball of $\operatorname{supp}(b_i)$. We can now write
    \[a = \sum_{i = 1}^N \sum_{j = 1}^{N_i} d_{i, j} \otimes (d_{i,j}' \otimes b_i').\]
    The triangle inequality implies $\operatorname{supp}(d_{i, j})$ is contained in the $\ell+\ell'$-ball of $\operatorname{supp}(a)$. The support of $x \otimes y$ is contained in the union of $\operatorname{supp}(x)$ and $\operatorname{supp}(y)$, and hence $\operatorname{supp}(d'_{i, j} \otimes b_i')$ is contained in the $\ell+\ell'$-ball of $\operatorname{supp}(a)$ as well.
\end{proof}

\begin{defn}\label{def::enlarge_category}
Let $W \in \{*, \Zbb_{\leq 0}, \Zbb_{\geq 0}, \Zbb\}$, we define a symmetric monoidal category $\textbf{C}'_{X}(W)$ as follows.
\begin{enumerate}
    \item Objects are given by assigning each point $w \in W$ an admissible tensor factor of $\mathcal{A}(X, q_w)$. Note that each object comes as a tensor factor on $A(X \times W, q)$ for some $q$.
    \item A morphism between objects $F, E \in \textbf{C}'_X(W)$ is an $R$-algebra isomorphism of finite spread between them as tensor factors. The locality information here is given by the matrix algebras $F, E$ are respectively contained in. Here, it is important to recall we actually use the $L^{\infty}$-metric on $X \times W$.
    \item The symmetric monoidal structure is given by stacking $F$ and $E$ in the stacking of their parent matrix algebras.
\end{enumerate}
As a special case, we write $\textbf{C}'_n(W) \coloneqq \textbf{C}'_{\Zbb^n}(W)$.
\end{defn}

See Figure~\ref{fig:enlargement} for a picture of one such object in $\mathbf{C}'_1(\Zbb)$. Observe also that $\textbf{C}(X)$ is cofinal in $\textbf{C}'_X(*)$. Indeed, if $B, B'$ are admissible tensor complements in $\mathcal{A}(X, q)$, then $B \otimes B'$, sitting in $\mathcal{A}(X, q^2)$, admits a locality-preserving isomorphism to $\mathcal{A}(X, q)$, viewed as an admissible tensor factor of itself. This similarly holds for $\textbf{C}(X \times W)$ in $\textbf{C}'_{X}(W)$.

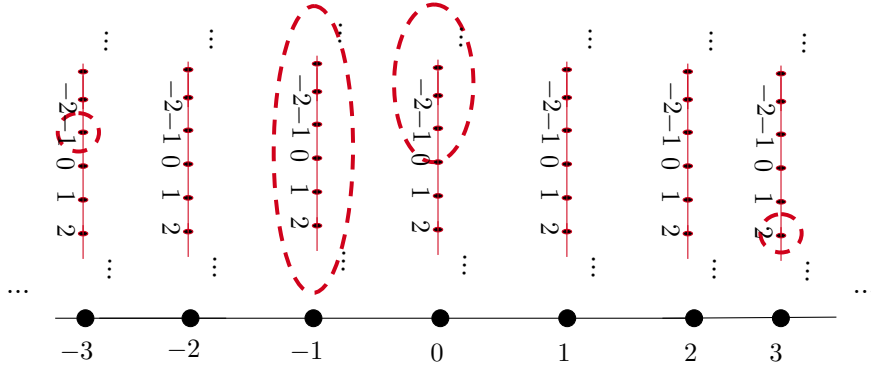
\begin{figure}
    \centering
\[\begin{tikzpicture}[x=0.75pt,y=0.75pt,yscale=-1,xscale=1]

%Shape: Circle [id:dp03820697115113547] 
\draw  [fill={rgb, 255:red, 0; green, 0; blue, 0 }  ,fill opacity=1 ] (267,183.75) .. controls (267,181.4) and (268.9,179.5) .. (271.25,179.5) .. controls (273.6,179.5) and (275.5,181.4) .. (275.5,183.75) .. controls (275.5,186.1) and (273.6,188) .. (271.25,188) .. controls (268.9,188) and (267,186.1) .. (267,183.75) -- cycle ;
%Shape: Circle [id:dp8591703178846204] 
\draw  [fill={rgb, 255:red, 0; green, 0; blue, 0 }  ,fill opacity=1 ] (331,183.75) .. controls (331,181.4) and (332.9,179.5) .. (335.25,179.5) .. controls (337.6,179.5) and (339.5,181.4) .. (339.5,183.75) .. controls (339.5,186.1) and (337.6,188) .. (335.25,188) .. controls (332.9,188) and (331,186.1) .. (331,183.75) -- cycle ;
%Straight Lines [id:da053129181114499824] 
\draw    (271.25,183.75) -- (335.25,183.75) ;
%Shape: Circle [id:dp3950523972908252] 
\draw  [fill={rgb, 255:red, 0; green, 0; blue, 0 }  ,fill opacity=1 ] (395,183.75) .. controls (395,181.4) and (396.9,179.5) .. (399.25,179.5) .. controls (401.6,179.5) and (403.5,181.4) .. (403.5,183.75) .. controls (403.5,186.1) and (401.6,188) .. (399.25,188) .. controls (396.9,188) and (395,186.1) .. (395,183.75) -- cycle ;
%Straight Lines [id:da05877142012542658] 
\draw    (335.25,183.75) -- (399.25,183.75) ;
%Shape: Circle [id:dp47988892390951754] 
\draw  [fill={rgb, 255:red, 0; green, 0; blue, 0 }  ,fill opacity=1 ] (459,183.75) .. controls (459,181.4) and (460.9,179.5) .. (463.25,179.5) .. controls (465.6,179.5) and (467.5,181.4) .. (467.5,183.75) .. controls (467.5,186.1) and (465.6,188) .. (463.25,188) .. controls (460.9,188) and (459,186.1) .. (459,183.75) -- cycle ;
%Straight Lines [id:da011505540059438268] 
\draw    (399.25,183.75) -- (463.25,183.75) ;
%Shape: Circle [id:dp18603722356801067] 
\draw  [fill={rgb, 255:red, 0; green, 0; blue, 0 }  ,fill opacity=1 ] (205,183.75) .. controls (205,181.4) and (206.9,179.5) .. (209.25,179.5) .. controls (211.6,179.5) and (213.5,181.4) .. (213.5,183.75) .. controls (213.5,186.1) and (211.6,188) .. (209.25,188) .. controls (206.9,188) and (205,186.1) .. (205,183.75) -- cycle ;
%Straight Lines [id:da26530926615211736] 
\draw    (209.25,183.75) -- (273.25,183.75) ;
%Straight Lines [id:da3818858924935026] 
\draw    (163.25,183.75) -- (227.25,183.75) ;
%Straight Lines [id:da8288595630296679] 
\draw    (447.25,183.75) -- (535,183.2) ;
%Shape: Circle [id:dp6674841806658004] 
\draw  [fill={rgb, 255:red, 0; green, 0; blue, 0 }  ,fill opacity=1 ] (152,183.75) .. controls (152,181.4) and (153.9,179.5) .. (156.25,179.5) .. controls (158.6,179.5) and (160.5,181.4) .. (160.5,183.75) .. controls (160.5,186.1) and (158.6,188) .. (156.25,188) .. controls (153.9,188) and (152,186.1) .. (152,183.75) -- cycle ;
%Straight Lines [id:da8159822409859336] 
\draw    (141,183.75) -- (205,183.75) ;
%Shape: Circle [id:dp7608921703315183] 
\draw  [fill={rgb, 255:red, 0; green, 0; blue, 0 }  ,fill opacity=1 ] (502.75,183.75) .. controls (502.75,181.4) and (504.65,179.5) .. (507,179.5) .. controls (509.35,179.5) and (511.25,181.4) .. (511.25,183.75) .. controls (511.25,186.1) and (509.35,188) .. (507,188) .. controls (504.65,188) and (502.75,186.1) .. (502.75,183.75) -- cycle ;
%Shape: Ellipse [id:dp5388046174552509] 
\draw  [color={rgb, 255:red, 208; green, 2; blue, 27 }  ,draw opacity=1 ][dash pattern={on 5.63pt off 4.5pt}][line width=1.5]  (271.61,26.2) .. controls (282.66,26.26) and (291.43,58.72) .. (291.21,98.71) .. controls (290.99,138.7) and (281.85,171.06) .. (270.8,171) .. controls (259.76,170.94) and (250.99,138.47) .. (251.21,98.49) .. controls (251.43,58.5) and (260.57,26.14) .. (271.61,26.2) -- cycle ;
%Shape: Ellipse [id:dp04570753595262267] 
\draw  [color={rgb, 255:red, 208; green, 2; blue, 27 }  ,draw opacity=1 ][fill={rgb, 255:red, 0; green, 0; blue, 0 }  ,fill opacity=1 ] (273.06,84.78) .. controls (274.36,84.77) and (275.42,85.27) .. (275.42,85.89) .. controls (275.42,86.52) and (274.36,87.03) .. (273.06,87.04) .. controls (271.75,87.05) and (270.7,86.55) .. (270.7,85.92) .. controls (270.7,85.29) and (271.75,84.78) .. (273.06,84.78) -- cycle ;
%Shape: Ellipse [id:dp06127702957261416] 
\draw  [color={rgb, 255:red, 208; green, 2; blue, 27 }  ,draw opacity=1 ][fill={rgb, 255:red, 0; green, 0; blue, 0 }  ,fill opacity=1 ] (273.06,101.82) .. controls (274.37,101.81) and (275.42,102.31) .. (275.42,102.93) .. controls (275.42,103.56) and (274.37,104.07) .. (273.06,104.08) .. controls (271.76,104.09) and (270.7,103.59) .. (270.7,102.96) .. controls (270.7,102.34) and (271.76,101.82) .. (273.06,101.82) -- cycle ;
%Straight Lines [id:da021561146527679376] 
\draw [color={rgb, 255:red, 208; green, 2; blue, 27 }  ,draw opacity=1 ]   (273.06,85.91) -- (273.06,102.95) ;
%Shape: Ellipse [id:dp7934527104760276] 
\draw  [color={rgb, 255:red, 208; green, 2; blue, 27 }  ,draw opacity=1 ][fill={rgb, 255:red, 0; green, 0; blue, 0 }  ,fill opacity=1 ] (273.07,118.86) .. controls (274.37,118.85) and (275.43,119.35) .. (275.43,119.98) .. controls (275.43,120.6) and (274.37,121.11) .. (273.07,121.12) .. controls (271.76,121.13) and (270.7,120.63) .. (270.7,120) .. controls (270.7,119.38) and (271.76,118.86) .. (273.07,118.86) -- cycle ;
%Straight Lines [id:da7515351842126057] 
\draw [color={rgb, 255:red, 208; green, 2; blue, 27 }  ,draw opacity=1 ]   (273.06,102.95) -- (273.07,119.99) ;
%Shape: Ellipse [id:dp25213042457597024] 
\draw  [color={rgb, 255:red, 208; green, 2; blue, 27 }  ,draw opacity=1 ][fill={rgb, 255:red, 0; green, 0; blue, 0 }  ,fill opacity=1 ] (273.07,135.9) .. controls (274.37,135.89) and (275.43,136.39) .. (275.43,137.02) .. controls (275.43,137.64) and (274.37,138.15) .. (273.07,138.16) .. controls (271.76,138.17) and (270.71,137.67) .. (270.71,137.04) .. controls (270.71,136.42) and (271.76,135.9) .. (273.07,135.9) -- cycle ;
%Straight Lines [id:da6246756459695235] 
\draw [color={rgb, 255:red, 208; green, 2; blue, 27 }  ,draw opacity=1 ]   (273.07,119.99) -- (273.07,137.03) ;
%Shape: Ellipse [id:dp36757439443332085] 
\draw  [color={rgb, 255:red, 208; green, 2; blue, 27 }  ,draw opacity=1 ][fill={rgb, 255:red, 0; green, 0; blue, 0 }  ,fill opacity=1 ] (273.06,68.27) .. controls (274.36,68.26) and (275.42,68.76) .. (275.42,69.39) .. controls (275.42,70.01) and (274.36,70.52) .. (273.06,70.53) .. controls (271.75,70.54) and (270.69,70.04) .. (270.69,69.41) .. controls (270.69,68.79) and (271.75,68.27) .. (273.06,68.27) -- cycle ;
%Straight Lines [id:da7208244174043381] 
\draw [color={rgb, 255:red, 208; green, 2; blue, 27 }  ,draw opacity=1 ]   (273.06,69.4) -- (273.06,86.44) ;
%Straight Lines [id:da14691373846103117] 
\draw [color={rgb, 255:red, 208; green, 2; blue, 27 }  ,draw opacity=1 ]   (273.05,57.15) -- (273.06,74.19) ;
%Straight Lines [id:da3062692673342321] 
\draw [color={rgb, 255:red, 208; green, 2; blue, 27 }  ,draw opacity=1 ]   (273.07,132.77) -- (273.07,149.81) ;
%Shape: Ellipse [id:dp3116730902761943] 
\draw  [color={rgb, 255:red, 208; green, 2; blue, 27 }  ,draw opacity=1 ][fill={rgb, 255:red, 0; green, 0; blue, 0 }  ,fill opacity=1 ] (273.05,54.16) .. controls (274.36,54.15) and (275.42,54.65) .. (275.42,55.27) .. controls (275.42,55.9) and (274.36,56.41) .. (273.05,56.42) .. controls (271.75,56.43) and (270.69,55.93) .. (270.69,55.3) .. controls (270.69,54.68) and (271.75,54.16) .. (273.05,54.16) -- cycle ;
%Straight Lines [id:da24809981185370333] 
\draw [color={rgb, 255:red, 208; green, 2; blue, 27 }  ,draw opacity=1 ]   (273.05,51.23) -- (273.06,68.27) ;
%Shape: Ellipse [id:dp17094667105928762] 
\draw  [color={rgb, 255:red, 208; green, 2; blue, 27 }  ,draw opacity=1 ][dash pattern={on 5.63pt off 4.5pt}][line width=1.5]  (332.14,25.42) .. controls (343.19,25.49) and (352.04,43.34) .. (351.92,65.3) .. controls (351.8,87.27) and (342.75,105.02) .. (331.7,104.96) .. controls (320.65,104.9) and (311.8,87.04) .. (311.92,65.08) .. controls (312.04,43.12) and (321.1,25.36) .. (332.14,25.42) -- cycle ;
%Shape: Ellipse [id:dp8353865226167692] 
\draw  [color={rgb, 255:red, 208; green, 2; blue, 27 }  ,draw opacity=1 ][fill={rgb, 255:red, 0; green, 0; blue, 0 }  ,fill opacity=1 ] (334.06,86.78) .. controls (335.36,86.77) and (336.42,87.27) .. (336.42,87.89) .. controls (336.42,88.52) and (335.36,89.03) .. (334.06,89.04) .. controls (332.75,89.05) and (331.7,88.55) .. (331.7,87.92) .. controls (331.7,87.29) and (332.75,86.78) .. (334.06,86.78) -- cycle ;
%Shape: Ellipse [id:dp4658830805504919] 
\draw  [color={rgb, 255:red, 208; green, 2; blue, 27 }  ,draw opacity=1 ][fill={rgb, 255:red, 0; green, 0; blue, 0 }  ,fill opacity=1 ] (334.06,103.82) .. controls (335.37,103.81) and (336.42,104.31) .. (336.42,104.93) .. controls (336.42,105.56) and (335.37,106.07) .. (334.06,106.08) .. controls (332.76,106.09) and (331.7,105.59) .. (331.7,104.96) .. controls (331.7,104.34) and (332.76,103.82) .. (334.06,103.82) -- cycle ;
%Straight Lines [id:da2848875422237974] 
\draw [color={rgb, 255:red, 208; green, 2; blue, 27 }  ,draw opacity=1 ]   (334.06,87.91) -- (334.06,104.95) ;
%Shape: Ellipse [id:dp9192739524876903] 
\draw  [color={rgb, 255:red, 208; green, 2; blue, 27 }  ,draw opacity=1 ][fill={rgb, 255:red, 0; green, 0; blue, 0 }  ,fill opacity=1 ] (334.07,120.86) .. controls (335.37,120.85) and (336.43,121.35) .. (336.43,121.98) .. controls (336.43,122.6) and (335.37,123.11) .. (334.07,123.12) .. controls (332.76,123.13) and (331.7,122.63) .. (331.7,122) .. controls (331.7,121.38) and (332.76,120.86) .. (334.07,120.86) -- cycle ;
%Straight Lines [id:da22607131041232742] 
\draw [color={rgb, 255:red, 208; green, 2; blue, 27 }  ,draw opacity=1 ]   (334.06,104.95) -- (334.07,121.99) ;
%Shape: Ellipse [id:dp012454965024634057] 
\draw  [color={rgb, 255:red, 208; green, 2; blue, 27 }  ,draw opacity=1 ][fill={rgb, 255:red, 0; green, 0; blue, 0 }  ,fill opacity=1 ] (334.07,137.9) .. controls (335.37,137.89) and (336.43,138.39) .. (336.43,139.02) .. controls (336.43,139.64) and (335.37,140.15) .. (334.07,140.16) .. controls (332.76,140.17) and (331.71,139.67) .. (331.71,139.04) .. controls (331.71,138.42) and (332.76,137.9) .. (334.07,137.9) -- cycle ;
%Straight Lines [id:da3032678921246873] 
\draw [color={rgb, 255:red, 208; green, 2; blue, 27 }  ,draw opacity=1 ]   (334.07,121.99) -- (334.07,139.03) ;
%Shape: Ellipse [id:dp4068640650854174] 
\draw  [color={rgb, 255:red, 208; green, 2; blue, 27 }  ,draw opacity=1 ][fill={rgb, 255:red, 0; green, 0; blue, 0 }  ,fill opacity=1 ] (334.06,70.27) .. controls (335.36,70.26) and (336.42,70.76) .. (336.42,71.39) .. controls (336.42,72.01) and (335.36,72.52) .. (334.06,72.53) .. controls (332.75,72.54) and (331.69,72.04) .. (331.69,71.41) .. controls (331.69,70.79) and (332.75,70.27) .. (334.06,70.27) -- cycle ;
%Straight Lines [id:da7666174032707697] 
\draw [color={rgb, 255:red, 208; green, 2; blue, 27 }  ,draw opacity=1 ]   (334.06,71.4) -- (334.06,88.44) ;
%Straight Lines [id:da4352551414731355] 
\draw [color={rgb, 255:red, 208; green, 2; blue, 27 }  ,draw opacity=1 ]   (334.05,59.15) -- (334.06,76.19) ;
%Straight Lines [id:da6133651985587508] 
\draw [color={rgb, 255:red, 208; green, 2; blue, 27 }  ,draw opacity=1 ]   (334.07,134.77) -- (334.07,151.81) ;
%Shape: Ellipse [id:dp4313303958673753] 
\draw  [color={rgb, 255:red, 208; green, 2; blue, 27 }  ,draw opacity=1 ][fill={rgb, 255:red, 0; green, 0; blue, 0 }  ,fill opacity=1 ] (334.05,56.16) .. controls (335.36,56.15) and (336.42,56.65) .. (336.42,57.27) .. controls (336.42,57.9) and (335.36,58.41) .. (334.05,58.42) .. controls (332.75,58.43) and (331.69,57.93) .. (331.69,57.3) .. controls (331.69,56.68) and (332.75,56.16) .. (334.05,56.16) -- cycle ;
%Straight Lines [id:da07374714387564452] 
\draw [color={rgb, 255:red, 208; green, 2; blue, 27 }  ,draw opacity=1 ]   (334.05,53.23) -- (334.06,70.27) ;
%Shape: Ellipse [id:dp07469787213448087] 
\draw  [color={rgb, 255:red, 208; green, 2; blue, 27 }  ,draw opacity=1 ][dash pattern={on 5.63pt off 4.5pt}][line width=1.5]  (507.12,131.22) .. controls (512.87,131.25) and (517.5,135.61) .. (517.47,140.95) .. controls (517.44,146.3) and (512.76,150.61) .. (507.01,150.57) .. controls (501.27,150.54) and (496.63,146.18) .. (496.66,140.84) .. controls (496.69,135.49) and (501.38,131.19) .. (507.12,131.22) -- cycle ;
%Shape: Ellipse [id:dp7512915285094866] 
\draw  [color={rgb, 255:red, 208; green, 2; blue, 27 }  ,draw opacity=1 ][fill={rgb, 255:red, 0; green, 0; blue, 0 }  ,fill opacity=1 ] (507.06,89.78) .. controls (508.36,89.77) and (509.42,90.27) .. (509.42,90.89) .. controls (509.42,91.52) and (508.36,92.03) .. (507.06,92.04) .. controls (505.75,92.05) and (504.7,91.55) .. (504.7,90.92) .. controls (504.7,90.29) and (505.75,89.78) .. (507.06,89.78) -- cycle ;
%Shape: Ellipse [id:dp34422714045819147] 
\draw  [color={rgb, 255:red, 208; green, 2; blue, 27 }  ,draw opacity=1 ][fill={rgb, 255:red, 0; green, 0; blue, 0 }  ,fill opacity=1 ] (507.06,106.82) .. controls (508.37,106.81) and (509.42,107.31) .. (509.42,107.93) .. controls (509.42,108.56) and (508.37,109.07) .. (507.06,109.08) .. controls (505.76,109.09) and (504.7,108.59) .. (504.7,107.96) .. controls (504.7,107.34) and (505.76,106.82) .. (507.06,106.82) -- cycle ;
%Straight Lines [id:da6936658815825126] 
\draw [color={rgb, 255:red, 208; green, 2; blue, 27 }  ,draw opacity=1 ]   (507.06,90.91) -- (507.06,107.95) ;
%Shape: Ellipse [id:dp08951353997960831] 
\draw  [color={rgb, 255:red, 208; green, 2; blue, 27 }  ,draw opacity=1 ][fill={rgb, 255:red, 0; green, 0; blue, 0 }  ,fill opacity=1 ] (507.07,123.86) .. controls (508.37,123.85) and (509.43,124.35) .. (509.43,124.98) .. controls (509.43,125.6) and (508.37,126.11) .. (507.07,126.12) .. controls (505.76,126.13) and (504.7,125.63) .. (504.7,125) .. controls (504.7,124.38) and (505.76,123.86) .. (507.07,123.86) -- cycle ;
%Straight Lines [id:da34515523870147935] 
\draw [color={rgb, 255:red, 208; green, 2; blue, 27 }  ,draw opacity=1 ]   (507.06,107.95) -- (507.07,124.99) ;
%Shape: Ellipse [id:dp995579682479986] 
\draw  [color={rgb, 255:red, 208; green, 2; blue, 27 }  ,draw opacity=1 ][fill={rgb, 255:red, 0; green, 0; blue, 0 }  ,fill opacity=1 ] (507.07,140.9) .. controls (508.37,140.89) and (509.43,141.39) .. (509.43,142.02) .. controls (509.43,142.64) and (508.37,143.15) .. (507.07,143.16) .. controls (505.76,143.17) and (504.71,142.67) .. (504.71,142.04) .. controls (504.71,141.42) and (505.76,140.9) .. (507.07,140.9) -- cycle ;
%Straight Lines [id:da274034008684084] 
\draw [color={rgb, 255:red, 208; green, 2; blue, 27 }  ,draw opacity=1 ]   (507.07,124.99) -- (507.07,142.03) ;
%Shape: Ellipse [id:dp2131832309135704] 
\draw  [color={rgb, 255:red, 208; green, 2; blue, 27 }  ,draw opacity=1 ][fill={rgb, 255:red, 0; green, 0; blue, 0 }  ,fill opacity=1 ] (507.06,73.27) .. controls (508.36,73.26) and (509.42,73.76) .. (509.42,74.39) .. controls (509.42,75.01) and (508.36,75.52) .. (507.06,75.53) .. controls (505.75,75.54) and (504.69,75.04) .. (504.69,74.41) .. controls (504.69,73.79) and (505.75,73.27) .. (507.06,73.27) -- cycle ;
%Straight Lines [id:da3576193057105205] 
\draw [color={rgb, 255:red, 208; green, 2; blue, 27 }  ,draw opacity=1 ]   (507.06,74.4) -- (507.06,91.44) ;
%Straight Lines [id:da44450153262950143] 
\draw [color={rgb, 255:red, 208; green, 2; blue, 27 }  ,draw opacity=1 ]   (507.05,62.15) -- (507.06,79.19) ;
%Straight Lines [id:da6737892060687923] 
\draw [color={rgb, 255:red, 208; green, 2; blue, 27 }  ,draw opacity=1 ]   (507.07,137.77) -- (507.07,154.81) ;
%Shape: Ellipse [id:dp439524015970865] 
\draw  [color={rgb, 255:red, 208; green, 2; blue, 27 }  ,draw opacity=1 ][fill={rgb, 255:red, 0; green, 0; blue, 0 }  ,fill opacity=1 ] (507.05,59.16) .. controls (508.36,59.15) and (509.42,59.65) .. (509.42,60.27) .. controls (509.42,60.9) and (508.36,61.41) .. (507.05,61.42) .. controls (505.75,61.43) and (504.69,60.93) .. (504.69,60.3) .. controls (504.69,59.68) and (505.75,59.16) .. (507.05,59.16) -- cycle ;
%Straight Lines [id:da3605326932209494] 
\draw [color={rgb, 255:red, 208; green, 2; blue, 27 }  ,draw opacity=1 ]   (507.05,56.23) -- (507.06,73.27) ;
%Shape: Ellipse [id:dp9179617736909164] 
\draw  [color={rgb, 255:red, 208; green, 2; blue, 27 }  ,draw opacity=1 ][dash pattern={on 5.63pt off 4.5pt}][line width=1.5]  (152.75,80.24) .. controls (158.5,80.27) and (163.13,84.63) .. (163.1,89.98) .. controls (163.07,95.32) and (158.39,99.63) .. (152.64,99.6) .. controls (146.9,99.57) and (142.26,95.21) .. (142.29,89.86) .. controls (142.32,84.52) and (147,80.21) .. (152.75,80.24) -- cycle ;
%Shape: Ellipse [id:dp4317711265894655] 
\draw  [color={rgb, 255:red, 208; green, 2; blue, 27 }  ,draw opacity=1 ][fill={rgb, 255:red, 0; green, 0; blue, 0 }  ,fill opacity=1 ] (155.06,88.78) .. controls (156.36,88.77) and (157.42,89.27) .. (157.42,89.89) .. controls (157.42,90.52) and (156.36,91.03) .. (155.06,91.04) .. controls (153.75,91.05) and (152.7,90.55) .. (152.7,89.92) .. controls (152.7,89.29) and (153.75,88.78) .. (155.06,88.78) -- cycle ;
%Shape: Ellipse [id:dp713333034385384] 
\draw  [color={rgb, 255:red, 208; green, 2; blue, 27 }  ,draw opacity=1 ][fill={rgb, 255:red, 0; green, 0; blue, 0 }  ,fill opacity=1 ] (155.06,105.82) .. controls (156.37,105.81) and (157.42,106.31) .. (157.42,106.93) .. controls (157.42,107.56) and (156.37,108.07) .. (155.06,108.08) .. controls (153.76,108.09) and (152.7,107.59) .. (152.7,106.96) .. controls (152.7,106.34) and (153.76,105.82) .. (155.06,105.82) -- cycle ;
%Straight Lines [id:da4083472120645292] 
\draw [color={rgb, 255:red, 208; green, 2; blue, 27 }  ,draw opacity=1 ]   (155.06,89.91) -- (155.06,106.95) ;
%Shape: Ellipse [id:dp550092616480915] 
\draw  [color={rgb, 255:red, 208; green, 2; blue, 27 }  ,draw opacity=1 ][fill={rgb, 255:red, 0; green, 0; blue, 0 }  ,fill opacity=1 ] (155.07,122.86) .. controls (156.37,122.85) and (157.43,123.35) .. (157.43,123.98) .. controls (157.43,124.6) and (156.37,125.11) .. (155.07,125.12) .. controls (153.76,125.13) and (152.7,124.63) .. (152.7,124) .. controls (152.7,123.38) and (153.76,122.86) .. (155.07,122.86) -- cycle ;
%Straight Lines [id:da8343489401963329] 
\draw [color={rgb, 255:red, 208; green, 2; blue, 27 }  ,draw opacity=1 ]   (155.06,106.95) -- (155.07,123.99) ;
%Shape: Ellipse [id:dp5184436878147374] 
\draw  [color={rgb, 255:red, 208; green, 2; blue, 27 }  ,draw opacity=1 ][fill={rgb, 255:red, 0; green, 0; blue, 0 }  ,fill opacity=1 ] (155.07,139.9) .. controls (156.37,139.89) and (157.43,140.39) .. (157.43,141.02) .. controls (157.43,141.64) and (156.37,142.15) .. (155.07,142.16) .. controls (153.76,142.17) and (152.71,141.67) .. (152.71,141.04) .. controls (152.71,140.42) and (153.76,139.9) .. (155.07,139.9) -- cycle ;
%Straight Lines [id:da5113401207060717] 
\draw [color={rgb, 255:red, 208; green, 2; blue, 27 }  ,draw opacity=1 ]   (155.07,123.99) -- (155.07,141.03) ;
%Shape: Ellipse [id:dp5275498038719495] 
\draw  [color={rgb, 255:red, 208; green, 2; blue, 27 }  ,draw opacity=1 ][fill={rgb, 255:red, 0; green, 0; blue, 0 }  ,fill opacity=1 ] (155.06,72.27) .. controls (156.36,72.26) and (157.42,72.76) .. (157.42,73.39) .. controls (157.42,74.01) and (156.36,74.52) .. (155.06,74.53) .. controls (153.75,74.54) and (152.69,74.04) .. (152.69,73.41) .. controls (152.69,72.79) and (153.75,72.27) .. (155.06,72.27) -- cycle ;
%Straight Lines [id:da45808880808160213] 
\draw [color={rgb, 255:red, 208; green, 2; blue, 27 }  ,draw opacity=1 ]   (155.06,73.4) -- (155.06,90.44) ;
%Straight Lines [id:da43650241101234055] 
\draw [color={rgb, 255:red, 208; green, 2; blue, 27 }  ,draw opacity=1 ]   (155.05,61.15) -- (155.06,78.19) ;
%Straight Lines [id:da6001002208139974] 
\draw [color={rgb, 255:red, 208; green, 2; blue, 27 }  ,draw opacity=1 ]   (155.07,136.77) -- (155.07,153.81) ;
%Shape: Ellipse [id:dp41925151520057236] 
\draw  [color={rgb, 255:red, 208; green, 2; blue, 27 }  ,draw opacity=1 ][fill={rgb, 255:red, 0; green, 0; blue, 0 }  ,fill opacity=1 ] (155.05,58.16) .. controls (156.36,58.15) and (157.42,58.65) .. (157.42,59.27) .. controls (157.42,59.9) and (156.36,60.41) .. (155.05,60.42) .. controls (153.75,60.43) and (152.69,59.93) .. (152.69,59.3) .. controls (152.69,58.68) and (153.75,58.16) .. (155.05,58.16) -- cycle ;
%Straight Lines [id:da448778911457817] 
\draw [color={rgb, 255:red, 208; green, 2; blue, 27 }  ,draw opacity=1 ]   (155.05,55.23) -- (155.06,72.27) ;
%Shape: Ellipse [id:dp3710219237530211] 
\draw  [color={rgb, 255:red, 208; green, 2; blue, 27 }  ,draw opacity=1 ][fill={rgb, 255:red, 0; green, 0; blue, 0 }  ,fill opacity=1 ] (208.06,87.78) .. controls (209.36,87.77) and (210.42,88.27) .. (210.42,88.89) .. controls (210.42,89.52) and (209.36,90.03) .. (208.06,90.04) .. controls (206.75,90.05) and (205.7,89.55) .. (205.7,88.92) .. controls (205.7,88.29) and (206.75,87.78) .. (208.06,87.78) -- cycle ;
%Shape: Ellipse [id:dp5797568599180717] 
\draw  [color={rgb, 255:red, 208; green, 2; blue, 27 }  ,draw opacity=1 ][fill={rgb, 255:red, 0; green, 0; blue, 0 }  ,fill opacity=1 ] (208.06,104.82) .. controls (209.37,104.81) and (210.42,105.31) .. (210.42,105.93) .. controls (210.42,106.56) and (209.37,107.07) .. (208.06,107.08) .. controls (206.76,107.09) and (205.7,106.59) .. (205.7,105.96) .. controls (205.7,105.34) and (206.76,104.82) .. (208.06,104.82) -- cycle ;
%Straight Lines [id:da5826784818122462] 
\draw [color={rgb, 255:red, 208; green, 2; blue, 27 }  ,draw opacity=1 ]   (208.06,88.91) -- (208.06,105.95) ;
%Shape: Ellipse [id:dp9036408542117008] 
\draw  [color={rgb, 255:red, 208; green, 2; blue, 27 }  ,draw opacity=1 ][fill={rgb, 255:red, 0; green, 0; blue, 0 }  ,fill opacity=1 ] (208.07,121.86) .. controls (209.37,121.85) and (210.43,122.35) .. (210.43,122.98) .. controls (210.43,123.6) and (209.37,124.11) .. (208.07,124.12) .. controls (206.76,124.13) and (205.7,123.63) .. (205.7,123) .. controls (205.7,122.38) and (206.76,121.86) .. (208.07,121.86) -- cycle ;
%Straight Lines [id:da5939141411610338] 
\draw [color={rgb, 255:red, 208; green, 2; blue, 27 }  ,draw opacity=1 ]   (208.06,105.95) -- (208.07,122.99) ;
%Shape: Ellipse [id:dp6820035957647845] 
\draw  [color={rgb, 255:red, 208; green, 2; blue, 27 }  ,draw opacity=1 ][fill={rgb, 255:red, 0; green, 0; blue, 0 }  ,fill opacity=1 ] (208.07,138.9) .. controls (209.37,138.89) and (210.43,139.39) .. (210.43,140.02) .. controls (210.43,140.64) and (209.37,141.15) .. (208.07,141.16) .. controls (206.76,141.17) and (205.71,140.67) .. (205.71,140.04) .. controls (205.71,139.42) and (206.76,138.9) .. (208.07,138.9) -- cycle ;
%Straight Lines [id:da651772208642059] 
\draw [color={rgb, 255:red, 208; green, 2; blue, 27 }  ,draw opacity=1 ]   (208.07,122.99) -- (208.07,140.03) ;
%Shape: Ellipse [id:dp4889733145830376] 
\draw  [color={rgb, 255:red, 208; green, 2; blue, 27 }  ,draw opacity=1 ][fill={rgb, 255:red, 0; green, 0; blue, 0 }  ,fill opacity=1 ] (208.06,71.27) .. controls (209.36,71.26) and (210.42,71.76) .. (210.42,72.39) .. controls (210.42,73.01) and (209.36,73.52) .. (208.06,73.53) .. controls (206.75,73.54) and (205.69,73.04) .. (205.69,72.41) .. controls (205.69,71.79) and (206.75,71.27) .. (208.06,71.27) -- cycle ;
%Straight Lines [id:da5760321954756246] 
\draw [color={rgb, 255:red, 208; green, 2; blue, 27 }  ,draw opacity=1 ]   (208.06,72.4) -- (208.06,89.44) ;
%Straight Lines [id:da929928422127428] 
\draw [color={rgb, 255:red, 208; green, 2; blue, 27 }  ,draw opacity=1 ]   (208.05,60.15) -- (208.06,77.19) ;
%Straight Lines [id:da9025771145068626] 
\draw [color={rgb, 255:red, 208; green, 2; blue, 27 }  ,draw opacity=1 ]   (208.07,135.77) -- (208.07,152.81) ;
%Shape: Ellipse [id:dp9222273701299334] 
\draw  [color={rgb, 255:red, 208; green, 2; blue, 27 }  ,draw opacity=1 ][fill={rgb, 255:red, 0; green, 0; blue, 0 }  ,fill opacity=1 ] (208.05,57.16) .. controls (209.36,57.15) and (210.42,57.65) .. (210.42,58.27) .. controls (210.42,58.9) and (209.36,59.41) .. (208.05,59.42) .. controls (206.75,59.43) and (205.69,58.93) .. (205.69,58.3) .. controls (205.69,57.68) and (206.75,57.16) .. (208.05,57.16) -- cycle ;
%Straight Lines [id:da362336914917747] 
\draw [color={rgb, 255:red, 208; green, 2; blue, 27 }  ,draw opacity=1 ]   (208.05,54.23) -- (208.06,71.27) ;
%Shape: Ellipse [id:dp16519243440641584] 
\draw  [color={rgb, 255:red, 208; green, 2; blue, 27 }  ,draw opacity=1 ][fill={rgb, 255:red, 0; green, 0; blue, 0 }  ,fill opacity=1 ] (399.06,87.78) .. controls (400.36,87.77) and (401.42,88.27) .. (401.42,88.89) .. controls (401.42,89.52) and (400.36,90.03) .. (399.06,90.04) .. controls (397.75,90.05) and (396.7,89.55) .. (396.7,88.92) .. controls (396.7,88.29) and (397.75,87.78) .. (399.06,87.78) -- cycle ;
%Shape: Ellipse [id:dp8569260807969875] 
\draw  [color={rgb, 255:red, 208; green, 2; blue, 27 }  ,draw opacity=1 ][fill={rgb, 255:red, 0; green, 0; blue, 0 }  ,fill opacity=1 ] (399.06,104.82) .. controls (400.37,104.81) and (401.42,105.31) .. (401.42,105.93) .. controls (401.42,106.56) and (400.37,107.07) .. (399.06,107.08) .. controls (397.76,107.09) and (396.7,106.59) .. (396.7,105.96) .. controls (396.7,105.34) and (397.76,104.82) .. (399.06,104.82) -- cycle ;
%Straight Lines [id:da8485338874969574] 
\draw [color={rgb, 255:red, 208; green, 2; blue, 27 }  ,draw opacity=1 ]   (399.06,88.91) -- (399.06,105.95) ;
%Shape: Ellipse [id:dp35516802619223187] 
\draw  [color={rgb, 255:red, 208; green, 2; blue, 27 }  ,draw opacity=1 ][fill={rgb, 255:red, 0; green, 0; blue, 0 }  ,fill opacity=1 ] (399.07,121.86) .. controls (400.37,121.85) and (401.43,122.35) .. (401.43,122.98) .. controls (401.43,123.6) and (400.37,124.11) .. (399.07,124.12) .. controls (397.76,124.13) and (396.7,123.63) .. (396.7,123) .. controls (396.7,122.38) and (397.76,121.86) .. (399.07,121.86) -- cycle ;
%Straight Lines [id:da09306711891020003] 
\draw [color={rgb, 255:red, 208; green, 2; blue, 27 }  ,draw opacity=1 ]   (399.06,105.95) -- (399.07,122.99) ;
%Shape: Ellipse [id:dp2660375876082033] 
\draw  [color={rgb, 255:red, 208; green, 2; blue, 27 }  ,draw opacity=1 ][fill={rgb, 255:red, 0; green, 0; blue, 0 }  ,fill opacity=1 ] (399.07,138.9) .. controls (400.37,138.89) and (401.43,139.39) .. (401.43,140.02) .. controls (401.43,140.64) and (400.37,141.15) .. (399.07,141.16) .. controls (397.76,141.17) and (396.71,140.67) .. (396.71,140.04) .. controls (396.71,139.42) and (397.76,138.9) .. (399.07,138.9) -- cycle ;
%Straight Lines [id:da9335350679252464] 
\draw [color={rgb, 255:red, 208; green, 2; blue, 27 }  ,draw opacity=1 ]   (399.07,122.99) -- (399.07,140.03) ;
%Shape: Ellipse [id:dp6064313882242107] 
\draw  [color={rgb, 255:red, 208; green, 2; blue, 27 }  ,draw opacity=1 ][fill={rgb, 255:red, 0; green, 0; blue, 0 }  ,fill opacity=1 ] (399.06,71.27) .. controls (400.36,71.26) and (401.42,71.76) .. (401.42,72.39) .. controls (401.42,73.01) and (400.36,73.52) .. (399.06,73.53) .. controls (397.75,73.54) and (396.69,73.04) .. (396.69,72.41) .. controls (396.69,71.79) and (397.75,71.27) .. (399.06,71.27) -- cycle ;
%Straight Lines [id:da8348738702357011] 
\draw [color={rgb, 255:red, 208; green, 2; blue, 27 }  ,draw opacity=1 ]   (399.06,72.4) -- (399.06,89.44) ;
%Straight Lines [id:da17198996362812524] 
\draw [color={rgb, 255:red, 208; green, 2; blue, 27 }  ,draw opacity=1 ]   (399.05,60.15) -- (399.06,77.19) ;
%Straight Lines [id:da8716836155419337] 
\draw [color={rgb, 255:red, 208; green, 2; blue, 27 }  ,draw opacity=1 ]   (399.07,135.77) -- (399.07,152.81) ;
%Shape: Ellipse [id:dp7998372613864952] 
\draw  [color={rgb, 255:red, 208; green, 2; blue, 27 }  ,draw opacity=1 ][fill={rgb, 255:red, 0; green, 0; blue, 0 }  ,fill opacity=1 ] (399.05,57.16) .. controls (400.36,57.15) and (401.42,57.65) .. (401.42,58.27) .. controls (401.42,58.9) and (400.36,59.41) .. (399.05,59.42) .. controls (397.75,59.43) and (396.69,58.93) .. (396.69,58.3) .. controls (396.69,57.68) and (397.75,57.16) .. (399.05,57.16) -- cycle ;
%Straight Lines [id:da6061186235421365] 
\draw [color={rgb, 255:red, 208; green, 2; blue, 27 }  ,draw opacity=1 ]   (399.05,54.23) -- (399.06,71.27) ;
%Shape: Ellipse [id:dp1494342346939096] 
\draw  [color={rgb, 255:red, 208; green, 2; blue, 27 }  ,draw opacity=1 ][fill={rgb, 255:red, 0; green, 0; blue, 0 }  ,fill opacity=1 ] (460.06,88.78) .. controls (461.36,88.77) and (462.42,89.27) .. (462.42,89.89) .. controls (462.42,90.52) and (461.36,91.03) .. (460.06,91.04) .. controls (458.75,91.05) and (457.7,90.55) .. (457.7,89.92) .. controls (457.7,89.29) and (458.75,88.78) .. (460.06,88.78) -- cycle ;
%Shape: Ellipse [id:dp02587659642503537] 
\draw  [color={rgb, 255:red, 208; green, 2; blue, 27 }  ,draw opacity=1 ][fill={rgb, 255:red, 0; green, 0; blue, 0 }  ,fill opacity=1 ] (460.06,105.82) .. controls (461.37,105.81) and (462.42,106.31) .. (462.42,106.93) .. controls (462.42,107.56) and (461.37,108.07) .. (460.06,108.08) .. controls (458.76,108.09) and (457.7,107.59) .. (457.7,106.96) .. controls (457.7,106.34) and (458.76,105.82) .. (460.06,105.82) -- cycle ;
%Straight Lines [id:da5087817473706655] 
\draw [color={rgb, 255:red, 208; green, 2; blue, 27 }  ,draw opacity=1 ]   (460.06,89.91) -- (460.06,106.95) ;
%Shape: Ellipse [id:dp3290973681135345] 
\draw  [color={rgb, 255:red, 208; green, 2; blue, 27 }  ,draw opacity=1 ][fill={rgb, 255:red, 0; green, 0; blue, 0 }  ,fill opacity=1 ] (460.07,122.86) .. controls (461.37,122.85) and (462.43,123.35) .. (462.43,123.98) .. controls (462.43,124.6) and (461.37,125.11) .. (460.07,125.12) .. controls (458.76,125.13) and (457.7,124.63) .. (457.7,124) .. controls (457.7,123.38) and (458.76,122.86) .. (460.07,122.86) -- cycle ;
%Straight Lines [id:da4025519718669739] 
\draw [color={rgb, 255:red, 208; green, 2; blue, 27 }  ,draw opacity=1 ]   (460.06,106.95) -- (460.07,123.99) ;
%Shape: Ellipse [id:dp14115453840475867] 
\draw  [color={rgb, 255:red, 208; green, 2; blue, 27 }  ,draw opacity=1 ][fill={rgb, 255:red, 0; green, 0; blue, 0 }  ,fill opacity=1 ] (460.07,139.9) .. controls (461.37,139.89) and (462.43,140.39) .. (462.43,141.02) .. controls (462.43,141.64) and (461.37,142.15) .. (460.07,142.16) .. controls (458.76,142.17) and (457.71,141.67) .. (457.71,141.04) .. controls (457.71,140.42) and (458.76,139.9) .. (460.07,139.9) -- cycle ;
%Straight Lines [id:da7610331186954356] 
\draw [color={rgb, 255:red, 208; green, 2; blue, 27 }  ,draw opacity=1 ]   (460.07,123.99) -- (460.07,141.03) ;
%Shape: Ellipse [id:dp16645798665771838] 
\draw  [color={rgb, 255:red, 208; green, 2; blue, 27 }  ,draw opacity=1 ][fill={rgb, 255:red, 0; green, 0; blue, 0 }  ,fill opacity=1 ] (460.06,72.27) .. controls (461.36,72.26) and (462.42,72.76) .. (462.42,73.39) .. controls (462.42,74.01) and (461.36,74.52) .. (460.06,74.53) .. controls (458.75,74.54) and (457.69,74.04) .. (457.69,73.41) .. controls (457.69,72.79) and (458.75,72.27) .. (460.06,72.27) -- cycle ;
%Straight Lines [id:da7102032143165505] 
\draw [color={rgb, 255:red, 208; green, 2; blue, 27 }  ,draw opacity=1 ]   (460.06,73.4) -- (460.06,90.44) ;
%Straight Lines [id:da8726857007623475] 
\draw [color={rgb, 255:red, 208; green, 2; blue, 27 }  ,draw opacity=1 ]   (460.05,61.15) -- (460.06,78.19) ;
%Straight Lines [id:da45901311894410934] 
\draw [color={rgb, 255:red, 208; green, 2; blue, 27 }  ,draw opacity=1 ]   (460.07,136.77) -- (460.07,153.81) ;
%Shape: Ellipse [id:dp205456688274672] 
\draw  [color={rgb, 255:red, 208; green, 2; blue, 27 }  ,draw opacity=1 ][fill={rgb, 255:red, 0; green, 0; blue, 0 }  ,fill opacity=1 ] (460.05,58.16) .. controls (461.36,58.15) and (462.42,58.65) .. (462.42,59.27) .. controls (462.42,59.9) and (461.36,60.41) .. (460.05,60.42) .. controls (458.75,60.43) and (457.69,59.93) .. (457.69,59.3) .. controls (457.69,58.68) and (458.75,58.16) .. (460.05,58.16) -- cycle ;
%Straight Lines [id:da8424298272884035] 
\draw [color={rgb, 255:red, 208; green, 2; blue, 27 }  ,draw opacity=1 ]   (460.05,55.23) -- (460.06,72.27) ;

% Text Node
\draw (258,194.4) node [anchor=north west][inner sep=0.75pt]    {$-1$};
% Text Node
\draw (329,195.4) node [anchor=north west][inner sep=0.75pt]    {$0$};
% Text Node
\draw (393,195.4) node [anchor=north west][inner sep=0.75pt]    {$1$};
% Text Node
\draw (457,195.4) node [anchor=north west][inner sep=0.75pt]    {$2$};
% Text Node
\draw (196,193.4) node [anchor=north west][inner sep=0.75pt]    {$-2$};
% Text Node
\draw (115,168.4) node [anchor=north west][inner sep=0.75pt]  [font=\large]  {$...$};
% Text Node
\draw (542,168.4) node [anchor=north west][inner sep=0.75pt]  [font=\large]  {$...$};
% Text Node
\draw (142,194.4) node [anchor=north west][inner sep=0.75pt]    {$-3$};
% Text Node
\draw (500,195.4) node [anchor=north west][inner sep=0.75pt]    {$3$};
% Text Node
\draw (270.47,75.06) node [anchor=north west][inner sep=0.75pt]  [rotate=-89.72]  {$-1$};
% Text Node
\draw (269.93,96.9) node [anchor=north west][inner sep=0.75pt]  [rotate=-89.72]  {$0$};
% Text Node
\draw (269.94,113.94) node [anchor=north west][inner sep=0.75pt]  [rotate=-89.72]  {$1$};
% Text Node
\draw (269.94,130.98) node [anchor=north west][inner sep=0.75pt]  [rotate=-89.72]  {$2$};
% Text Node
\draw (271.02,58.55) node [anchor=north west][inner sep=0.75pt]  [rotate=-89.72]  {$-2$};
% Text Node
\draw (288.72,147.52) node [anchor=north west][inner sep=0.75pt]  [font=\large,rotate=-91.82]  {$...$};
% Text Node
\draw (287.04,32.38) node [anchor=north west][inner sep=0.75pt]  [font=\large,rotate=-91.82]  {$...$};
% Text Node
\draw (331.47,77.06) node [anchor=north west][inner sep=0.75pt]  [rotate=-89.72]  {$-1$};
% Text Node
\draw (330.93,98.9) node [anchor=north west][inner sep=0.75pt]  [rotate=-89.72]  {$0$};
% Text Node
\draw (330.94,115.94) node [anchor=north west][inner sep=0.75pt]  [rotate=-89.72]  {$1$};
% Text Node
\draw (330.94,132.98) node [anchor=north west][inner sep=0.75pt]  [rotate=-89.72]  {$2$};
% Text Node
\draw (332.02,60.55) node [anchor=north west][inner sep=0.75pt]  [rotate=-89.72]  {$-2$};
% Text Node
\draw (349.72,149.52) node [anchor=north west][inner sep=0.75pt]  [font=\large,rotate=-91.82]  {$...$};
% Text Node
\draw (348.04,34.38) node [anchor=north west][inner sep=0.75pt]  [font=\large,rotate=-91.82]  {$...$};
% Text Node
\draw (504.47,80.06) node [anchor=north west][inner sep=0.75pt]  [rotate=-89.72]  {$-1$};
% Text Node
\draw (503.93,101.9) node [anchor=north west][inner sep=0.75pt]  [rotate=-89.72]  {$0$};
% Text Node
\draw (503.94,118.94) node [anchor=north west][inner sep=0.75pt]  [rotate=-89.72]  {$1$};
% Text Node
\draw (503.94,135.98) node [anchor=north west][inner sep=0.75pt]  [rotate=-89.72]  {$2$};
% Text Node
\draw (505.02,63.55) node [anchor=north west][inner sep=0.75pt]  [rotate=-89.72]  {$-2$};
% Text Node
\draw (522.72,152.52) node [anchor=north west][inner sep=0.75pt]  [font=\large,rotate=-91.82]  {$...$};
% Text Node
\draw (521.04,37.38) node [anchor=north west][inner sep=0.75pt]  [font=\large,rotate=-91.82]  {$...$};
% Text Node
\draw (152.47,79.06) node [anchor=north west][inner sep=0.75pt]  [rotate=-89.72]  {$-1$};
% Text Node
\draw (151.93,100.9) node [anchor=north west][inner sep=0.75pt]  [rotate=-89.72]  {$0$};
% Text Node
\draw (151.94,117.94) node [anchor=north west][inner sep=0.75pt]  [rotate=-89.72]  {$1$};
% Text Node
\draw (151.94,134.98) node [anchor=north west][inner sep=0.75pt]  [rotate=-89.72]  {$2$};
% Text Node
\draw (153.02,62.55) node [anchor=north west][inner sep=0.75pt]  [rotate=-89.72]  {$-2$};
% Text Node
\draw (170.72,151.52) node [anchor=north west][inner sep=0.75pt]  [font=\large,rotate=-91.82]  {$...$};
% Text Node
\draw (169.04,36.38) node [anchor=north west][inner sep=0.75pt]  [font=\large,rotate=-91.82]  {$...$};
% Text Node
\draw (205.47,78.06) node [anchor=north west][inner sep=0.75pt]  [rotate=-89.72]  {$-1$};
% Text Node
\draw (204.93,99.9) node [anchor=north west][inner sep=0.75pt]  [rotate=-89.72]  {$0$};
% Text Node
\draw (204.94,116.94) node [anchor=north west][inner sep=0.75pt]  [rotate=-89.72]  {$1$};
% Text Node
\draw (204.94,133.98) node [anchor=north west][inner sep=0.75pt]  [rotate=-89.72]  {$2$};
% Text Node
\draw (206.02,61.55) node [anchor=north west][inner sep=0.75pt]  [rotate=-89.72]  {$-2$};
% Text Node
\draw (223.72,150.52) node [anchor=north west][inner sep=0.75pt]  [font=\large,rotate=-91.82]  {$...$};
% Text Node
\draw (222.04,35.38) node [anchor=north west][inner sep=0.75pt]  [font=\large,rotate=-91.82]  {$...$};
% Text Node
\draw (396.47,78.06) node [anchor=north west][inner sep=0.75pt]  [rotate=-89.72]  {$-1$};
% Text Node
\draw (395.93,99.9) node [anchor=north west][inner sep=0.75pt]  [rotate=-89.72]  {$0$};
% Text Node
\draw (395.94,116.94) node [anchor=north west][inner sep=0.75pt]  [rotate=-89.72]  {$1$};
% Text Node
\draw (395.94,133.98) node [anchor=north west][inner sep=0.75pt]  [rotate=-89.72]  {$2$};
% Text Node
\draw (397.02,61.55) node [anchor=north west][inner sep=0.75pt]  [rotate=-89.72]  {$-2$};
% Text Node
\draw (414.72,150.52) node [anchor=north west][inner sep=0.75pt]  [font=\large,rotate=-91.82]  {$...$};
% Text Node
\draw (413.04,35.38) node [anchor=north west][inner sep=0.75pt]  [font=\large,rotate=-91.82]  {$...$};
% Text Node
\draw (457.47,79.06) node [anchor=north west][inner sep=0.75pt]  [rotate=-89.72]  {$-1$};
% Text Node
\draw (456.93,100.9) node [anchor=north west][inner sep=0.75pt]  [rotate=-89.72]  {$0$};
% Text Node
\draw (456.94,117.94) node [anchor=north west][inner sep=0.75pt]  [rotate=-89.72]  {$1$};
% Text Node
\draw (456.94,134.98) node [anchor=north west][inner sep=0.75pt]  [rotate=-89.72]  {$2$};
% Text Node
\draw (458.02,62.55) node [anchor=north west][inner sep=0.75pt]  [rotate=-89.72]  {$-2$};
% Text Node
\draw (475.72,151.52) node [anchor=north west][inner sep=0.75pt]  [font=\large,rotate=-91.82]  {$...$};
% Text Node
\draw (474.04,36.38) node [anchor=north west][inner sep=0.75pt]  [font=\large,rotate=-91.82]  {$...$};

\end{tikzpicture}
\]
    \caption{Example of an object in $\textbf{C}'_1(W)$ for $W = \Zbb$. Over each integer $i \in \Zbb$ lays a tensor factor of $\mathcal{A}(\Zbb, q_i)$. The dashed ovals indicate the support of these tensor factors.}
    \label{fig:enlargement}
\end{figure}

We are now equipped to prove the following theorem.

\delooping*

\begin{proof}
    Again, due to the obstructions on the $\pi_0$-level, we wish to work with certain enlargements of our original categories. Just as how Azumaya $R$-algebras range over all possible tensor factors of matrix algebras at a single point, we wish to place admissible tensor factors of $\mathcal{A}(\Zbb^{n-1}, q)$ on each point to enlarge the category for $n \geq 0$.\\
    
    For a fixed $n$, for each $W \in \{*, \Zbb_{\geq 0}, \Zbb_{\leq 0}, \Zbb\}$, we use the category $\mathbf{C}'_{n-1}(W)$ as defined in Definition~\ref{def::enlarge_category}. We also note that if $n = 1$, then $\textbf{C}'_{0}(W)$ is constructed by placing Azumaya algebras over each point of $W$. In particular, $\textbf{C}'_{0}(*)$ is $\operatorname{Az}(R)$.\\
    
    Recall that $\mathbf C(W \times \Zbb^{n-1})$ is cofinal and full within $\mathbf{C}'_{n-1}(W)$, so the connected component at identity of their respective K-theory spaces are equivalent (i.e., their higher K-groups above $\pi_0$ are the same). Since we are only proving the statement after applying $\Omega$, we can recast the problem to the setting of $\mathbf{C}'_{n-1}(W)$. Define $u\colon \mathbf{C}'_{n-1}(*) \to \mathbf{C}'_{n-1}(\Zbb_{\geq 0})$ and $v\colon \mathbf{C}'_{n-1}(*) \to \mathbf{C}'_{n-1}(\Zbb_{\leq 0})$ as the usual inclusion functors. Theorem~\ref{thm::thomason_pullback} gives a homotopy pullback diagram
    % https://q.uiver.app/#q=WzAsNCxbMCwwLCJLKFxcbWF0aGNhbHtDfScoKikpIl0sWzEsMCwiSyhcXG1hdGhjYWx7Q30nKFxcWmJiX3tcXGdlcSAwfSkpIl0sWzEsMSwiSyhQKSJdLFswLDEsIksoXFxtYXRoY2Fse0N9JyhcXFpiYl97XFxsZXEgMH0pKSJdLFswLDEsInUiXSxbMCwzLCJ2IiwyXSxbMywyXSxbMSwyXV0=
\[\begin{tikzcd}
	{K(\mathbf{C}'_{n-1}(*))} & {K(\mathbf{C}'_{n-1}(\Zbb_{\geq 0}))} \\
	{K(\mathbf{C}'_{n-1}(\Zbb_{\leq 0}))} & {K(P)}
	\arrow["u", from=1-1, to=1-2]
	\arrow["v"', from=1-1, to=2-1]
	\arrow[from=1-2, to=2-2]
	\arrow[from=2-1, to=2-2]
\end{tikzcd},\]
where $P$ denotes Thomason's simplified double mapping cylinder construction (Definition~\ref{def::cylinder}). A similar argument as in Lemma~\ref{lem::contractible} shows that $K(\mathbf{C}'_{n-1}(\Zbb_{\leq 0}))$ and $K(\mathbf{C}'_{n-1}(\Zbb_{\geq 0}))$ are contractible, as the admissible tensor factors are placed pointwise. Alternatively, we note that since $\textbf{C}(\Zbb_{\leq 0})$, for instance, is cofinal in $\textbf{C}'_{n-1}(\Zbb_{\leq 0})$, Lemma~\ref{lem::contractible} shows it suffices to show $\textbf{C}'_{n-1}(\Zbb_{\leq 0})$ is connected. A construction similar to Lemma~\ref{lem::contractible} or Lemma 1.3 of \cite{pedersen_weibel} then shows it is connected.\\

Hence we have that $K(\mathbf{C}'_{n-1}(*)) \simeq \Omega K(P)$ by Corollary~\ref{cor::thomason_corollary}. We wish to now show that $K(P) \simeq K(\mathbf{C}'_{n-1}(\Zbb))$. By the universal property of $P$, there is a natural strong symmetric monoidal functor
\[F\colon P \to \mathbf{C}'_{n-1}(\Zbb),\quad (A^{-}, A, A^{+}) \mapsto A^{-} \otimes A \otimes A^{+}.\]
It suffices to show that Quillen's Theorem A holds. For $Y \in \mathbf{C}'_{n-1}(\Zbb)$, we  decompose the comma category $Y \downarrow F$ as an increasing union of sub-categories  $\bigcup_{d \geq 0} \operatorname{Fil}_{d}$ where $\operatorname{Fil}_{d}$ is the full subcategory of $Y \downarrow F$ consisting of morphisms $\alpha: Y \to F(A^{-1}, A, A^{+})$ such that $\alpha$ and $\alpha^{-1}$ have spread bounded by distance $d$.\\

We wish to now construct an initial object in $\operatorname{Fil}_{d}$, which will show that $B\operatorname{Fil}_{d}$ is contractible. For an object $Z \in \textbf{C}'_{n-1}(\Zbb)$, we write $Z(n)$ as the factor of $Z$ at $n \in \mathbb{Z}$.  Now we may write $Y = Y^{-}_d \otimes Y_d \otimes Y^{+}_d$ where we have that
\[Y_d(n) = \begin{cases}
    Y(n), -d \leq n \leq d\\
    R, \text{ otherwise }
\end{cases}\hspace{-0.5cm},\ Y_d^{-}(n) = \begin{cases}
    Y(n), n < -d\\
    R, \text{ otherwise }
\end{cases}\hspace{-0.5cm},\ Y_d^{+}(n) = \begin{cases}
    Y(n), n > d\\
    R, \text{ otherwise }
\end{cases}\hspace{-0.4cm}.\]
 We write $\sigma: Y \to Y_d^{-} \otimes f(Y_d) \otimes Y^+_d = F(Y_d^{-}, f(Y_d), Y_d^{+})$ to denote the obvious isomorphism, where $f(Y_d)(p, 0) = \bigotimes_{-d \leq n \leq d} Y_d(p, n)$ for all $p$ and $f(Y_d)(p, n) = R$ otherwise\footnote{\textbf{Note:} The $f(Y_d)$ used in the proof of Theorem~\ref{thm::higher_delooping} here is the notation $Y_d$ in \cite{pedersen_weibel}.}.\\

Now suppose $\alpha: Y \to F(A^{-}, A, A^{+})$ is in $\operatorname{Fil}_{d}$. We take $e_{-}(Y_d)$ as the tensor factor of $Y_d$ that is sent to $\Zbb_{< 0}$ by $\alpha$, $e_0(Y_d)$ as the tensor factor of $Y_d$ that is sent to $\{0\}$ by $\alpha$, and $e_+(Y_d)$ as the tensor factor of $Y_d$ that is sent to $\Zbb_{> 0}$, and in particular we have
\[Y_d = e_{-}(Y_d) \otimes e_0(Y_d) \otimes e_{+}(Y_d).\]
We explain how to construct the components $e_{-}(Y_d), e_0(Y_d),$ and $e_+(Y_d)$ as follows. For $X \subseteq \Zbb$, we let $\mathcal{B}(X, q)$ denote the tensor product (in the colimit sense) of factors assigned at each point $x \in X$ for the configuration $F(A^{-}, A, A^+)$. Observe that
\[\mathcal{B}(\Zbb \cap (-\infty, -2d), q) \subseteq \alpha(Y_d^{-}) \subseteq A^{-}.\]
Note here that the first inclusion makes sense because we are imposing the $L^{\infty}$-metric. Since $\mathcal{B}(\Zbb \cap (-\infty, -2d), q)$ is an admissible tensor factor of a locally matrix algebra, Corollary~\ref{cor::tensor_factor_splitting} implies that we can split
\[\alpha(Y^{-}_d) = \mathcal{B}(\Zbb \cap (-\infty, -2d), q) \otimes D^{-}.\]
Observe that by construction $D^{-} \subseteq \mathcal{B}(\Zbb \cap [-2d, 0), q)$ and clearly $D^{-}$ is a faithfully flat tensor factor of $\mathcal{B}(\Zbb, q)$, as both $\alpha(Y_d^{-})$ and $\mathcal{B}(\Zbb \cap (-\infty, -2d), q)$ are by Lemma~\ref{lem::admissible_faithfully_flat}. Thus Lemma~\ref{lem::tensor_complement_exist} implies that
\[\mathcal{B}(\Zbb \cap [-2d, 0), q) = D^{-} \otimes E^{-}.\]
Note that Lemma~\ref{lem::tensor_complement_exist} is un-necessary if $n = 1$, as it just follows from the centralizer theorem (Lemma~\ref{lem::centralizer}) in that case.\\

Now, we define $e_-(Y_d)$ as $\alpha^{-1}(E^{-})$ and we symmetrically define $e_+(Y_d)$. We also define $e_0(Y_d)$ as $\alpha^{-1}(A)$. This gives a decomposition
\[Y_d = e_-(Y_d) \otimes e_0(Y_d) \otimes e_+(Y_d) \in \mathbf{C}'_{n-1}([-d, d]) \simeq \mathbf{C}'_{n-1}(*).\]
From here, we define a morphism in $P$, as
\[\eta \coloneqq (\psi, \psi^{-}, \psi^{+}, f^{-}(e_{-}(Y_d)), f^+(e_{+}(Y_d))): (Y^{-}_d, f(Y_d), Y^{+}_d) \to (A^{-}, A, A^{+})\]
where $f^{-}$ and $f^{+}$ denotes the obvious projection to the position $0 \in \Zbb$. Here
\[\psi\colon f(Y_d) = f^{-}(e_{-}(Y_d)) \otimes f(e_0(Y_d)) \otimes f^{+}(e_{+}(Y_d)) \to f^{-}(e_{-}(Y_d)) \otimes A \otimes f^{+}(e_{+}(Y_d))\]
is the map $1 \otimes \alpha \circ f^{-1} \otimes 1$. The map $\psi^{-}$ is given by
\[\psi^{-} \colon Y^{-}_d \otimes f^{-}(e_{-}(Y_d)) \xrightarrow{1 \otimes (f^{-})^{-1}} Y^{-}_d \otimes e_{-}(Y_d) \xrightarrow{\alpha} A^{-},\]
and similarly for $\psi^{+}$. This is almost what we want, but we need to show that $f(e_{-}(Y_d)), f(e_0(Y_d)), f(e_+(Y_d))$ are all objects in $\mathbf{C}'_{n-1}(*)$ and the relevant maps constructed here are locality-preserving. Lemma~\ref{lem::technical_admissible}, which we will state and prove below, shows that these factors are admissible. It then follows from Lemma~\ref{lem::invert_local_preserve_equal} that the relevant maps constructed here are locality-preserving. Uniqueness of the morphism can be verified from the fact that the relevant categories other than $P$ are all groupoids and that two morphisms in $P$ are identified if the two objects in their 5-tuple differ by isomorphisms.
\end{proof}

\begin{remark}\label{rmk::more_generally}
The proof above extends to show that $\mathbf{Q}(X) \simeq \Omega \mathbf{Q}(X \times \Zbb)$ for a wide class of metric spaces. This gives a genuinely different $\Omega$-spectrum $$\mathbf{Q}(X), \mathbf{Q}(X\times \ZZ), \dots,  \mathbf{Q}(X\times \ZZ^n), \dots$$ for $X$ unbounded. In the first two versions of this preprint, we suggested a connection with coarse homotopy theory, citing in particular the bornological coarse framework of \cite{bunke2020homotopy}. Subsequently, \cite{ludewig2026quantumcellularautomatacoarse} established such a relation in the unitary case over $\mathbb{C}$, following the strategy of this paper.
\end{remark}

\begin{remark}
As suggested in the proof above, the case for $\mathbf{Q}(*) \simeq \Omega \mathbf{Q}(\Zbb^1)$ (or more generally, $\mathbf{Q}(X) \simeq \Omega \mathbf{Q}(X \times \Zbb)$ for $X$ bounded) can be proven without using Lemma~\ref{lem::tensor_complement_exist}, Lemma~\ref{lem::invert_local_preserve_equal}, or Lemma~\ref{lem::technical_admissible}, or the more discussions on general admissible tensor factors from Construction~\ref{defn::admissible}.
\end{remark}

We now state and prove the technical lemma used in the proof of Theorem~\ref{thm::higher_delooping}.
\begin{lemma}\label{lem::technical_admissible}
    Let $F, E \in \mathbf{C}'_X(\Zbb)$ and $\beta: F \to E$ be an $R$-algebra isomorphism of finite spread $< \ell$. For a subset $S \subseteq \Zbb$, write $F(S)$ and $E(S)$ as the factor of $F$ and $E$ over $S$.
    \begin{enumerate}
        \item $\beta^{-1}(E(0))$ is admissible.
        \item Since we have that 
        \[E(\Zbb \cap (-\infty, -\ell]) \subseteq \beta(F(\Zbb \cap (-\infty, 0)) \subseteq E(\Zbb \cap (-\infty, +\ell)),\]
       Corollary~\ref{cor::tensor_factor_splitting} shows that we can write
        \[\beta(F(\Zbb \cap (-\infty, 0))) = B \otimes E(\Zbb \cap (-\infty, -\ell]).\]
        Here $B$ is also admissible.
    \end{enumerate}
    (1) shows $e_0(Y_d)$ is admissible by taking $\beta = \alpha$. (2) shows $e_-(Y_d)$ is admissible by taking $\beta = \alpha^{-1}$, and this symmetrically verifies $e_+(Y_d)$.
\end{lemma}

\begin{proof}
    (1) $E(0)$ is clearly admissible in $E$ (or $E\otimes E'$) by the characterization of Lemma~\ref{lem::invert_local_preserve_equal}. Since a locality-preserving map $\beta^{-1}$ takes admissible to admissible, $\beta^{-1}(E(0))$ is admissible.\\
    
    (2) A similar proof as in Lemma~\ref{lem::invert_local_preserve_equal} shows that $B$ is an invertible tensor factor of $E(\Zbb \cap (-\ell, +\ell))$, and $E(\Zbb \cap (-\ell, +\ell))$ is an invertible tensor factor of some matrix algebra. This is because we can once again linearly project to the identity, by which we mean a surjective $R$-linear map that sends the identity to identity. This follows from a general fact that if $P$ is a sub-module of a free $R$-module $P'$ and let $x \in P$, then we can restrict the projection of $P' \to R \cdot x$ onto the domain of $P$ to achieve the desired map. Lemma~\ref{lem::invertible_transitive} then shows that $B$ is an invertible tensor factor of some matrix algebra, so Lemma~\ref{lem::invert_local_preserve_equal} implies $B$ is admissible.
\end{proof}

Following through the proof of Theorem~\ref{thm::higher_delooping} in the case for $n = 1$ gives the following corollary.
\begin{cor}\label{cor::qz_azumaya}
$\mathbf{Q}(\Zbb) = K(\operatorname{Az}(R))$, and hence $K(\mathbf{C}(\Zbb))$ is a delooping of $K(\operatorname{Az}(R))$.
\end{cor}

\begin{proof}
  When $n = 1$, Theorem~\ref{thm::higher_delooping} shows $K(\textbf{C}'(\Zbb))$ is a delooping of $\textbf{C}'(*) = K(\operatorname{Az}(R))$. In other words,
  $\textbf{Q}(\Zbb) = \Omega K(\textbf{C}(\Zbb)) = \Omega K(\textbf{C}'_0(\Zbb)) = K(\operatorname{Az}(R))$.
\end{proof}

Theorem~\ref{thm::higher_delooping} also gives a sequence of deloopings of $K(\operatorname{Az}(R)) = \mathbf{Q}(\Zbb)$. In light of how negative K-theory can be constructed \cite{PEDERSEN1984461}, the groups $\mathcal{Q}(\Zbb^n)/\mathcal{C}(\Zbb^n)$ for $n > 1$ can really be thought of as the \textit{negative homotopy groups} for $K(\operatorname{Az}(R))$. We conclude this section by discussing how these negative homotopy groups can be seen as a generalization of the Brauer groups.\\

The proof of Theorem~\ref{thm::higher_delooping} extends more generally to show that $K(\mathbf{C}'_X(*)) \simeq \Omega K(\mathbf{C}(X \times \Zbb))_0$. The space $K(\mathbf{C}'_X(*))$ can be thought of as the K-theory space of invertible/admissible tensor factors on $X$. When we take $X = \Zbb^n$ and apply $\pi_0$ to both sides, we obtain the analog over a ring of Theorem 3.15 of \cite{haah2023invertible} (which was done in the unitary case over $\mathbb{C}$). When we go over the unitary case in Section~\ref{sec::star}, this would recover Theorem 3.15 of \cite{haah2023invertible}. This inspired the following definitions. 
\begin{defn}
    Let $\mathbf{C^{proj}}(X)$ be the full subcategory of $\mathbf{C}'_X(*)$ whose objects are algebras of the form 
    \begin{equation}
        \bigotimes_{x\in X} \End_R(P_x) 
    \end{equation}
     with finitely generated projective $R$-modules $P_x$ and $P_x=R$ for all but a locally finite subset of $X$. In particular, when $R$ is a field, $\mathbf{C^{proj}}(X)=\mathbf{C}(X)$.
\end{defn}
\begin{defn}
     The Brauer group of admissible tensor factors on $X$ is
    \begin{equation}
        \mathrm{Br}_X(R)=\pi_0 K(\mathbf{C}'_X(*))/ \pi_0 K(\mathbf{C^{proj}}(X)). 
    \end{equation}
\end{defn}
\begin{exmp}\label{example::Brauer_ring}
    For $X=*$, this recovers the usual Brauer group $\mathrm{Br}_X(R)=\mathrm{Br}(R)$ by Theorem~\ref{thm::k0_qz}.
\end{exmp}
\begin{exmp}\label{example::higher_Brauer}
    Over a field $\mathbb F$, Corollary~\ref{cor::qca_space_properties}, Theorem~\ref{thm::pi0_is_coarse}, Remark~\ref{rmk:coarse_connected} implies
    \begin{equation}     
    \mathrm{Br}_{\ZZ^n}(\mathbb F)=\pi_0 K(\mathbf{C}'_{\Zbb^n}(*))=\pi_1K(\mathbf C(\ZZ^{n+1}))=\mathcal Q(\ZZ^{n+1})/ \CalC(\ZZ^{n+1}).
    \end{equation}
    for $n\geq 1$. This provides a sequence of higher Brauer groups index by $n$.
\end{exmp}

\section{The QCA Conjecture over $C^*$-Algebras with Unitary Circuits}\label{sec::star}

In this section, we review the definition of QCA from the physics literature~\cite{farrelly2020review}; see also appendix~A of~\cite{yang2025categorifying}. Then we apply the strategy developed in Sections~\ref{sec::space} and~\ref{sec::spectra} to construct the $\Omega$-spectrum $\mathbb{QCA}$, thereby resolving the QCA conjecture. Throughout this section, we fix $R = \CC$. 

To distinguish this notion from the version of QCA considered in earlier sections, we use the term $*$-QCA. Note, however, that $*$-QCA is the standard notion of QCA in the physics literature.

\subsection{The $*$-QCA Group}Let $(X,\rho)$ be a countable, locally finite metric space and let 
$q \in \mathbb{N}^X_{\mathrm{lf}}$. The algebra of local observables $\SA(X,q)$ is a normed $*$-algebra.

Indeed, for every $A \in \SA(X,q)$ there exists a finite subset 
$\Gamma \subset X$ such that 
\[
A \in \SA(\Gamma,q),
\]
where $\SA(\Gamma,q)$ is a finite-dimensional matrix algebra. 
We define
\[
\|A\| := \|A\|_{\mathrm{op}},
\]
the operator norm of $A$ inside $\SA(\Gamma,q)$. The involution $A \mapsto A^*$ is given by matrix adjoint (conjugate transpose).

\begin{prop}
The norm completion $\overline{\SA(X,q)}$ of $\SA(X,q)$ is a $C^*$-algebra.
\end{prop}

\begin{proof}
See Section~6.2 of \cite{bratteli2012operator2}.
\end{proof}

\begin{defn}
Given two quantum spin systems $q, q'$ over $(X,\rho)$, a \emph{locality-preserving $*$-isomorphism} is a $C^*$-algebra isomorphism
\[
\alpha: \overline{\SA(X,q)} \longrightarrow \overline{\SA(X,q')}
\]
of finite spread $l>0$. That is, for any $x\in X$ and any
\[
A\in \SA(\{x\}, q) = \mathrm{Mat}(\CC^{q_x}),
\]
we have
\[
\alpha(A)\in \SA(D_\ell(x), q'), 
\qquad 
D_\ell(x)=\{y\in X: d(x,y)\le l\}.
\]
A spin system equipped with a locality-preserving $*$-automorphism $[q,\alpha]$ is called a \emph{$*$-quantum cellular automaton} ($*$-QCA). For fixed $q$, the set of $*$-QCA forms a group under composition, denoted $\CQ^*(X,q)$.
\end{defn}

\begin{lemma}
Let 
\[
\alpha_0: \SA(X,q)\longrightarrow \SA(X,q')
\]
be a locality-preserving isomorphism as in Definition~\ref{def:locality_preserve_iso}. Suppose $\alpha_0$ is \textit{adjoint preserving} in the sense that 
\[
\alpha_0(A^*)=\alpha_0(A)^*
\quad\text{for all } A\in \SA(X,q),
\]
where the superscript $*$ denotes the adjoint (the conjugate transpose). Then $\alpha_0$ extends uniquely to a locality-preserving $*$-isomorphism
\[
\alpha: \overline{\SA(X,q)}\longrightarrow \overline{\SA(X,q')}.
\]
\end{lemma}

\begin{proof}
For $A\in \SA(X,q)$, the $C^*$-identity gives
\[
\|A\|^2=\|A^*A\|
=\sup\{|\lambda|: A^*A-\lambda I \text{ is not invertible}\}.
\]
The latter is the spectral radius of $A^*A$, and equality holds since $A^*A$ is positive.

Let $S\subset X$ be the support of $A^*A$. The restricted map
\[
\alpha_0|_S: \SA(S,q)\longrightarrow \alpha_0(\SA(S,q))
\]
is an abstract matrix algebra isomorphism and hence inner by the Skolem--Noether theorem. In particular, it preserves the spectral radius. Thus
\begin{align}
\|A\|^2
&= \sup\{|\lambda|: \alpha_0(A^*A)-\lambda I \text{ not invertible}\} \\
&= \|\alpha_0(A)^*\alpha_0(A)\| \\
&= \|\alpha_0(A)\|^2.
\end{align}

Therefore, $\alpha_0$ is isometric and hence continuous. It extends uniquely to the completion. Applying the same argument to $\alpha_0^{-1}$ shows that the extension is a $C^*$-algebra isomorphism.
\end{proof}

\begin{cor}
The group $\CQ^*(X,q)$ is canonically isomorphic to the group
\[
\{\alpha\in \CQ(X,q): \alpha(A^*)=\alpha(A)^* \text{ for all } A\in \SA(X,q)\}.
\] Thus, we will not differentiate between the two. 
\end{cor}

Since the $*$-property is invariant under stabilization, we define the \emph{total $*$-QCA group} $\CQ^*(X)$ using the same colimit construction as in Definition~\ref{def::qca_group}.

\begin{defn}
Let $q \in \NN^{X}_{\mathrm{lf}}$ be a spin system.  
A \emph{single-layer unitary circuit} consists of:

\begin{enumerate}
\item A uniformly bounded partition
\[
X = \coprod_{j} X_j,
\qquad 
\sup_j \mathrm{diam}(X_j) < \infty.
\]

\item For each $X_j$, an automorphism
\[
\alpha_j \in \operatorname{PU}_{q(X_j)}, \text{ where }q(X_j):=\prod_{x\in X_j} q_x,
\]
given by conjugation by a unitary matrix. Here $\operatorname{PU}_{n}$ denotes the $n \times n$ projective unitary matrix group.
\end{enumerate}

Then $\bigotimes_j \alpha_j \in \CQ^*(X,q)$.  
Let $\CalC^*(X,q)\subset \CQ^*(X,q)$ denote the subgroup they generate.  
Passing to the colimit yields $\CalC^*(X)\subset \CQ^*(X)$.
\end{defn}

\begin{prop}[Lemma~2.10 and Theorem~2.3 of~\cite{freedman2020classification}]\label{prop::unitary_normal}
$\CalC^*(X)$ is a normal subgroup of $\CQ^*(X)$, and the quotient $\CQ^*(X)/\CalC^*(X)$ is abelian.
\end{prop}

\begin{proof}
A SWAP can be obtained through a conjugation by some unitary matrix. The remainder follows as in the proof of the corresponding statement in the non-$*$ setting in Section~\ref{subsec::circuits}. 
\end{proof}

\begin{theorem}\label{thm::unitary_commutator}
$\CalC^*(X)=[\CQ^*(X),\CQ^*(X)]$.
\end{theorem}

\begin{proof}
By Proposition~\ref{prop::unitary_normal}, it suffices to show $\CalC^*(X)\subseteq [\CQ^*(X),\CQ^*(X)]$. By the main theorem of \cite{10.2969/jmsj/00130270}, the group $\operatorname{PU}_n$ is perfect for $n\ge2$, and every element in $\operatorname{PU}_n$ can be expressed as a single commutator. Hence a similar proof as in Lemma~\ref{lem::roots_k_commutator} shows $\CalC^*(X)=[\CalC^*(X),\CalC^*(X)]
\subseteq [\CQ^*(X),\CQ^*(X)].$
\end{proof}

\subsection{The $*$-QCA Space and Spectrum}

We now repeat the construction of Section~\ref{sec::space} in the $*$-setting.

\begin{defn}
The category of quantum spin systems under locality-preserving $*$-isomorphisms forms a symmetric monoidal category $\mathbf{C}^*(X)$:
\begin{itemize}
\item Objects: $\mathcal{A}(X,q)$ for $q:X\to \Nbb_{>0}$.
\item Morphisms: locality-preserving $*$-isomorphisms.
\end{itemize}
The symmetric monoidal structure is given by pointwise stacking.
\end{defn}

\begin{defn}
The \emph{space of $*$-QCA} over $X$ is
\[
\mathbf{Q}^*(X)
\coloneqq \Omega K(\mathbf{C}^*(X)).
\]
We write $\mathbf{Q}^*_i(X):=\pi_i\mathbf{Q}^*(X)$. Then
\[
\mathbf{Q}^*_i(X)=K_{i+1}(\mathbf{C}^*(X))
\quad\text{for all } i\ge0.
\]
\end{defn}

When $X = \Zbb^d$, $\mathbf{Q}^*(X)$ is the space $\mathrm{QCA}(\Zbb^d)$ in Section~\ref{subsec::intro}. We can still define the automorphism group of $\textbf{C}^*(X)$ in this case, and $\operatorname{Aut}(\mathbf{C}^*(X)) = \mathcal{Q}^*(X)$. Combining this with Theorem~\ref{thm::unitary_commutator}, a similar proof as in Theorem~\ref{thm::plus_construction} shows that
\begin{theorem}
There is an equivalence $K(\mathbf{C}^*(X))
\simeq
K_0(\mathbf{C}^*(X))
\times B\CQ^*(X)^+$, In particular, we have that
\[
\mathbf{Q}^*(X)\simeq \Omega(B\CQ^*(X)^+),
\qquad
\mathbf{Q}^*_0(X)=\CQ^*(X)/\CalC^*(X).
\]
\end{theorem}

Finally, we obtain an analogous delooping theorem.
\begin{theorem}
There is a homotopy equivalence
\[
\mathbf{Q}^*(\Zbb^{n-1})
\simeq
\Omega\,\mathbf{Q}^*(\Zbb^n),
\qquad n>0.
\]
\end{theorem}

\begin{proof}
Repeat the proof of Section~\ref{sec::spectra}, with the additional requirement that the admissible tensor factors be $*$-subalgebras (i.e.\ closed under adjoint). The analogous splittings and the morphisms constructed in the proof of Theorem~\ref{thm::higher_delooping} respect the $*$-structure. To see this, in Construction~\ref{defn::admissible}, we have the following \[\mathcal{A}(X \times \Zbb_{\leq -\ell}, q) \otimes \mathcal B = \alpha(\mathcal{A}(X \times \Zbb_{\leq 0}, q)). \] If $b\in \mathcal B$, then $(1\otimes b)^*=1\otimes b^*$ is in the right-hand-side because $\alpha$ commutes with taking adjoint and $\mathcal{A}(X \times \Zbb_{\leq 0}, q)$ is closed under adjoint. This then implies $b^*\in \mathcal B$, hence $\mathcal{B}$ is closed under adjoint. In Lemma~\ref{lem::technical_admissible}, there is a similar expression \[\beta(F(\Zbb \cap (-\infty, 0)) = B \otimes E(\Zbb \cap (-\infty, -\ell]).\] The same reasoning shows $B$ is closed under adjoint, provided $\beta$ commutes with adjoint.  The statement for morphisms follows similarly.
\end{proof}

\begin{remark}
 We note that in this case the admissible tensor factors (or rather equivalently, the invertible tensor factors as in Lemma~\ref{lem::invert_local_preserve_equal}) are exactly the \textit{invertible subalgebras} appearing in~\cite{haah2023invertible}. And Corollary 2.5 of \cite{haah2023invertible} showed that all invertible subalgebras are indeed tensor factors.
\end{remark}

\section{Calculations over Points and Lines}\label{sec::compute}

In this section, we calculate the groups $\mathbf{Q}_i(*)$ and $\mathbf{Q}_i(\Zbb)$ for all $i$. By Theorem~\ref{thm::higher_delooping} and Corollary~\ref{cor::qz_azumaya}, both calculations reduce to knowing the K-theory of Azumaya algebras, which has been calculated over any ring in \cite{c80f3555-900f-35b5-ab05-6d397e6a44e6}. Nevertheless, following \cite{c80f3555-900f-35b5-ab05-6d397e6a44e6}, we explain in detail how to compute them over a field in Section~\ref{subsec::field} for the sake of completeness and concreteness. Table~\ref{table_q_z} records the groups $\textbf{Q}_i(\mathbb{Z})$ over a field $k$, as well as a specialization to the case where $k = \mathbb{C}$.\\

We will also see in Section~\ref{subsec::field} that for $i$ sufficiently large, $\mathbf{Q}_i(\Zbb^n)$ should also be rational, because $K_i(\operatorname{FP}) = K_i(R) \otimes \Qbb$ for $i > 0$. The statement for $\operatorname{FP}$ has been pointed out in \cite{weibel2013k}, pointed out and proven in low degrees for \cite{bass1968algebraic}, and is implicit in \cite{May1977EInfinityRingSpaces}. In Section~\ref{subsec::rationalization}, we explain why this rationalization phenomenon occurs at $i = 1$ over an Euclidean domain, to give the reader some intuition on why rationalization may occur.

\begin{table}[th]
\centering

\caption{The groups $\mathbf{Q}_i(\mathbb{Z})$ over a field $k$ and specifically over the complex numbers $\mathbb{C}$. We explain how to obtain the case for $\Cbb$ in Corollary~\ref{cor::case_for_complex}.}
\label{table_q_z}
\begin{adjustbox}{width=0.6\linewidth}
\begin{tabular}{lcc}
\toprule
The groups $\mathbf{Q}_i(\mathbb{Z})$ & Over a field $k$ & Over $\mathbb{C}$\\
\midrule
$i = 0$ & $K_0(\operatorname{Az}(k))$ & $\Qbb_{>0}$\\
$i = 1$ & $\frac{\Qbb}{\Zbb} \otimes k^{\times}$ & 0 \\
$i = 2$ & $\mu(k) \oplus (K_2(k) \otimes \Qbb)$ & $\frac{\Qbb}{\Zbb} \oplus K_2(\Cbb)$\\
$i \geq 3$ even & $K_i(k) \otimes \Qbb$ & $K_i(\Cbb)$\\
$i \geq  3$ odd & $K_i(k) \otimes \Qbb$ & $K_i(\Cbb) \otimes \Qbb$\\
\bottomrule
\end{tabular}
\end{adjustbox}
\end{table}

\subsection{Homotopy Groups of $\mathbf{Q}(*)$ and $\mathbf{Q}(\Zbb)$}\label{subsec::field}

In this section, we wish to calculate the homotopy groups $\textbf{Q}(*)$ and $\textbf{Q}(\Zbb)$. Note that if the metric space $X$ in our set-up is bounded, then $X$ is coarsely equivalent to a point $*$, and hence $\textbf{Q}(X) \simeq \textbf{Q}(*)$. This includes the case when $X$ is a finite set of points.\\

Theorem~\ref{thm::higher_delooping} implies that $\textbf{Q}(*) = \Omega \textbf{Q}(\Zbb)$ and $\textbf{Q}(\Zbb) = K(\operatorname{Az}(R))$. Weibel computed the K-theory groups of $\operatorname{Az}(R)$ as follows.

\begin{theorem}[Theorem 9 of \cite{weibel2013k}]\label{thm::weibel_azumaya}
    Let $U(R)$ (or $R^\times$) be the units of $R$, $\mu(R)$ be the roots of unity of $R$, $\operatorname{SK}_1(R)$ denote the kernel of the determinant map $K_1(R) \to U(R)$, and $\operatorname{TPic}(R)$ be the torsion subgroup of the Picard group $\operatorname{Pic}(R)$. 
\begin{itemize}
    \item $K_1(\operatorname{Az}(R)) = \operatorname{TPic}(R) \oplus (\Qbb/\Zbb \otimes U(R)) \oplus (\Qbb \otimes \operatorname{SK}_1(R))$.
    \item $K_2(\operatorname{Az}(R)) = \mu(R) \oplus (\Qbb \otimes K_2(R))$.
    \item $K_i(\operatorname{Az}(R)) = \Qbb \otimes K_i(R)$ for $i \geq 3$.
\end{itemize}
There is a description for $K_0 \operatorname{Az}(R)$ in \cite{weibel2013k}, but we omit it here.
\end{theorem}

In the remainder of this section, we will discuss in detail on how to obtain the result above for $i \geq 1$ when $R$ is a field. Note that the outputs match exactly that of Table~\ref{table_q_z} because $\operatorname{Pic}(R) = \operatorname{SK}_1(R) = 0$.

Let $R$ be a field. For each $n$, there is an exact sequence
\[0 \to R^\times \hookrightarrow \operatorname{GL}_n(R) \twoheadrightarrow
 \operatorname{PGL}_n(R) \to 0\]
where $r \in R^{\times}$ is sent to $\operatorname{GL}_n(R)$ as diagonal matrices whose diagonal entries are all $r$. This extends well through the colimit to the form
\[0 \to R^{\times} \hookrightarrow \operatorname{GL}_{\otimes}(R) \twoheadrightarrow \operatorname{PGL}(R) \to 0\]
where $\operatorname{GL}_{\otimes}(R)$ is the colimit of $\operatorname{GL}_n(R)$'s over the poset $\Nbb$ under divisibility and a map $\operatorname{GL}_{k}(R) \to \operatorname{GL}_{n}(R)$ by $n = kd$ is given by sending $A \in \operatorname{GL}_{k}(R)$ to the block diagonal matrix with $d$-copies of $A$ (note this is the same as taking the Kronecker product with the identity matrix $I_d$). The group $\operatorname{PGL}(R)$ is defined similarly.\\

Observe we have that
\[\operatorname{GL}_{\otimes}(R) = \operatorname{Aut}(\operatorname{Free}(R)^{f.g, \cong}_{\otimes}) \text{ and } \operatorname{PGL}(R) = \operatorname{Aut}(\operatorname{Mat}_R) = \operatorname{Aut}(\operatorname{Az}(R)),\]
Here, $\operatorname{Free}(R)^{f.g, \cong}_{\otimes}$ is the category of finitely generated non-zero free $R$-modules, morphisms are $R$-linear automorphisms, and is symmetric monoidal under tensor products. The second inequality follows from the Skolem-Noether theorem that automorphisms of matrix algebras over a field are always inner. Proposition 3 of \cite{c80f3555-900f-35b5-ab05-6d397e6a44e6} shows that $\operatorname{GL}_{\otimes}(R), \operatorname{PGL}(R)$ have perfect commutator subgroups, and
\[K(\operatorname{Free}(R)^{f.g, \cong}_{\otimes}) \simeq \Qbb_{>0} \times \operatorname{BGL}_{\otimes}(R)^+ \text{ and } K(\operatorname{Az}(R)) \simeq K_0(\operatorname{Az}(R)) \times \operatorname{BPGL}(R)^+.\]

Thus, it suffices to calculate the homotopy groups of $\operatorname{PGL}(R)^+$. 

\begin{prop}
The fiber sequence $BR^{\times} \to \operatorname{BGL}_{\times}(R) \to \operatorname{BPGL}(R)$ induces a fiber sequence
\[BR^{\times} \to \operatorname{BGL}_{\otimes}(R)^+ \to \operatorname{BPGL}(R)^+,\]
where the plus-constructions are done both at the maximal perfect normal subgroups of $\operatorname{BGL}_{\otimes}(R)$ and $\operatorname{BPGL}(R)^+$ respectively.
\end{prop}

\begin{proof}
Let $\mathcal{P}(G)$ denote the maximal perfect normal subgroup of $G$. We use the main theorem of \cite{BERRICK1983172}, which asserts that for a fiber sequence $F \to E \xrightarrow{p} B$ of spaces with homotopy types of connected CW-complexes induces a fiber sequence of plus-constructions at the maximal perfect normal subgroups
\[F^+ \to E^+ \to B^+\]
if and only if the following two conditions hold:
\begin{enumerate}
    \item $\pi_1(p: E \to B)$ sends $\mathcal{P}(\pi_1(E))$ surjectively onto $\mathcal{P}(\pi_1(B))$.
    \item $\mathcal{P}(\pi_1(E))$ acts trivially on $\pi_*(F^+)$. (To be precise, one sends $\mathcal{P}(\pi_1(E))$ to $\mathcal{P}(\pi_1(B))$ first, and let $\mathcal{P}(\pi_1(B))$ act on the fiber)
\end{enumerate}
Now in our context, the map between $\pi_1$ is literally the quotient map $\operatorname{GL}_{\otimes}(R) \to \operatorname{PGL}(R)$ modding out $R^{\times}$, and clearly it sends $E_{\otimes}(R)$, the commutator subgroup of $\operatorname{GL}_{\otimes}(R)$, surjectively onto $\operatorname{PSL}(R)$.\\

Since $R$ is commutative, $R^\times$ is abelian, so its maximal perfect normal subgroup is trivial. Thus, $(BR^{\times})^+ = BR^{\times}$ and its higher homotopy groups being $0$ above degree $1$. To show the second condition, it suffices for us to show that $\operatorname{E}_{\otimes}(R)$, the commutator subgroup of $\operatorname{GL}_{\otimes}(R)$, acts trivially on $\pi_1(BR^{\times}) = R^{\times}$. Indeed, the action is just by conjugation. On the other hand, $R^\times$ lies in the center of $\operatorname{GL}_{\otimes}(R)$, and hence the action is trivial. This concludes the proof. Note the proof holds without needing to assume $R$ is a field.
\end{proof}

Thus for $n > 2$, we see that
\[\pi_n(\operatorname{Az}(R)) = \pi_n(\operatorname{PGL}(R)^+) = \pi_n(\operatorname{BGL}_\otimes (R)^+).\]
It was pointed out in \cite{c80f3555-900f-35b5-ab05-6d397e6a44e6, 951cf383-5a4d-3775-b8b2-0bafd2f162fe} and implicitly in \cite{May1977EInfinityRingSpaces} that the following fact holds.

\begin{prop}
    For $n > 0$, $\pi_n(\operatorname{BGL}_{\otimes}(R)^+) = K_n(R) \otimes \Qbb$ for a general ring $R$.
\end{prop}

We now have an exact sequence of the form
\[0 \to K_2(R) \otimes \Qbb \to K_2(\operatorname{Az}(R)) \to R^{\times} \to K_1(R) \otimes \Qbb \to K_1(\operatorname{Az}(R)) \to 0.\]
Since $R$ is a field, $K_1(R) = R^{\times}$, and the map $R^{\times} \to R^{\times} \otimes \Qbb$ is precisely rationalization, whose kernel is $\mu(R)$. This gives a short exact sequence
\[0 \to K_2(R) \otimes \Qbb \to K_2(\operatorname{Az}(R)) \to \mu(R) \to 0,\]
which splits since rational vector spaces are injective $\Zbb$-modules. The group $K_1(\operatorname{Az}(R))$ is the cokernel of rationalization, which is precisely $\Qbb/\Zbb \otimes U(R)$.\\

Let us now discuss how to get the case for $\Cbb$ in the table.

\begin{cor}\label{cor::case_for_complex}
    Over the field $\Cbb$, the groups $\mathbf{Q}_i(\Zbb)$ matches those in Table~\ref{table_q_z}.
\end{cor}

\begin{proof}
The case for $i = 0$ follows from the Brauer group $\operatorname{Br}(\Cbb) = 0$. The case for $i = 1$ follows from $\operatorname{PGL}(\Cbb) = \operatorname{PSL}(\Cbb)$, and $\operatorname{PSL}(\Cbb)$ is the commutator subgroup of $\operatorname{PGL}(\Cbb)$ obtained from a similar colimit for $\operatorname{PSL}_n(\Cbb)$'s. The case for $i = 2$ follows from (a) $\mu(\Cbb)$ can be identified with $\Qbb/\Zbb$ and (b) a general fact that $K_{2n}(\Cbb)$ is uniquely divisible for $n > 0$ (Theorem VI.1.6 of \cite{weibel2013k}). The case for $i \geq 3$ even follows from the fact before. The case for $i \geq 3$ odd follows directly from rationalization. 
\end{proof}

\subsection{Rationalization of $K_1$}\label{subsec::rationalization}
The reader may be curious as to how rationalization shows up. We give a self-contained proof on why $\pi_1(\operatorname{BGL}_{\otimes}(R)^+)$ is $K_1(R) \otimes \Qbb$ for $R$ an Euclidean domain. For now, we will stick with a general ring and only insert the condition that $R$ is Euclidean later. We observe that the commutator subgroup $E_{\otimes}(R)$ can be equivalently realized as the colimit
 \[E_{\otimes}(R) = \operatorname{colim}_{(\Nbb, |)} E_n(R)\]
 of $E_n(R) \subseteq \operatorname{GL}_n(R)$ in the colimit diagram for $\operatorname{GL}_{\otimes}(R)$. For the ease of notation, we can now define the following.
\begin{defn}
    We write $K_1^{\otimes}(R) = \pi_1(\operatorname{BGL}_{\otimes}(R)^+) = \operatorname{GL}_{\otimes}(R)/E_{\otimes}(R)$.
\end{defn}
 
 As in the case of algebraic K-theory, one may be tempted to construct a determinant map $\operatorname{det}: K_1^{\otimes}(R) \to R^{\times}$ by first constructing a determinant map $\operatorname{det}: \operatorname{GL}_{\otimes}(R) \to R^{\times}$. The issue, however, is that the natural map
\[\operatorname{det}: \operatorname{GL}_n(R) \to R^{\times}\]
does not commute with the maps in the colimits. Indeed, in general, 
\[\det(I_k \otimes A) = \det(A)^k \text{ as opposed to } \det(A).\]

Let us look at an example which motivates a general construction below.
\begin{exmp}
    Let $R = \Rbb$ be the field of real numbers. For each $n$, we may construct a group homomorphism
    \[f_n: \operatorname{GL}_n(\Rbb) \to \Rbb_{>0}, f_n(A) = |\det(A)|^{1/n},\]
    where $\Rbb_{>0}$ has a group structure under multiplication with $1$ being the identity. Observe that 
    \[f_n(A B) = |\det(AB)|^{1/n} = |\det(A)|^{1/n} |\det(B)|^{1/n} = f_n(A) f_n(B),\]
    \[f_{kn}(I_k \otimes A) = |\det(A)^k|^{1/nk} = |\det(A)|^{1/n} = f_n(A).\]
    The definition of colimit then admits a unique morphism from the universal property
    \[f: \operatorname{GL}_{\times}(\Rbb) \to \Rbb_{>0}.\]
    Clearly $f$ is surjective. Furthermore, for any $A \in E_n(R)$, $\det(A) = 1$ and hence $f_n(A) = 1$, so $E_n(R)$ sits in the kernel of $f$. There is then a well-defined map
    \[f': \operatorname{GL}_{\otimes}(\Rbb)/E_{\otimes}(\Rbb) \to \Rbb_{>0},\quad [A] \mapsto f(A).\]
\end{exmp}

Let us now generalize the set-up above. Observe that as an abelian group, $\Rbb^{\times}$ is isomorphic to $\Zbb/2 \times \Rbb_{>0}$, and $\Rbb_{>0}$ is exactly the rationalization of $\Rbb^{\times}$. The ability to exponentiate $1/n$ is exactly saying $\Rbb_{>0}$, as an abelian group under multiplication, is a $\Qbb$-vector space.

\begin{defn}
    For each $n$, we define a map
    \[f_n: \operatorname{GL}_n(R) \to R^{\times} \otimes_{\Zbb} \Qbb, f_n(A) = \det(A) \otimes \frac{1}{n}.\]
\end{defn}

Observe that
\[f_n(AB) = \det(AB) \otimes \frac{1}{n} = \det(A) \det(B) \otimes \frac{1}{n} = \det(A) \otimes \frac{1}{n} + \det(B) \otimes \frac{1}{n} = f_n(A) + f_n(B).\]
\[f_{kn}(A \otimes I_k) = \det(A)^k \otimes \frac{1}{kn} = \det(A) \otimes \frac{k}{kn} = f_n(A).\]
Thus, this gives a well-defined map $f: \operatorname{GL}_{\otimes}(R) \to R^{\times} \otimes_{\Zbb} \Qbb.$\\

Now observe that for any matrix $A$ such that $\det(A) = 1$, we have that $A \in \ker(f_n)$. Indeed, $f_n(A) = \det(A) \otimes \frac{1}{n} = 1 \otimes \frac{1}{n} = (1)^n \otimes \frac{1}{n} = 1 \otimes \frac{n}{n} = 1 \otimes 1 = (0 \cdot 1) \otimes 1 = 1 \otimes 0 \cdot 1 = 1 \otimes 0$. Here, the equality that $1 \otimes 1 = (0 \cdot 1) \otimes 1$ is done by remembering that as a $\Zbb$-module, the action of $0$ on $R^{\times}$ sends everything to the multiplicative identity. On the other hand, for any $A \in E_n(R)$, $\det(A) = 1$. Thus, we have a well-defined map $f': \operatorname{GL}_{\otimes}(R)/E_{\otimes}(R) \to R^{\times} \otimes_{\Zbb} \Qbb$.

\begin{prop}
    $\ker(f_n: \operatorname{GL}_n(R) \to R^{\times} \otimes_{\Zbb} \Qbb)$ is exactly $\{A \in \operatorname{GL}_n(R)\ |\ \det(A) \in \operatorname{Tor}(R^{\times})\}$.
\end{prop}

\begin{proof}
    Suppose $A \in \operatorname{GL}_n(A)$ has torsion determinant, and say $\det(A)^k = 1$ then observe that
    \[f_n(A) = f_{kn}(I_k \otimes A) = (\det(A)^k) \otimes \frac{1}{nk} = 1 \otimes \frac{1}{nk} = 1 \otimes 0,\]
    where the last equality follows from the sequence of reductions in the previous definition. Conversely, suppose $A \in \ker(f_n)$. Observe that $f_n$ is the composition of the following maps
    \[\operatorname{GL}_n(A) \xrightarrow{\det} R^{\times} \xrightarrow{r \mapsto r \otimes 1} R^{\times} \otimes_{\Zbb} \Qbb \xrightarrow{\cdot \frac{1}{n}} R^{\times} \otimes_{\Zbb} \Qbb.\]
    If $\det(A)$ is torsion, we are done. Otherwise, $\{\det(A)^n: n \in \Zbb\}$ forms an infinite cyclic subgroup of $R^{\times}$, and by definition of rationalization is not sent to zero under the map $r \mapsto r \otimes 1$. The map $\cdot \frac{1}{n}$ is a $\Qbb$-vector-space isomorphism. Thus, we conclude that $\det(A)$ not being torsion implies $A \notin \ker(f_n)$.
\end{proof}

\begin{prop}
    Write $F_n(R) \coloneqq \ker(f_n)$, then $\operatorname{ker}(f: \operatorname{GL}_{\otimes}(R) \to R^{\times} \otimes_{\Zbb} \Qbb) = \operatorname{colim}_{(\Nbb, |)} F_n(R)$.
\end{prop}

\begin{proof}
    $\ker(f)$ is the equalizer of two maps $f, 0: \operatorname{GL}_{\otimes}(R) \to R^{\times} \otimes_{\Zbb} \Qbb$, where $0$ sends everything to the zero element. This is a finite limit. In the category of sets, filtered colimit commutes with finite limits, so $\operatorname{colim}_{(\Nbb, |)} F_n(R)$ is the set-theoretic equalizer of $f, 0$. On the other hand, the set-theoretic kernel is the group-theoretic kernel. This concludes the proof.
\end{proof}

\begin{prop}
    The map $f: \operatorname{GL}_{\otimes}(R) \to R^{\times} \otimes_{\Zbb} \Qbb$ is surjective. This descends down to a split surjection:
    \[f': \operatorname{GL}_{\otimes}(R)/E_{\otimes}(R) \to R^{\times} \otimes_{\Zbb} \Qbb.\]
\end{prop}

\begin{proof}
    For any $r \otimes \frac{p}{q} \in R^{\times} \otimes_{\Zbb} \Qbb$ with $r \in R^{\times}$ and $p/q \in \Qbb$, consider the $q \times q$ diagonal matrix $A$ whose first entry is $r^p$ and the rest are $1$'s. Then
    \[f([A]) = f_q(A) = \det(A) \otimes \frac{1}{q} = r^p \otimes \frac{1}{q} = r \otimes \frac{p}{q}.\]
    Evidently, since $E_{\times}(R)$ is contained in the kernel of $f$, $f'$ is still surjective.\\

    We define $g: R^{\times} \otimes_{\Zbb} \Qbb \to \operatorname{GL}_{\otimes}(R)$ by sending $r \otimes \frac{p}{q}$ to the matrix $A$ mentioned above. Now $g$ is defined on the elementary tensors. We then uniquely extend it to the mixed tensors. Note that if $g$ is well-defined, $f'$ is then a split surjection.\\
    
    We check that this is a well-defined map. Indeed, suppose $r \otimes \frac{p}{q} = s \otimes \frac{p'}{q'}$, we seek to show that $g(r \otimes \frac{p}{q}) = g(s \otimes \frac{p'}{q'})$. Indeed, observe that
    \[r \otimes \frac{p}{q} = s \otimes \frac{p'}{q'} \implies r^p \otimes \frac{1}{q} = s^{p'} \otimes \frac{1}{q'} \implies r^{pq'} \otimes 1 = s^{p'q} \otimes 1\]
    where the last implication is given by multiplying $qq'$ to both sides. Thus, we have that $r^{pq'} = s^{p'q}$ up to torsion, so there exists some $u \in \operatorname{Tor}(R)$ such that $r^{pq'} = u s^{p'q}$ with $u^k = 1$.\\
    
    Write $A = g(r \otimes \frac{p}{q})$ and $B = g(s \otimes \frac{p'}{q'})$. $A$ is a $q \times q$-matrix and $B$ is a $q' \times q'$-matrix.\\

    Now $[A] = [I_{q'} \otimes A]$ and $[B] = [I_{q} \otimes B]$. We have that $\det(I_{q'} \otimes A) = r^{pq'}$ and $\det(I_{q} \otimes B) = s^{p' q}$. Tensor both $I_{q'} \otimes A$ and $I_{q} \otimes B$ by $I_k$ on the left (and call the resulting matrices $A_1, B_1$). Recall that, modulo elementary matrices, we have the identification of block diagonal matrices
    \[\begin{pmatrix}
        C & 0\\
        0 & D
    \end{pmatrix} = \begin{pmatrix}
        C D & 0\\
        0 & I
    \end{pmatrix}\]
    when $C$ and $D$ have the same dimensions (this is a consequence of III.1.2.1 of \cite{weibel2013k}). But also, since the image of $g$ are literally diagonal matrices, we can without loss change $I_k \otimes I_{q'} \otimes A$ and $I_k \otimes I_{q} \otimes B$ both to diagonal matrices whose first entry is the determinant and the rest are $1$'s. From here, we see that it suffices to show that $s^{p' q k} = r^{pq'k}$. Now indeed, we have that
    \[s^{p' q k} = (u s^{p'q})^k = (r^{pq'})^k = r^{pq'k}.\]
    This shows that $[A] = [B]$ in $\operatorname{GL}_{\otimes}(R)/E_{\otimes}(R)$. Thus, we have a well-defined map $g: R^{\times} \otimes_{\Zbb} \Qbb \to \operatorname{GL}_{\otimes}(R)$ and clearly $f' \circ g$ is the identity.
\end{proof}
Putting the discussions above together, we have that:
\begin{theorem}
    There is a commutative diagram:
% https://q.uiver.app/#q=WzAsMyxbMCwwLCJcXGZyYWN7XFxvcGVyYXRvcm5hbWV7R0x9X3tcXHRpbWVzfShSKX17RV97XFx0aW1lc30oUil9Il0sWzAsMiwiXFxmcmFje1xcb3BlcmF0b3JuYW1le0dMfV97XFx0aW1lc30oUil9e0YoUil9Il0sWzIsMCwiUl57XFx0aW1lc30gXFxvdGltZXNfe1xcWmJifSBcXG1hdGhiYntRfSJdLFsxLDIsIlxcVGlsZGV7Zn0sXFx0ZXh0eyBpc29tb3JwaGlzbX0iLDJdLFswLDIsImYnIl0sWzAsMSwiXFxQaGkiLDIseyJzdHlsZSI6eyJoZWFkIjp7Im5hbWUiOiJlcGkifX19XV0=
\[\begin{tikzcd}
	{K_1^{\otimes}(R) = \frac{\operatorname{GL}_{\otimes}(R)}{E_{\otimes}(R)}} && {R^{\times} \otimes_{\Zbb} \mathbb{Q}} \\
	\\
	{\frac{\operatorname{GL}_{\otimes}(R)}{F(R)}}
	\arrow["{f'}", from=1-1, to=1-3]
	\arrow["\Phi"', two heads, from=1-1, to=3-1]
	\arrow["{\Tilde{f},\text{ isomorphism}}"', from=3-1, to=1-3]
\end{tikzcd}\]
Furthermore, $f'$ is a split surjection and
\[K_1^{\otimes}(R) = \ker(f') \oplus R^{\times} \otimes_{\Zbb} \mathbb{Q}.\]
\end{theorem}

Now we specialize to the case when $R$ is an Euclidean domain.
\begin{cor}
    Let $R$ be an Euclidean domain, then
    \[K^{\otimes}_1(R) = R^{\times} \otimes_{\Zbb} \Qbb.\]
\end{cor}

\begin{proof}
    We claim that $E_\otimes(R)$ and $F(R)$ are actually the same. Clearly, $E_\times(R)$ is contained in $F(R)$. Now recall when $R$ is an Euclidean domain, $E_n(R) = \operatorname{SL}_n(R)$. Suppose $[A] \in F(R)$, then it is represented by some matrix $A \in F_n(R)$, and $\det(A) \in R^{\times}$ has finite order say $k$. Now $I_k \otimes A \in \operatorname{SL}_n(R)$, and from the colimit definition we know that $[I_k \otimes A] = [A]$. This shows that $[A] \in E_{\otimes}(R)$.
\end{proof}

\begin{remark}\label{rmk:when_sk_vanish}
    One can also show that $\operatorname{ker}(f') = 0$ whenever $\operatorname{SK}_1(R) = 0$. This is expected given Theorem~\ref{thm::weibel_azumaya} and include many cases, such as when $R$ is a local ring. Rationalized algebraic K-theory is also known to line up with rational Chow groups in algebraic geometry, which may have connections to QCA in this work.
\end{remark}

\appendix 
\section{Coarse homology theory}\label{sec::coarse}
In this section, we give a brief introduction to coarse homology theory. Note that the maps in this section are typically not continuous. For more details, we refer the reader to Section 2 of \cite{block1997large}, Chapter~5 of \cite{roe2003lectures}, and Chapter 7 of \cite{nowak2023large}.
\begin{defn}
Let $(X, \rho)$ be a metric space. Then
\begin{enumerate}
\item $X$ is called \emph{uniformly discrete} if there exists a constant $C > 0$ such that for any two distinct points $x,y \in X$ we have
\[
\rho(x,y) \ge C.
\]

\item A uniformly discrete metric space $X$ is called \emph{locally finite} if for every $x \in X$ and every $r \ge 0$ we have
\[
\# B(x,r) < \infty.
\]

\item A locally finite metric space $X$ is said to have \emph{bounded geometry} if for every $r \in \mathbb{R}$ there exists a number $N(r)$ such that for every point $x \in X$ we have
\[
\# B(x,r) \le N(r).
\]
\end{enumerate}
\end{defn}

\begin{defn}
Let $X$ and $Y$ be metric spaces. A map $f : X \to Y$ is \emph{coarse} if the following two conditions are satisfied:
\begin{enumerate}
\item there exists a function $\eta_{+} : [0,\infty) \to [0,\infty)$ such that
\[
\rho_Y\bigl(f(x), f(y)\bigr) \le \eta_{+}\bigl(\rho_X(x,y)\bigr)
\]
for all $x,y \in X$, and
\item $f$ is metrically proper. That is, for every bounded subset $B \subseteq Y$, the
preimage $f^{-1}(B)$ is a bounded subset of $X$.
\end{enumerate}
\end{defn}
\begin{defn}
Let $(X,\rho_X)$ and $(Y,\rho_Y)$ be metric spaces.
Two maps $f,g \colon X \to Y$ are said to be \emph{close} if
\[
\sup_{x \in X} \rho_Y\bigl(f(x),g(x)\bigr) < \infty.
\]
\end{defn}
\begin{defn}
Let $(X,\rho_X)$ and $(Y,\rho_Y)$ be metric spaces.
A map $f \colon X \to Y$ is called a \emph{coarse equivalence} if there exists
a coarse map $g \colon Y \to X$ such that
\[
g \circ f \text{ is close to } \mathrm{id}_X
\quad\text{and}\quad
f \circ g \text{ is close to } \mathrm{id}_Y.
\]
\end{defn}

\begin{exmp}
    The inclusion $\ZZ\rightarrow \RR$ is a coarse equivalence. The map in the reverse direction could be the floor function. 
\end{exmp}
Coarse homology theories are intended to capture large-scale geometric
information and therefore should be invariant under coarse equivalence.
Since any metric space of interest for us is coarsely equivalent
to a uniformly discrete, locally finite metric space of bounded geometry,
it suffices to define coarse homology for spaces in this class.
For a general metric space $X$, its coarse homology is defined to be the
coarse homology of any space of this class
coarsely equivalent to $X$.

\begin{defn}
Let $(X,\rho)$ be a uniformly discrete metric space of bounded geometry and
let $A$ be an abelian group.
The \emph{coarse $n$-chain group} $CC_n(X,A)$ is defined to be the $A$-module
of formal linear combinations
\[
c = \sum_{\bar{x}} c_{\bar{x}}\, \bar{x},
\]
where $\bar{x} = [x_0,\dots,x_n] \in X^{n+1}$ and $c_{\bar{x}} \in A$, such that
there exists a constant $P_c > 0$ with
\[
c_{\bar{x}} = 0 \quad \text{whenever} \quad
\max_{i,j} \rho(x_i,x_j) \ge P_c.
\]
The number $P_c$ is called the \emph{propagation} of the chain $c$.
\end{defn}
\begin{remark}
     The above definition does agree with our definition earlier in Definition~\ref{def::coarse_diagonal}. The existence of propagation $P_c$ implies that the sum in $c$ is necessarily locally finite. That is, each $x\in X$ appears in finitely many terms in any chain. Additionally, the distance between $\bar{x} = [x_0,\dots,x_n]$ and the diagonal is bounded above by its distance to $[x_0,\dots,x_0]$. The latter is smaller than or equal to $l:=nP_c$.
\end{remark}

As in the classical simplicial setting, we define a boundary operator
\[
\partial_n \colon CC_n(X,A) \to CC_{n-1}(X,A)
\]
on generators by
\[
\partial_n[x_0,\dots,x_n]
= \sum_{i=0}^n (-1)^i [x_0,\dots,\hat{x}_i,\dots,x_n],
\]
where $\hat{x}_i$ denotes omission of the $i$th entry, and extend
$\partial_n$ to all of $CC_n(X,R)$ by linearity.
The propagation condition ensures that $\partial_n$ is well defined.

\begin{prop}
For every $n \ge 0$, one has $\partial_n \circ \partial_{n+1} = 0$.
\end{prop}

As a consequence, the collection $\{CC_n(X,R),\partial_n\}$ forms a chain
complex, called the \emph{coarse chain complex} of $X$.
Its homology groups are defined as follows.

\begin{defn}
Let $(X,\rho)$ be a uniformly discrete metric space of bounded geometry.
The \emph{coarse homology groups} of $X$ with coefficients in $A$ are
\[
CH_n(X,R)
:= \frac{\ker\bigl(\partial_n \colon CC_n(X,A) \to CC_{n-1}(X,A)\bigr)
   }{\operatorname{im}\bigl(\partial_{n+1} \colon CC_{n+1}(X,A) \to CC_n(X,A)\bigr)}.
\]
\end{defn}

Standard arguments from homological algebra show that coarse maps induce
homomorphisms on coarse homology.
Indeed, if $f \colon X \to Y$ is a coarse map, then $f$ induces a chain map
\[
f_\# \colon CC_n(X,A) \to CC_n(Y,A),
\qquad
f_\#[x_0,\dots,x_n] := [f(x_0),\dots,f(x_n)],
\]
which is well defined by the coarse conditions on $f$.

\begin{prop}
Any coarse map $f \colon X \to Y$ induces homomorphisms
\[
f_* \colon CH_n(X,A) \to CH_n(Y,A)
\]
for all $n \ge 0$, satisfying the following properties:
\begin{enumerate}
\item If $g \colon Y \to Z$ is coarse, then $(g \circ f)_* = g_* \circ f_*$.
\item If $f = \mathrm{id}_X$, then $f_*$ is the identity on $CH_n(X,A)$.
\item If $f,g \colon X \to Y$ are close, then $f_* = g_*$.
\end{enumerate}
\end{prop}

Applying the last property to a coarse equivalence and its coarse inverse,
one immediately obtains coarse invariance of coarse homology.

\begin{theorem}
Let $X$ and $Y$ be uniformly discrete metric spaces of bounded geometry.
If $X$ and $Y$ are coarsely equivalent, then
\[
CH_n(X,A) \cong CH_n(Y,A)
\quad \text{for all } n \ge 0.
\]
\end{theorem}
\begin{exmp}[Coarse homology of \(\mathbb R^n\)~\cite{roe1993coarse}]
\label{thm:coarse-Rn}

\[
CH_q(\mathbb{R}^n,A)\cong
\begin{cases}
0, & q\neq n,\\[4pt]
A, & q=n.
\end{cases}
\]
\end{exmp}
\section{Group Completion}\label{sec::group_complete}

\begin{defn}
Let $M$ be a commutative monoid. The \textit{(Grothendieck) group completion} of $M$ is an abelian group $M^{gp}$ with a monoid homomorphism $i: M \to M^{gp}$ characterized by the following universal property. For any monoid homomorphism $f: M \to A$ from $M$ to an abelian group $A$, there is a unique factorization:
\[\begin{tikzcd}[ampersand replacement=\&]
	M \& A \\
	{M^{gp}}
	\arrow["f", from=1-1, to=1-2]
	\arrow["i"', from=1-1, to=2-1]
	\arrow["{\exists ! g}"', dashed, from=2-1, to=1-2]
\end{tikzcd}.\]
More explicitly, $M^{gp}$ may be constructed as the quotient of the free abelian group on elements in $M$ with certain relations as
\[M^{gp} = \Zbb[M]/\langle [m] + [n] = [m + n], \forall m, n \in \Zbb \rangle.\]
The group completion operation is functorial with respect to homomorphisms of commutative monoids.
\end{defn}

\begin{exmp}
    Let $(\mathcal{C}, \otimes)$ be a (small) symmetric monoidal category. We can define a monoid $M$ associated to $\mathcal{C}$ whose elements are given by isomorphism classes of objects in $\mathcal{C}$, and the monoid operation given by $\otimes$. The zeroth K-theory of $\mathcal{C}$, denoted $K_0(\mathcal{C})$, is the group completion of $M$.
\end{exmp}

\begin{defn}
    Let $M$ be a commutative monoid and $R \subset M \times M$ be an equivalence relation such that $M/R$ is a monoid.
\end{defn}

Note that this makes $R$ a submonoid of $M \times M$ under coordinate-wise multiplication. Now, $M/R$ is the coequalizer of the diagram
\[\begin{tikzcd}
	R & M
	\arrow["{\pi_2}"', shift right=2, from=1-1, to=1-2]
	\arrow["{\pi_1}", shift left=2, from=1-1, to=1-2]
\end{tikzcd}.\]
by the two projections in the category of commutative monoids. Since group completion is left-adjoint, it preserves colimits. Thus, we have the following coequalizer diagram in the category of abelian groups
\[\begin{tikzcd}
	R^{gp} & M^{gp}
	\arrow["{\pi_2'}"', shift right=2, from=1-1, to=1-2]
	\arrow["{\pi_1'}", shift left=2, from=1-1, to=1-2]
\end{tikzcd}.\]

Here we prove a technical lemma (which is Lemma~\ref{lem::gp_quotient_commute} in the main text) that will be used in the proof of Theorem~\ref{thm::pi0_is_coarse}.
\begin{lemma}
Let $\iota: M \to M^{gp}$ denote the group completion map, the coequalizer above is $M^{gp}/I$ where $I$ is the subgroup generated by $\iota(a) - \iota(b)$ for all $(a, b) \in R$. 
\end{lemma}

\begin{proof}
By the coequalizer diagram above, we know that the coequalizer is $M^{gp}/J$ where 
\[J = \langle \pi_1'(x) - \pi_2'(x)\ |\ x \in R^{gp} \rangle.\]
Observe there is an explicit way to construct $\pi_1'$ (resp. $\pi_2'$) as follows. The composition $R \xrightarrow{\pi_1} M \xrightarrow{\iota} M^{gp}$ is a morphism into abelian groups, so the universal property induces a map $\pi_1': R^{gp} \to M^{gp}$ such that $\pi_1' \circ \iota_{R} = \iota \circ \pi_1$ (here $\iota_{R}: R \to R^{gp}$ is the group completion map).\\

Now suppose we have $\iota(a) - \iota(b) \in I$ with $(a, b) \in R$. We can write $a = \pi_1(a, b)$ and $b = \pi_2(a, b)$ to see that
\[\iota(a) - \iota(b) = \iota \circ \pi_1(a, b) - \iota \circ \pi_2(a, b) = \pi_1'(\iota_R(a, b)) - \pi_2'(\iota_R(a, b)) \in J.\]
Thus we have that $I \subseteq J$.\\

Conversely, suppose we have $\pi_1'(x) - \pi_2'(x) \in J$ for $x \in R^{gp}$. For any $x \in R^{gp}$, we can write $x = \iota_R(x_1) - \iota_R(x_2)$ for $x_1, x_2 \in R$. It follows that
\begin{align*}
    \pi_1'(x) - \pi_2'(x) &= \pi_1'(\iota_R(x_1)) - \pi_1'(\iota_R(x_2)) - (\pi_2'(\iota_R(x_1)) - \pi_2'(\iota_R(x_2)))\\
    &= \iota \circ \pi_1(x_1) - \iota \circ \pi_1(x_2) - \iota \circ \pi_2(x_1) + \iota \circ \pi_2(x_2)\\
    &= (\iota \circ \pi_1(x_1)+ \iota \circ \pi_2(x_2)) - (\iota \circ \pi_1(x_2) + \iota \circ \pi_2(x_1))\\
    &= \iota (\pi_1(x_1)+\pi_2(x_2)) - \iota (\pi_1(x_2)+\pi_2(x_1))
\end{align*}
Write $x_1 = (a, b)$ and $x_2 = (c, d)$. Showing that the term above is in $I$ amounts to showing that $(a, b), (c, d) \in R$ implies $(a + d, b + c) \in R$, which is true since $R$ is an equivalence relation and a submonoid. Thus, we have that $J \subset I$.
\end{proof}

The group completion described above can be thought of as happening on the level of $\pi_0$. There is also a topological version of group completion, which appears as Definition~\ref{def::k_theory_symmetric} in the main paper for the case of $B\mathcal{C}$. The same definition of group completion can be applied to any homotopy commutative, homotopy associative $H$-space. Under a suitable hypothesis, the group completion is an infinite loop space. A comprehensive discussion of this more generally can be found in \cite{adams}.

\printbibliography

\end{document}